\documentclass[reqno]{amsart}
\usepackage{amsfonts}
\usepackage{mathrsfs}
\usepackage{color}
\usepackage{cite}
\usepackage{amscd}
\usepackage{latexsym}
\usepackage{amsfonts}
\usepackage{graphicx}
\usepackage{CJK,indentfirst,amsmath,amsfonts,amssymb,amsthm,cite,cases, subeqnarray,setspace}

\allowdisplaybreaks

\begin{document}
\title[Quantum Euler-Poisson equation]
{The QKP limit of the quantum \\ Euler-Poisson equation}
\author{Huimin Liu and Xueke Pu}

\address{Huimin Liu \newline
Department of Mathematics, Chongqing University, Chongqing 401331, P.R.China} \email{ hmliucqu@163.com}

\address{Xueke Pu \newline
Department of Mathematics, Chongqing University, Chongqing 401331, P.R.China} \email{ xuekepu@cqu.edu.cn}

\thanks{This work is supported in part by NSFC (11471057).}

\begin{abstract}
In this paper, we consider the derivation of the Kadomtsev-Petviashvili (KP) equation for cold ion-acoustic wave in the long wavelength limit of the two-dimensional quantum Euler-Poisson system, under different scalings for varying directions in the Gardner-Morikawa transform. It is shown that the types of the KP equation depend on the scaled quantum parameter $H>0$. The QKP-I is derived for $H>2$, QKP-II for $0<H<2$ and the dispersive-less KP (dKP) equation for the critical case $H=2$. The rigorous proof for these limits is given in the well-prepared initial data case, and the norm that is chosen to close the proof is anisotropic in the two directions, in accordance with the anisotropic structure of the KP equation as well as the Gardner-Morikawa transform. The results can be generalized in several directions.
\end{abstract}

\maketitle \numberwithin{equation}{section}
\newtheorem{proposition}{Proposition}[section]
\newtheorem{theorem}{Theorem}[section]
\newtheorem{lemma}[theorem]{Lemma}
\newtheorem{remark}[theorem]{Remark}
\newtheorem{hypothesis}[theorem]{Hypothesis}
\newtheorem{definition}{Definition}[section]
\newtheorem{corollary}{Corollary}[section]
\newtheorem{assumption}{Assumption}[section]
\section{Introduction}
\setcounter{section}{1}\setcounter{equation}{0}

The Kadomtsev-Petviashvili (KP) equation is a two-dimensional extension of the KdV equation derived in \cite{KP}, as ``universal" models for the propagation of weakly nonlinear dispersive long waves that are essentially one-dimensional with weak transverse effects, when studying the stability of the solitary waves of the KdV equation. In addition to being an important dispersive models, both the KdV and the KP equations approximately describe the evolution of long waves in many physical settings, such as shallow-water waves with weakly non-linear restoring forces, long internal waves in a density-stratified ocean, ion acoustic waves in a plasma, acoustic waves on a crystal lattice and nonlinear matter-wave pulses in Bose-Einstein condensates (BEC). The following is the classical form of the Kadomtsev-Petviashvili equation (KP)
\begin{subequations}\label{equation1}
\begin{numcases}{}
u_{t}+uu_{x}+\mu u_{xxx}+\lambda v_{y}=0,\\
v_{x}=u_{y},
\end{numcases}
\end{subequations}
where $u=u(x,y,t)$, $(x,y)\in R^{2}$, $t\geq0$ in two-dimensional space. The constant $\lambda$ measures the transverse dispersion effects and are normalized to $\pm1$. When $\mu>0$, \eqref{equation1} is called the KP-I equation for $\lambda=-1$ and KP-II for $\lambda=+1$. When $\mu<0$, a simple transform shows that it corresponds to KP-I when $\lambda=+1$ and KP-II when $\lambda=-1$. 
When $\mu=0$ and $\lambda\neq0$, the equation \eqref{equation1} degenerates to the dispersive-less KP equation (dKP) which is integrable \cite{FM}. 
For $\mu=\lambda=0$, the equation \eqref{equation1} degenerates to the Burgers equation which exhibits singularities in a finite time. Like the KdV equation, the KP-I as well as KP-II equation \eqref{equation1} are completely integrable by using the inverse scattering transform \cite{D}. In the KdV equation waves are strictly one-dimensional, while in the KP equation this restriction is relaxed. Still, both in the KdV and the KP equation, waves have to travel in the positive $x$-direction. To be physically meaningful, the wave propagation direction has to be not-too-far from the $x$-direction, i.e. with only slow variations of solutions in the $y$-direction. Because of the asymmetry in the $x$- and $y$-directions, the waves described by the KP equation behave differently in the direction of propagation ($x$-direction) and transverse ($y$-direction), and oscillations in the $y$-direction tend to be smoother or, to be of small deviation in other words. The KP equation can be used to model waves of long wavelength with weakly nonlinear restoring forces and frequency dispersion and can be justified from various physics contexts.

There is a lot of work concerning the rigorous or formal justification of the KP limit. For clarify, we list only a few. First, a rigorous comparison between analytic solutions of 3D water wave problem and those of KP but on a time interval not allowing to observe the KP dynamic was given in \cite{K2}. Gallay and Schneider \cite{GS} obtained rigorously the dynamic of the KP-II equation to that of a Boussinesq equation. Youssef and Lannes proved \cite{BL} rigorously that a solution of a general class of quasilinear hyperbolic system (but not the 3D water wave problem) can be approximated by two waves moving in two opposite directions and satisfying a coupled or uncoupled system of KP-II equations at different orders. Moreover Lannes showed in \cite{L} the consistency of the KP-II approximation from a Boussinesq system. Pu \cite{Pu} derived the 2D KP-II equation rigorously from the dynamics of ions in a hot plasma, while leaves the cold plasma case open. Chiron and Rousset \cite{CR} proved rigorously the convergence to the Korteweg-de Vries (KdV) equation in 1D and to the KP-I equation in higher dimensions for the nonlinear Schr\"odinger equation with nonzero limit at infinity by a compactness argument. Then Chiron \cite{Chi} derived rigorously in some sense a gKdV or gKP-I equation involving cubic nonlinearity for either suitable nonlinearities for nonlinear Schr\"odinger equation either a Landau-Lifshitz type equation. On the formal level, there is much work recently. For example, the KP-I asymptotic dynamics for the Gross-Pitaevskii equation in three dimension is derived in \cite{Ber}. The 2D KP-II equation can be derived from dusty plasma with variable dust charge or ion acoustic waves, and the modified KP equation can be derived in an inhomogeneous plasma with finite temperature drifting ions \cite{MSD}.


In this paper, we aim to justify rigorously the quantum Kadomtsev-Petviashvili (QKP) equation \eqref{e8} from the quantum Euler-Poisson (QEP) system \eqref{e1}, which is an important ion acoustic wave model. For simplicity, `QKP' will refer either to QKP-I or to QKP-II in what follows, depending on the scaled quantum parameter $H>0$. Such a QEP equation cannot be categorized mathematically into the equations mentioned above from which rigorous KP justification was made, due to the different structure of the QEP equation. This makes the present paper interesting. The quantum Euler-Poisson system comes into play from the classic models mainly due to the presence of the Bohm potential, whose effect is embodied with a term containing the Planck's constant $\hbar$ indicating the quantum effect. 
Haas et al. \cite{H,HGG} used the quantum hydrodynamics model (QHD) to study quantum ion acoustic waves in the weakly nonlinear theory and obtained a deformed Korteweg-de Vries equation which involves the parameter $H$, proportional to the Planck's constant $\hbar$. They observed several characteristic features of pure quantum origin for the linear, weakly nonlinear and fully nonlinear waves. Such an approximation by the KdV equation was justified recently \cite{LP}. As a first step towards a justification of the QKP equation as an envelope equation, we consider in this paper the following 2D quantum Euler-Poisson equations with two species quantum plasmas:
\begin{subequations}\label{equ1}
\begin{numcases}{}
\partial_{t}n_{i}+\nabla\cdot(n_{i}\mathbf{u_{i}})=0, \label{equ1-1}\\
\partial_{t}\mathbf{u_{i}}+\mathbf{u_{i}}\cdot\nabla \mathbf{u_{i}}=-\nabla\phi,\label{equ1-2}\\
\Delta\phi=n_{e}-n_{i},\label{equ1-3}
\end{numcases}
\end{subequations}
where $n_{e,i}$ are the electronic and ionic number densities, $\mathbf{u_{i}}=(u_{i_{1}},u_{i_{2}})$ the ionic velocities, $\phi$ the scalar potential at time $t\geq0$ and position $x=(x_{1}, x_{2}) \in R^{2}$.
Particularly, the relation between of the electrostatic potential and the electron density satisfies
\begin{equation}
\begin{split}\label{tum}
\phi=-\frac{1}{2}+\frac{1}{2}n_{e}^{2} -\frac{H^{2}\Delta\sqrt{n_{e}}}{2\sqrt{n_{e}}}.
\end{split}
\end{equation}
${H^{2}}\Delta(\sqrt{n_{e}})/{2\sqrt{n_{e}}}$ is the so-called quantum Bohm potential, $H={\hbar\omega_{p_{e}}}/{2\kappa_{B}T_{F_{e}}}>0$ is the nondimensional quantum parameter, $\hbar$ is the Planck constant divided by $2\pi$, $\kappa_{B}$ is the Boltzmann's constant, $T_{F_{e}}$ is the Fermi temperature and $\omega_{p_{e}}=({4\pi n_{0}e^{2}}/{m_{e}})^{1/2}$, $n_{0}$ is the equilibrium density for both electrons and ions, $-e$ is the electron charge 
and $m_{e,i}$ the electron and ionic mass. $\mathbf{u_{i}}$ is the ion fluid speed normalized to the ion acoustic velocity $C_{s}=(2\kappa_{B}T_{F_{e}}/m_{i})^{1/2}$. The time and space variables are in units of the ion plasma period $\omega_{p_{i}}^{-1}=({m_{i}}/{4\pi n_{0}e^{2}})^{1/2}$ and the Debye radius $\lambda_{D}=({2\kappa_{B}T_{F_{e}}}/{4\pi n_{0}e^{2}})^{1/2}$ respectively. We assume that the electrons obey the equation of state in two-dimension (Manfredi and Haas \cite{MH})
\begin{equation*}
\begin{split}
p_{e}=\frac{m_{e}v_{F_{e}}^{2}}{4n_{0}}n_{e}^{3},
\end{split}
\end{equation*}
where the electron Fermi velocity is $v_{F_{e}}$ connected to the Fermi temperature $T_{F_{e}}$ by $m_{e}v_{F_{e}}^{2}/2=\kappa_{B}T_{F_{e}}$.
The quantum parameter $H$ is a measure of quantum diffraction effects and only modifies the dispersive coefficient. Physically, $H$ is the ratio between the electron plasmon energy and the electron Fermi energy. 

This model \eqref{equ1}-\eqref{tum} is the basic model to be studied in the following, which will lead to the QKP equation \eqref{e8} under the Gardner-Morikawa transform \eqref{trans}. The formal derivation is given in Section 2. The main interest in this paper is to make such a formal derivation rigorous, which is presented in Section 3. The implication of the justification is at least twofold. It not only makes interesting all the results on QKP equation up to date, but also it states that the solution of the QEP can exist on a very long time interval $[0,\varepsilon^{-3/2}\tau]$, where $\tau$ is the time scale of the rigorously QKP equation observed and $\varepsilon$ is the scale (under Sobolev norm) of the initial data of the QEP equation. It also states that the approximation error of the QKP equation to the QEP is of order $O(\varepsilon^2)$. For details, see Theorem \ref{thm1} and the remarks that follow.

In this paragraph, we make some remarks on the existing work that is closely related to our present work. Indeed, in the past one or two decades, many efforts have been made to rigorously justify various equations, such as the nonlinear Schr\"odinger equation \cite{K,KSM}, the KdV equation \cite{CR,C,GP12,LP}, the KP equation \cite{CR,K2,Pu,L,BL,LS}, the Zakharov-Kuznetsov (ZK) equation \cite{LL,Pu} and very recently the Ostrovsky equation\cite{BM}. Whether the KP equation provides a good approximation to the 2D quantum Euler-Poisson system is not known. 
As said above, many significant results already exist. First, without quantum effects, Guo and Pu \cite{GP12} established rigorously the KdV limit for the ion Euler-Poisson system in 1D for both cold and hot plasma cases, where the electron density satisfies the classical Maxwell-Boltzmann law. Recently, Liu and Pu \cite{LP} obtained rigorously the QKdV limit for the one-dimensional QEP system for the cold as well as hot plasma, the electron equilibrium is given by a Fermi-Dirac distribution. As in the study of transversal stability of unidimensional solitons, the KP equation arises as a bidimensional generalization of the KdV, so that we have reason to believe that the KP equation provides a good approximation to the solution of the 2D quantum Euler-Poisson problem, but there are different singularities between $x$ and $y$ directions for the Gardner-Morikawa transform \eqref{trans} compared to KdV limit, which is one of the difficult aspects in this paper. Hence, the aim of this paper is to take a new step in this direction and to justify a system of QKP equations, likely to furnish a better approximation to the exact solution of the QEP system.

On the other hand, Pu \cite{Pu} derived rigorously the 2D KP-II equation from the Euler-Poisson equation for hot plasma and derived the 3D ZK equation for both the cold and hot plasma cases, in which the hot isothermal electron are described by the Boltzmann distribution. However, it leaves open the rigorous derivation of 2D KP-II equation from the important cold plasma case. We may need to note that the scalings between KP limit and ZK limit are different and the ZK limit scaling is isotropic and is much more like the KdV limit case. The main reason may lie in the facts that the 2D Euler-Poisson system in the cold plasma case is not Friedrich symmetrizable and the scaling is anisotropic in the two directions, which leads to difficulties for obtaining uniform estimates for the remainder $(N_i,N_e,\mathbf{U})$. 
In our present paper, we show that the QKP equation indeed gives a rigorous approximation to the solution of the 2D quantum Euler-Poisson system for a cold plasma with the Bohm potential. The essential difference compared to \cite{Pu} is that a new triple norm \eqref{|||} is defined for the remainder, which will lead to a closed estimates inequality. The result in this paper gives affirmatively the rigorous justification that leaves open in \cite{Pu}. This makes the present paper different and more interesting, while also makes the proof in Section 3 more tough.


Before we end the Introduction, we would like to point out several possible generalizations, whose formal or rigorous justifications will not be given below. Firstly, the ion momentum equation \eqref{equ1-2} does not contain ion pressure, which generally depends on ion density with the form $P_i(n_i)=T_i\ln n_i$ for $T_i\geq0$. The present paper corresponds to the cold ion case $T_i=0$, but the result in this paper can be generalized to general case $T_i>0$, and indeed, the proof will be slightly simpler since in this case, the system is Friedrich symmetrizable. The result in this paper can be also generalized to the general $\gamma$-law of the ion pressure, i.e., when $P_i(n_i)=T_in_{i}^{\gamma}$ for $\gamma\geq1$. 
Secondly, without quantum effects, the result in this paper gives rigorous KP-II justification from the Euler-Poisson equations for the ions in cold plasma, which leaves open in \cite{Pu}. Thirdly, in the Euler-Poisson system we take \eqref{tum} as the relation between the electrostatic potential and the electron density. But we can also obtain similar results for the case that the hot isothermal electron are described by the Boltzmann distribution as in \cite{Pu}. Finally, we can generalize the result to justify the 3D ZK equation from the 3D QEP equation. For this the Gardner-Morikawa transform \eqref{trans} should be changed into the following form, consisting with the isotropic property of the ZK equation
\begin{equation*}
x\rightarrow\epsilon^{1/2}(x-Vt),\  \ y\rightarrow\epsilon^{1/2}y, \ \ z\rightarrow\epsilon^{1/2} z, \ \ t\rightarrow\epsilon^{3/2}t.
\end{equation*}
All the generalizations can be made rigorous, but for clarity we only focus on \eqref{equ1} with \eqref{tum} and no more remarks on these generalizations will be made below.


This paper is organized as follows. In Section 2, we present the formal derivation of the QKP equation \eqref{e8} and state the main result in Theorem \ref{thm1}. In Section 3, we present uniform estimates for the remainders in \eqref{cut}. The main estimates are stated in Proposition \ref{P1} and \ref{P2}. Finally, we complete the proof in Section 4.

\section{Formal expansion and main results}

\subsection{Formal QKP expansion}\label{1.1}
In this subsection, we derive the QKP equation from the 2D Euler-Poisson equations \eqref{equ1}-\eqref{tum}. Consider
the following Gardner-Morikawa type of transformation in \eqref{equ1}-\eqref{tum}
\begin{equation}\label{trans}
x_{1}\rightarrow\epsilon^{1/2}(x_{1}-Vt),\  \ x_{2}\rightarrow\epsilon x_{2}, \ \ t\rightarrow\epsilon^{3/2}t,
\end{equation}
where $\varepsilon$ stands for the amplitude of the initial disturbance and is assumed to be small compared with unity and $V$
is the wave speed to be determined. Then we obtain the parameterized system
\begin{subequations}\label{equ2}
\begin{numcases}{}
\varepsilon\partial_{t}n_{i}-V\partial_{x_{1}}n_{i}
+\partial_{x_{1}}(n_{i}u_{i_{1}})+\varepsilon^{1/2}\partial_{x_{2}}(n_{i}u_{i_{2}})=0, \label{equ2-1}\\
\varepsilon\partial_{t}u_{i_{1}}-V\partial_{x_{1}}u_{i_{1}}
+u_{i_{1}}\partial_{x_{1}}u_{i_{1}}+\varepsilon^{1/2}u_{i_{2}}\partial_{x_{2}}u_{i_{1}}=-\partial_{x_{1}}\phi,\label{equ2-2}\\
\varepsilon\partial_{t}u_{i_{2}}-V\partial_{x_{1}}u_{i_{2}}
+u_{i_{1}}\partial_{x_{1}}u_{i_{2}}+\varepsilon^{1/2} u_{i_{2}}\partial_{x_{2}}u_{i_{2}}=-\varepsilon^{1/2}\partial_{x_{2}}\phi,\label{equ2-3}\\
\varepsilon\partial_{x_{1}}^{2}\phi+\varepsilon^{2}\partial_{x_{2}}^{2}\phi=n_{e}-n_{i}, \label{equ2-4}
\end{numcases}
\end{subequations}
and
\begin{equation}\label{equa1'}
\phi=-\frac{1}{2}+\frac{1}{2}n_{e}^{2}-\frac{H^{2}}{2\sqrt{n_{e}}}\left(\varepsilon\partial_{x_{1}}^{2}\sqrt{n_{e}}
+\varepsilon^{2}\partial_{x_{2}}^{2}\sqrt{n_{e}}\right).
\end{equation}
We consider the following formal expansion around the equilibrium solution $(n_{i},n_{e},u_{i_{1}},u_{i_{2}})=(1,1,0,0)$,
\begin{equation}\label{expan-formal}
\begin{cases}
n_{i}=1+\epsilon n_{i}^{(1)}+\epsilon^{2}n_{i}^{(2)}+\epsilon^{3}n_{i}^{(3)}+\epsilon^{4}n_{i}^{(4)}
+\epsilon^{5}n_{i}^{(5)}+\epsilon^{6}n_{i}^{(6)}+\cdots,\\
n_{e}=1+\epsilon n_{e}^{(1)}+\epsilon^{2}n_{e}^{(2)}+\epsilon^{3}n_{e}^{(3)}+\epsilon^{4}n_{e}^{(4)}
+\epsilon^{5}n_{e}^{(5)}+\epsilon^{6}n_{e}^{(6)}+\cdots,\\
u_{i_{1}}=\epsilon u_{i_{1}}^{(1)}+\epsilon^{2}u_{i_{1}}^{(2)}+\epsilon^{3}u_{i_{1}}^{(3)}+\epsilon^{4}u_{i_{1}}^{(4)}
+\epsilon^{5}u_{i_{1}}^{(5)}+\epsilon^{6}u_{i_{1}}^{(6)}+\cdots,\\
u_{i_{2}}=\epsilon^{3/2} u_{i_{2}}^{(1)}+\epsilon^{5/2}u_{i_{2}}^{(2)}+\epsilon^{7/2}u_{i_{2}}^{(3)}+\epsilon^{9/2}u_{i_{2}}^{(4)}
+\epsilon^{11/2}u_{i_{2}}^{(5)}+\epsilon^{13/2}u_{i_{2}}^{(6)}+\cdots.
\end{cases}
\end{equation}
Plugging \eqref{expan-formal} into \eqref{equ2}-\eqref{equa1'}, we get a power series of $\epsilon$, whose coefficients depend on
$(n_{i}^{(k)},n_{e}^{(k)},u_{i_{1}}^{(k)},u_{i_{2}}^{(k)})$ for $k=1,2,\cdots$.

\subsubsection{Derivation of the QKP equation for $n_{i}^{(1)}$}\label{1.11}

At the order $O(\epsilon)$, we obtain
\begin{subequations}\label{e1}
\begin{numcases}{(\mathcal S_0)}
-V\partial_{x_{1}}n_{i}^{(1)}+\partial_{x_{1}}u_{i_{1}}^{(1)}=0,\\
-V\partial_{x_{1}}u_{i_{1}}^{(1)}=-\partial_{x_{1}}n_{e}^{(1)},\\
0=n_{e}^{(1)}-n_{i}^{(1)}.
\end{numcases}
\end{subequations}
To get a nontrivial solution of $(n_{i}^{(1)},n_{e}^{(1)},u_{i_{1}}^{(1)})$, we let the determinant of the coefficient matrix of
\eqref{e1} vanish to obtain
\begin{equation}\label{e2}
V^{2}=1.
\end{equation}

At the order $O(\epsilon^\frac{3}{2})$, we obtain
\begin{equation}\label{e3}
-V\partial_{x_{1}}u_{i_{2}}^{(1)}=-\partial_{x_{2}}n_{e}^{(1)}.
\end{equation}

At the order $O(\epsilon^2)$, we obtain
\begin{subequations}\label{e4}
\begin{numcases}{(\mathcal S_1)}
\partial_{t}n_{i}^{(1)}-V\partial_{x_{1}}n_{i}^{(2)}+\partial_{x_{1}}u_{i_{1}}^{(2)}+\partial_{x_{1}}(n_{i}^{(1)}u_{i_{1}}^{(1)})
+\partial_{x_{2}}u_{i_{2}}^{(1)}=0, \label{e4-1}\\
\partial_{t}u_{i_{1}}^{(1)}-V\partial_{x_{1}}u_{i_{1}}^{(2)}+u_{i_{1}}^{(1)}\partial_{x_{1}}u_{i_{1}}^{(1)}\nonumber \\
 \ \ \ \ =-\partial_{x_{1}}n_{e}^{(2)}-n_{e}^{(1)}\partial_{x_{1}}n_{e}^{(1)}
+\frac{H^{2}}{4}\partial_{x_{1}}^{3}n_{e}^{(1)}\label{e4-2},\\
\partial_{x_{1}}^{2}n_{e}^{(1)}=n_{e}^{(2)}-n_{i}^{(2)}.\label{e4-3}
\end{numcases}
\end{subequations}
From \eqref{e1}, we may assume that
\begin{equation}\label{e5}
n_{e}^{(1)}=n_{i}^{(1)}, \ u_{i_{1}}^{(1)}=Vn_{i}^{(1)},
\end{equation}
which also make \eqref{e1} valid, thanks to \eqref{e2}. Then from \eqref{e3}, we have
\begin{equation}\label{e6}
\partial_{x_{1}}u_{i_{2}}^{(1)}=V\partial_{x_{2}}n_{i}^{(1)},
\end{equation}
thanks to \eqref{e5}. Therefore, to solve $n_{e}^{(1)},n_{i}^{(1)},u_{i_{1}}^{(1)},u_{i_{2}}^{(1)}$, we need only to solve $n_{i}^{(1)}$.

To find out the equation satisfied by $n_{i}^{(1)}$, we take $\partial_{x_{1}}$ of \eqref{e4-3}, multiply \eqref{e4-1} by $V$,
and then add them to \eqref{e4-2}. We obtain
\begin{equation}\label{e7}
\partial_{t}n_{i}^{(1)}+(\frac{3}{2}V+\frac{1}{2V})n_{i}^{(1)}\partial_{x_{1}}n_{i}^{(1)}
+\frac{1}{2V}(1-\frac{H^{2}}{4})\partial_{x_{1}}^{3}n_{i}^{(1)}+\frac{1}{2}\partial_{x_{2}}u_{i_{2}}^{(1)}=0.
\end{equation}
Differentiating this equation with respect to $x_{1}$, and using \eqref{e6}, we obtain
\begin{equation}\label{e8}
\partial_{x_{1}}\left(\partial_{t}n_{i}^{(1)}+(\frac{3}{2}V+\frac{1}{2V})n_{i}^{(1)}\partial_{x_{1}}n_{i}^{(1)}
+\frac{1}{2V}(1-\frac{H^{2}}{4})\partial_{x_{1}}^{3}n_{i}^{(1)}\right)+\frac{V}{2}\partial_{x_{2}}^{2}n_{i}^{(1)}=0.
\end{equation}
This is the quantum Kadomtsev-Petviashvili-I (QKP-I) equation for $H>2$, quantum Kadomtsev-Petviashvili-II (QKP-II) equation for $0<H<2$ and dispersive-less Kadomtsev-Petviashvili (dKP) equation for the critical case $H=2$, satisfied by the first order profile $n_{i}^{(1)}$.

We have the following well-posedness theorem for QKP-I and QKP-II, which  was shown by using PDE techniques \cite{BJ} and \cite{Mol}, respectively.
\begin{theorem}\label{thms}
The Cauchy problems for the QKP-II (QKP-I) eqution \eqref{e8} are globally (locally) well-posed in $H^{s}$ for $s\geq0$.
\end{theorem}
When $H=2$, \eqref{e8} degenerates to the dispersive-less Kadomtsev-Petviashvili (dKP) equation which was derived earlier than the KP equation by Lin, Reissner and Tsien \cite{LRT} and Khokhlov and Zabolotskaya \cite{Z} in three spatial dimensions. The local well-posedness of the Cauchy problem for dispersive-less KP equation(dKP) has been proved in certain Sobolev spaces in \cite{R}.
\begin{theorem}\label{thms0}
The Cauchy problem for the dKP eqution \eqref{e8} is locally well-posed for any initial data $n_{i0}^{(1)}$ in $H^{s}$ for $s>2$.
\end{theorem}
The system of \eqref{e5}, \eqref{e6} and \eqref{e8} is a closed system. Once $n_{i}^{(1)}$ is solved from \eqref{e8}, we have all the other profiles $n_{e}^{(1)},u_{i_{1}}^{(1)},u_{i_{2}}^{(1)}$ from \eqref{e5}, \eqref{e6}.
From \eqref{e4-1} and \eqref{e4-3}, we may assume
\begin{equation}\label{e9}
\begin{cases}
u_{i_{1}}^{(2)}=Vn_{i}^{(2)}+\underline{u_{i_{1}}}_{kp}^{(1)},\\
n_{e}^{(2)}=n_{i}^{(2)}+\underline{n_{e}}_{kp}^{(1)},
\end{cases}
\end{equation}
where $\underline{u_{i_{1}}}_{kp}^{(1)},\underline{n_{e}}_{kp}^{(1)}$ depend only on $n_{i}^{(1)}$.

At the order $O(\epsilon^\frac{5}{2})$, we obtain
\begin{equation}\label{e10}
\begin{split}
\partial_{t}u_{i_{2}}^{(1)}&-V\partial_{x_{1}}u_{i_{2}}^{(2)}+u_{i_{1}}^{(1)}\partial_{x_{1}}u_{i_{2}}^{(1)}\\
&=-(\partial_{x_{2}}n_{e}^{(1)}+n_{e}^{(1)}\partial_{x_{2}}n_{e}^{(1)}-\frac{H^{2}}{4}\partial_{x_{1}}^{2}\partial_{x_{2}}n_{e}^{(1)}).
\end{split}
\end{equation}
By using \eqref{e9} and rearranging, we have
\begin{equation}\label{e11}
\begin{split}
V\partial_{x_{1}}u_{i_{2}}^{(2)}=&\partial_{t}u_{i_{2}}^{(1)}+u_{i_{1}}^{(1)}\partial_{x_{1}}u_{i_{2}}^{(1)}
+\partial_{x_{2}}n_{i}^{(2)}+n_{i}^{(1)}\partial_{x_{2}}n_{i}^{(1)}\\
&-\frac{H^{2}}{4}\partial_{x_{1}}^{2}\partial_{x_{2}}n_{i}^{(1)}+\underline{n_{e}}_{kp}^{(1)}.
\end{split}
\end{equation}

\subsubsection{Derivation of the Linearized QKP equation for $n_{i}^{(k)}(k\geq2)$}\label{1.12}

From \eqref{e9} and \eqref{e11}, we see that to determine $(n_{i}^{(2)},n_{e}^{(2)},u_{i_{1}}^{(2)},u_{i_{2}}^{(2)})$,
we need only to determine $n_{i}^{(2)}$.

At the order $O(\epsilon^3)$, we obtain
\begin{subequations}\label{e12}
\begin{numcases}{}
\partial_{t}n_{i}^{(2)}-V\partial_{x_{1}}n_{i}^{(3)}+\partial_{x_{1}}(u_{i_{1}}^{(3)}+n_{i}^{(1)}u_{i_{1}}^{(2)}+n_{i}^{(2)}u_{i_{1}}^{(1)})
+\partial_{x_{2}}(u_{i_{2}}^{(2)}+n_{i}^{(1)}u_{i_{2}}^{(1)})=0,\label{e12-1}\\
\partial_{t}u_{i_{1}}^{(2)}-V\partial_{x_{1}}u_{i_{1}}^{(3)}+\partial_{x_{1}}(u_{i_{1}}^{(1)}u_{i1}^{(2)})+u_{i_{2}}^{(1)}\partial_{x_{2}}u_{i_{1}}^{(1)}
\nonumber \\ \ \ \ \
=-\partial_{x_{1}}n_{e}^{(3)}-\partial_{x_{1}}(n_{e}^{(1)}n_{e}^{(2)})\nonumber\\ \ \ \ \ \ \ \ \
-\frac{H^{2}}{4}(-\partial_{x_{1}}^{3}n_{e}^{(2)}-\partial_{x_{1}}\partial_{x_{2}}^{2}n_{e}^{(1)}
+n_{e}^{(1)}\partial_{x_{1}}^{3}n_{e}^{(1)}+2\partial_{x_{1}}n_{e}^{(1)}\partial_{x_{1}}^{2}n_{e}^{(1)}),\label{e12-2}\\
\partial_{x_{1}}^{2}n_{e}^{(2)}+(\partial_{x_{1}}n_{e}^{(1)})^{2}+n_{e}^{(1)}\partial_{x_{1}}^{2}n_{e}^{(1)}
+\partial_{x_{2}}^{2}n_{e}^{(1)}-\frac{H^{2}}{4}\partial_{x_{1}}^{4}n_{e}^{(1)}
=n_{e}^{(3)}-n_{i}^{(3)}.\label{e12-3}
\end{numcases}
\end{subequations}
We take $\partial_{x_{1}}$ of \eqref{e12-3}, multiply \eqref{e12-1} by $V$,
and then add them to \eqref{e12-2}, we obtain the linearized inhomogeneous QKP equation
\begin{equation}\label{e13}
\partial_{x_{1}}\left(\partial_{t}n_{i}^{(2)}+(\frac{3}{2}V+\frac{1}{2V})\partial_{x_{1}}(n_{i}^{(1)}n_{i}^{(2)})
+\frac{1}{2V}(1-\frac{H^{2}}{4})\partial_{x_{1}}^{3}n_{i}^{(2)}\right)+\frac{V}{2}\partial_{x_{2}}^{2}n_{i}^{(2)}=\underline{A}_{kp}^{(1)},
\end{equation}
where we have used \eqref{e9} and \eqref{e11}. Here $\underline{A}_{kp}^{(1)}$ depends only on $n_{i}^{(1)}$ and
comes from the inhomogeneous dependence of $u_{i_{1}}^{(2)}$ and $n_{e}^{(2)}$ on $n_{i}^{(2)}$ in \eqref{e9}.

At the order $O(\epsilon^\frac{7}{2})$, we obtain
\begin{equation}\label{e14}
\begin{split}
\partial_{t}u_{i_{2}}^{(2)}&-V\partial_{x_{1}}u_{i_{2}}^{(3)}+u_{i_{1}}^{(1)}\partial_{x_{1}}u_{i_{2}}^{(2)}+u_{i_{1}}^{(2)}\partial_{x_{1}}u_{i_{2}}^{(1)}
+u_{i_{2}}^{(1)}\partial_{x_{2}}u_{i_{2}}^{(1)}\\
=&-\partial_{x_{2}}n_{e}^{(3)}+\partial_{x_{2}}(n_{e}^{(1)}n_{e}^{(2)})+\frac{H^{2}}{4}(-\partial_{x_{1}}^{2}\partial_{x_{2}}n_{e}^{(2)}\\
&+n_{e}^{(1)}\partial_{x_{1}}^{2}\partial_{x_{2}}n_{e}^{(1)}+\partial_{x_{1}}n_{e}^{(1)}\partial_{x_{1}x_{2}}n_{e}^{(1)}
+\partial_{x_{2}}n_{e}^{(1)}\partial_{x_{1}}^{2}n_{e}^{(1)}).
\end{split}
\end{equation}
Inductively, we can derive all the profiles $(n_{i}^{(k)},n_{e}^{(k)},\mathbf{u_{i}^{(k)}})$ for $k\geq3$. Proceeding as above, we obtain the following linearized inhomogeneous QKP equation for $k\geq3$
\begin{equation}\label{e15}
\partial_{x_{1}}\left(\partial_{t}n_{i}^{(k)}+(\frac{3}{2}V+\frac{1}{2V})\partial_{x_{1}}(n_{i}^{(1)}n_{i}^{(k)})
+\frac{1}{2V}(1-\frac{H^{2}}{4})\partial_{x_{1}}^{3}n_{i}^{(k)}\right)+\frac{V}{2}\partial_{x_{2}}^{2}n_{i}^{(k)}=\underline{A}_{kp}^{(k-1)},
\end{equation}
where the inhomogeneous term $\underline{A}_{kp}^{(k-1)}$ depends only on $(n_{i}^{(j)},n_{e}^{(j)},\mathbf{u_{i}^{(j)}})$ for $1\leq j\leq k-1$.
For $n_{i}^{(k)}$, we have
\begin{theorem}\label{thmt}
The Cauchy problem for the linearized QKP-II (QKP-I/dKP) equation \eqref{e15} $(k\geq2)$ is globally (locally) well-posed in $H^{s}$ for $s>2$.
\end{theorem}

\subsection{Main result}\label{sect-rem} To show that $n_{i}^{(1)}$ converges to a solution of the QKP equation \eqref{e8} as $\epsilon\rightarrow0$, we must make the above procedure rigorous. Let $(n_{e},n_{i},\mathbf{u_{i}})$ be the solution of the scaled system \eqref{equ1} of the following expansion

\begin{equation}\label{cut}
\begin{cases}
n_{i}=1+\epsilon n_{i}^{(1)}+\epsilon^{2}n_{i}^{(2)}+\epsilon^{3}n_{i}^{(3)}+\epsilon^{4}n_{i}^{(4)}
+\epsilon^{5}n_{i}^{(5)}+\epsilon^{6}n_{i}^{(6)}+\epsilon^{5}N_{i},\\
n_{e}=1+\epsilon n_{e}^{(1)}+\epsilon^{2}n_{e}^{(2)}+\epsilon^{3}n_{e}^{(3)}+\epsilon^{4}n_{e}^{(4)}
+\epsilon^{5}n_{e}^{(5)}+\epsilon^{6}n_{e}^{(6)}+\epsilon^{5}N_{e},\\
u_{i_{1}}=\epsilon u_{i_{1}}^{(1)}+\epsilon^{2}u_{i_{1}}^{(2)}+\epsilon^{3}u_{i_{1}}^{(3)}+\epsilon^{4}u_{i_{1}}^{(4)}
+\epsilon^{5}u_{i_{1}}^{(5)}+\epsilon^{6}u_{i_{1}}^{(6)}+\epsilon^{5}U_{1},\\
u_{i_{2}}=\epsilon^{3/2} u_{i_{2}}^{(1)}+\epsilon^{5/2}u_{i_{2}}^{(2)}+\epsilon^{7/2}u_{i_{2}}^{(3)}+\epsilon^{9/2}u_{i_{2}}^{(4)}
+\epsilon^{11/2}u_{i_{2}}^{(5)}+\epsilon^{13/2}u_{i_{2}}^{(6)}+\epsilon^{5}U_{2},
\end{cases}
\end{equation}
where $(n_{i}^{(1)}$,$n_{e}^{(1)}$,$u_{i_{1}}^{(1)}$,$u_{i_{2}}^{(1)})$ satisfies \eqref{e5}, \eqref{e6} and \eqref{e8}, ($n_{i}^{(k)}$,$n_{e}^{(k)}$,$u_{i_{1}}^{(k)}$,$u_{i_{2}}^{(k)}$) satisfies \eqref{e9}, \eqref{e11} and \eqref{e13} for $2\leq k\leq6$, and $(N_{i},N_{e},\mathbf{U})$ is the remainder. To simplify the notation slightly, we set
\begin{equation*}
\begin{cases}
\tilde {n_i}= n_i^{(1)}+\epsilon n_i^{(2)}+\epsilon^{2}n_i^{(3)}+\epsilon^{3}n_i^{(4)}+\epsilon^{4}n_i^{(5)}+\epsilon^{5}n_i^{(6)},\\
\tilde{n_e}= n_e^{(1)}+\epsilon n_e^{(2)}+\epsilon^{2}n_e^{(3)}+\epsilon^{3}n_e^{(4)}+\epsilon^{4}n_e^{(5)}+\epsilon^{5}n_e^{(6)},\\
\tilde{u_{i1}}= u_{i1}^{(1)}+\epsilon u_{i1}^{(2)}+\epsilon^{2}u_{i1}^{(3)}+\epsilon^{3}u_{i1}^{(4)}+\epsilon^{4}u_{i1}^{(5)}+\epsilon^{5}u_{i1}^{(6)},\\
\tilde{u_{i2}}=\varepsilon^{1/2}u_{i2}^{(1)}+\epsilon^{3/2}u_{i2}^{(2)}+\epsilon^{5/2}u_{i2}^{(3)}+\epsilon^{7/2}u_{i2}^{(4)}
+\epsilon^{9/2}u_{i2}^{(5)}+\epsilon^{11/2}u_{i2}^{(6)}.
\end{cases}
\end{equation*}

After careful computations, we obtain the following remainder system for $(N_{i},N_{e},\mathbf{U})$,
\begin{subequations}\label{rem}
\begin{numcases}{}\displaystyle
\partial_{t}N_{i}-\frac{V-u_{i_{1}}}{\epsilon}\partial_{x_{1}}N_{i}+\frac{\epsilon^{1/2}}{\epsilon}u_{i_{2}}\partial_{x_{2}}N_{i}
+\frac{n_{i}}{\epsilon}\partial_{x_{1}}U_{1}+\frac{\epsilon^{1/2}}{\epsilon}n_{i}\partial_{x_{2}}U_{2}\nonumber\\ \ \ \ \ \ \
+\partial_{x_{1}}\tilde{u_{i_{1}}}N_{i}+\epsilon^{1/2}\partial_{x_{2}}\tilde{u_{i_{2}}}N_{i}
+\partial_{x_{1}}\tilde{n_{i}}U_{1}+\epsilon^{1/2}\partial_{x_{2}}\tilde{n_{i}}U_{2}+\epsilon \mathcal{R}_{1}=0,\label{rem-1}\\
\partial_{t}U_{1}-\frac{V-u_{i_{1}}}{\epsilon}\partial_{x_{1}}U_{1}+\frac{\epsilon^{1/2}}{\epsilon}u_{i_{2}}\partial_{x_{2}}U_{1}
+\partial_{x_{1}}\tilde{u_{i_{1}}}U_{1}+\epsilon^{1/2}\partial_{x_{2}}\tilde{u_{i_{1}}}U_{2}+\epsilon \mathcal{R}_{2}^{(1)}\nonumber\\ \ \ \ \ \ \
=-\frac{n_{e}}{\epsilon}\partial_{x_{1}}N_{e}-\partial_{x_{1}}\tilde{n_{e}}N_{e}
+\frac{H^{2}}{4}\left(\frac{\partial_{x_{1}}^{3}N_{e}+\epsilon\partial_{x_{1}}\partial_{x_{2}}^{2}N_{e}}{n_{e}}\right)\nonumber\\ \ \ \ \ \ \ \ \ \
+\frac{H^{2}}{4}\left(-\frac{A_{1}}{n_{e}^{2}}+\frac{A_{2}}{n_{e}^{3}}+\frac{\epsilon\mathcal{R}_{2}^{(2)}
+\epsilon\mathcal{R}_{2}^{(3)}}{n_{e}^{3}}\right),\label{rem-2}\\
\partial_{t}U_{2}-\frac{V-u_{i_{1}}}{\epsilon}\partial_{x_{1}}U_{2}+\frac{\epsilon^{1/2}}{\epsilon}u_{i_{2}}\partial_{x_{2}}U_{2}
+\partial_{x_{1}}\tilde{u_{i_{2}}}U_{1}+\epsilon^{1/2}\partial_{x_{2}}\tilde{u_{i_{2}}}U_{2}+\epsilon^{3/2}\mathcal{R}_{3}^{(1)}\nonumber\\ \ \ \ \ \ \
=-\frac{\epsilon^{1/2}n_{e}}{\epsilon}\partial_{x_{2}}N_{e}-\epsilon^{1/2}\partial_{x_{2}}\tilde{n_{e}}N_{e}
+\frac{\epsilon^{1/2}H^{2}}{4}\left(\frac{\partial_{x_{1}}^{2}\partial_{x_{2}}N_{e}+\epsilon\partial_{x_{2}}^{3}N_{e}}{n_{e}}\right)\nonumber\\ \ \ \ \ \ \ \ \ \ \
+\frac{\epsilon^{1/2}H^{2}}{4}\left(-\frac{B_{1}}{n_{e}^{2}}+\frac{B_{2}}{n_{e}^{3}}
+\frac{\epsilon\mathcal{R}_{3}^{(2)}+\epsilon\mathcal{R}_{3}^{(3)}}{n_{e}^{3}}\right),\label{rem-3}\\
\frac{n_{e}}{\epsilon}\partial_{x_{1}}^{2}N_{e}+n_{e}\partial_{x_{2}}^{2}N_{e}+
(\partial_{x_{1}}\tilde{n_{e}}\partial_{x_{1}}N_{e}+\epsilon\partial_{x_{2}}\tilde{n_{e}}\partial_{x_{2}}N_{e})
+\big(\epsilon^{4}(\partial_{x_{1}}N_{e})^{2}+\epsilon^{5}(\partial_{x_{2}}N_{e})^{2}\big)\nonumber\\ \ \ \ \ \ \
+(\partial_{x_{1}}^{2}\tilde{n_{e}}+\epsilon\partial_{x_{2}}^{2}\tilde{n_{e}})N_{e}+\mathcal{R}_{4}^{(1)}
-\frac{H^{2}}{4}\left(\frac{\partial_{x_{1}}^{4}N_{e}+2\epsilon\partial_{x_{1}}^{2}\partial_{x_{2}}^{2}N_{e}
+\epsilon^{2}\partial_{x_{2}}^{4}N_{e}}{n_{e}}\right)\nonumber\\ \ \ \ \ \ \
+\frac{H^{2}}{4}\left(\frac{C_{1}}{n_{e}^{2}}-\frac{C_{2}}{n_{e}^{3}}+\frac{C_{3}}{n_{e}^{4}}
+\frac{\mathcal{R}_{4}^{(2)}+\mathcal{R}_{4}^{(3)}}{n_{e}^{4}}\right)=\frac{N_{e}-N_{i}}{\epsilon^{2}}\label{rem-4},
\end{numcases}
\end{subequations}
where $\mathcal{R}_{1}$, $\mathcal{R}_{2}^{1}$, $\mathcal{R}_{2}^{2}$, $\mathcal{R}_{3}^{1}$, $\mathcal{R}_{3}^{2}$ and $\mathcal{R}_{4}^{1}$, $\mathcal{R}_{4}^{2}$ only depend on ($n_{i}^{(k)}$,$n_{e}^{(k)}$,$u_{i_{1}}^{(k)}$,$u_{i_{2}}^{(k)}$) for $1\leq k\leq6$, $\mathcal{R}_{2}^{3}$, $\mathcal{R}_{3}^{3}$ and $\mathcal{R}_{4}^{3}$ are smooth functions of $N_{e}$, and do not involve any derivatives of $N_{e}$. For clarity, we put the concrete expressions of $A_{i}, B_{i}(1\leq i\leq2)$ and $C_{j}(1\leq j\leq3)$ in Appendix and give the estimates of $\mathcal{R}_{i}^{j}$ in Lemma \ref{Lem1}.

We need to derive uniform in $\epsilon$ estimates for the remainder $(N_{e},N_{i},U_{1},U_{2})$, to make the above derivation rigorous. From Theorem \ref{thms}, we may assume that the known profiles $(\tilde{n_{i}},\tilde{n_{e}},\tilde{\mathbf{u_{i}}})$ are smooth enough
such that there exist some $C>0$ and some $s\geq5$,
\begin{equation}\label{equ3'}
\begin{split}
\sup_{[0,\tau_{\ast}]}\|(\tilde{n_{i}},\tilde{n_{e}},\tilde{\mathbf{u_{i}}})\|_{H^{s}}\leq C ,
\end{split}
\end{equation}
where $\tau_{\ast}$ is the existence time in Theorem \ref{thms}/\ref{thms0}. The basic plan is to estimate some uniform bound for $(N_{e},\mathbf{U})$ first and then recover the estimate for $N_{i}$ from the estimate of $N_{e}$ by the equation \eqref{rem}. We want to apply the Gronwall lemma to complete the proof. To apply the Gronwall inequality to complete the proof, we define the triple norm
\begin{equation}\label{|||}
\begin{split}
|\!|\!|(N_{e},N_{i},\mathbf{U})|\!|\!|_{\epsilon}^{2}: =&\sum_{0\leq\alpha+\beta\leq3}\epsilon^{\alpha+2\beta} \|\partial_{x_{1}}^{\alpha}\partial_{x_{2}}^{\beta}N_{i}\|_{L^{2}}^{2} +\sum_{0\leq\alpha+\beta\leq4}\epsilon^{\alpha+2\beta} \|\partial_{x_{1}}^{\alpha}\partial_{x_{2}}^{\beta}(U_{1},U_{2}) \|_{L^{2}}^{2}\\
&+\sum_{0\leq\alpha+\beta\leq7}\epsilon^{\alpha+2\beta} \|\partial_{x_{1}}^{\alpha}\partial_{x_{2}}^{\beta}N_{e}\|_{L^{2}}^{2}.
\end{split}
\end{equation}
We note that this norm is anisotropic in the sense that the powers of $\varepsilon$ in the two spatial directions are different, in accordance with the anisotropic structure of the QKP equation as well as the Gardner-Morikawa transform \eqref{trans}.

Our main result of this paper is the following
\begin{theorem}\label{thm1}
Let $s \geq5$ such that \eqref{equ3'} holds and $(n_{i}^{(j)},n_{e}^{(j)},\mathbf{u_{i}}^{(j)})\in H^{s}$ $(1\leq j\leq6)$ be a solution on the interval $[0,\tau_{\ast})$ constructed in Theorem \ref{thms}/\ref{thms0} and Theorem \ref{thmt} for the QKP/dKP equations with initial data $(n_{i0}^{(j)},n_{e0}^{(j)},\mathbf{u}_{\mathbf{i}0}^{(j)})\in H^{s}$. Assume the initial data $(n_{i0},n_{e0},\mathbf{u}_{\mathbf{i}0})$ for the QEP system \eqref{equ1}-\eqref{tum} has the expansion of the form \eqref{cut} and $(N_{i},N_{e},\mathbf{U})|_{t=0}=(N_{i0},N_{e0},\mathbf{U}_{0})$ satisfy \eqref{rem}.
Then for  $0<\tau<\tau_{\ast}$, there exists $\epsilon_{0}>0$ such that if $0<\epsilon<\epsilon_{0}$, the solution of the QEP system \eqref{equ1}-\eqref{tum} with initial data $(n_{i0},n_{e0},\mathbf{u}_{\mathbf{i}0})$ can be expressed as in the expansion \eqref{cut}, and the solutions $(N_{i},N_{e},\mathbf{U})$ of \eqref{rem} satisfy
\begin{equation}
\begin{split}\label{ineq}
\sup_{[0,\tau]}|\!|\!|(N_{e},\mathbf{U},N_{i})|\!|\!|_{\epsilon}^{2}\leq C(1+|\!|\!|(N_{e0},\mathbf{U}_{0},N_{i0})|\!|\!|_{\epsilon}^{2}).
\end{split}
\end{equation}
\end{theorem}
From \eqref{ineq}, we see that the $H^1$-norm of the remainder $(N_i,N_e,\mathbf{U})$ is bounded uniformly in $\epsilon$. Note also the Gardner-Morikawa transform \eqref{trans}, we see that
\begin{equation}
\sup_{[0,\epsilon^{-3/2}\tau]}
\left\|\left(
\begin{array}{c}
(n_{i}-1)/\epsilon\\ (n_{e}-1)/\epsilon\\ u_{i1}/\epsilon\\ u_{i2}/\epsilon^{\frac{3}{2}}
\end{array}
\right)-QKP/dKP\right\|_{H^1}\leq C\epsilon,
\end{equation}
for some $C>0$ independent of $\epsilon>0$. Here `QKP/dKP' is the solution of the first approximation $(n_i^{(1)},n_e^{(1)},\mathbf{u_i}^{(1)})$ in \eqref{e8}.

The following commutate estimates will be frequently used throughout.
\begin{lemma}[Commutator Estimate]\label{Lem2}
Let $m\geq1$ be an integer, and then the commutator which is defined by the following
\begin{equation}
\begin{split}\label{commu}
[\nabla^{m},f]g:=\nabla^{m}(fg)-f\nabla^{m}g,
\end{split}
\end{equation}
can be bounded by
\begin{equation}
\begin{split}\label{e16}
\|[\nabla^{m},f]g\|_{L^{p}}\leq \|\nabla f\|_{L^{p_{1}}}\|\nabla^{m-1}g\|_{L^{p_{2}}}+\|\nabla^{m} f\|_{L^{p_{3}}}\|g\|_{L^{p_{4}}},
\end{split}
\end{equation}
where $p,p_{2},p_{3}\in(1,\infty)$  and
\begin{equation*}
\begin{split}
\frac{1}{p}=\frac{1}{p_{1}}+\frac{1}{p_{2}}=\frac{1}{p_{3}}+\frac{1}{p_{4}}.
\end{split}
\end{equation*}
\end{lemma}
\begin{proof}
The proof can be found in \cite{CM86,KP88}, for example.
\end{proof}

\section{Uniform energy estimates}\label{sect-energy}
In this section, we give the energy estimates uniformly in $\epsilon$ for the remainder $(N_{e},N_{i},\mathbf{U})$, which requires a combination of energy method and analysis of the remainder equation \eqref{rem}. To simplify the presentation, we assume that \eqref{rem} has smooth solutions in $[0,\tau_{\epsilon}]$ for $\tau_{\epsilon}>0$ depending on $\epsilon$.
Let $\tilde{C}$ be a constant independent of $\epsilon$, which will be determined later, much larger than the bound $|\!|\!|(N_{e},N_{i},\mathbf{U})(0)|\!|\!|_{\epsilon}^{2}$ of the initial data. It is classical that there exists $\tau_{\epsilon}>0$ such that in $[0,\tau_{\epsilon}]$,
\begin{equation}
\begin{split}\label{priori}
|\!|\!|(N_{e},N_{i},\mathbf{U})|\!|\!|_{\epsilon}^{2}\leq\tilde{C}.
\end{split}
\end{equation}
As a direct corollary, there exists some $\epsilon_{1}>0$ such that $n_{e} \ and  \ n_{i}$ are bounded from above and below, say $\frac{1}{2}<n_{i},n_{e}<\frac{3}{2}$ and $|\mathbf{u_{i}}|<\frac{1}{2}$ when $\epsilon<\epsilon_{1}$. Since $\mathcal{R}_{2}^{3},\mathcal{R}_{3}^{3}, \mathcal{R}_{4}^{3}$ are smooth functions of $N_{e}$, there exists some constant $C_{1}=C_{1}(\epsilon\tilde{C})$ for any $\alpha,\beta\geq0$ such that
\begin{equation*}
\begin{split}
\left|\partial_{n_{e}^{(j)}}^{\alpha} \partial_{N_{e}}^{\beta}(\mathcal{R}_{2}^{3}, \mathcal{R}_{3}^{3}, \mathcal{R}_{4}^{3})\right|\leq C_{1}=C_{1}(\epsilon\tilde{C}),
\end{split}
\end{equation*}
where $C_{1}(\cdot)$ can be chosen to be nondecreasing in its argument.

The purpose of this section is to prove Proposition \ref{P1} and \ref{P2}. Since the proof of Proposition \ref{P1} will be almost the same to that of Proposition \ref{P2}, the proof of Proposition \ref{P1} will be omitted. In Subsection \ref{2.1}, we first show three lemmas that will be frequently used later. In Subsection \ref{2.2} and Subsection \ref{2.3}, we present and prove the two main propositions. Here, we only present the details of Lemma \ref{P3}, while estimates of some similar results are postponed to Subsection \ref{2.5}. For simplicity, we use $\|\cdot\|$ instead of $\|\cdot\|_{L^{2}}$ in the following.

\subsection{Basic estimates}\label{2.1}
We first prove the following Lemmas \ref{L1}-\ref{L3}, in which we bound $N_{i}$ and $\partial_{t}N_{e}$ in terms of $N_{e}$.
\begin{lemma}\label{L1}
Let $(N_{i},N_{e},\mathbf{U})$ be a solution to \eqref{rem} and $\alpha, \beta, k\geq0$ be integer. There exist some constants $0<\epsilon_{1}<1$ and $C_{1}=C_{1}(\epsilon\tilde{C})$ such that for every $0<\epsilon<\epsilon_{1}$,
\begin{equation}\label{equL1}
\begin{split}
C_{1}^{-1}\sum_{0\leq\alpha+\beta\leq k}\epsilon^{\alpha+2\beta}\|\partial_{x_{1}}^{\alpha}\partial_{x_{2}}^{\beta}N_{i}\|^{2}
\leq & \sum_{0\leq\alpha+\beta\leq k+4}\epsilon^{\alpha+2\beta}\|\partial_{x_{1}}^{\alpha}\partial_{x_{2}}^{\beta}N_{e}\|^{2}\\
\leq & C_{1}\sum_{0\leq\alpha+\beta\leq k}\epsilon^{\alpha+2\beta}\|\partial_{x_{1}}^{\alpha}\partial_{x_{2}}^{\beta}N_{i}\|^{2}.
\end{split}
\end{equation}
\end{lemma}
\begin{proof}
When $k=0$, taking inner product of \eqref{rem-4} with $N_{e}$ and integration by parts, we have
\begin{align}\label{equ6}\displaystyle
\|N_{e}\|&^{2}+\epsilon\int n_{e}(\partial_{x_{1}}N_{e})^{2}+\epsilon^{2}\int n_{e}(\partial_{x_{2}}N_{e})^{2}\nonumber\\
&+\frac{\epsilon^{2}H^{2}}{4}\int\frac{(\partial_{x_{1}}^{2}N_{e})^{2}+2\epsilon(\partial_{x_{1}x_{2}}N_{e})^{2}
+\epsilon^{2}(\partial_{x_{2}}^{2}N_{e})^{2}}{n_{e}}\nonumber\\
=&-\epsilon\int\partial_{x_{1}}n_{e}N_{e}\partial_{x_{1}}N_{e}-\epsilon^{2}\int\partial_{x_{2}}n_{e}N_{e}\partial_{x_{2}}N_{e}
+\epsilon^{2}\int\big(\partial_{x_{1}}\tilde{n_{e}}\partial_{x_{1}}N_{e}+\epsilon\partial_{x_{2}}\tilde{n_{e}}\partial_{x_{2}}N_{e}\big)N_{e}\nonumber\\
&+\epsilon^{6}\int\big((\partial_{x_{1}}N_{e})^{2}+\epsilon(\partial_{x_{2}}N_{e})^{2}\big)N_{e}
+\epsilon^{2}\int\big(\partial_{x_{1}}^{2}\tilde{n_{e}}+\epsilon\partial_{x_{2}}^{2}\tilde{n_{e}}\big)N_{e}^{2}\nonumber\\
&+\epsilon^{2}\int\mathcal{R}_{4}^{(1)}N_{e}
-\frac{\epsilon^{2}H^{2}}{2}\int(\partial_{x_{1}}\frac{1}{n_{e}})\partial_{x_{1}}^{2}N_{e}\partial_{x_{1}}N_{e}\nonumber\\
&-\frac{\epsilon^{2}H^{2}}{4}\int(\partial_{x_{1}}^{2}\frac{1}{n_{e}})\partial_{x_{1}}^{2}N_{e}N_{e}
-\frac{\epsilon^{3}H^{2}}{2}\int(\partial_{x_{1}}\frac{1}{n_{e}})\partial_{x_{1}x_{2}}N_{e}\partial_{x_{2}}N_{e}\nonumber\\
&-\frac{\epsilon^{3}H^{2}}{2}\int(\partial_{x_{2}}\frac{1}{n_{e}})\partial_{x_{1}x_{2}}N_{e}\partial_{x_{1}}N_{e}
-\frac{\epsilon^{3}H^{2}}{2}\int(\partial_{x_{1}x_{2}}\frac{1}{n_{e}})\partial_{x_{1}x_{2}}N_{e}N_{e}\nonumber\\
&-\frac{\epsilon^{4}H^{2}}{2}\int(\partial_{x_{2}}\frac{1}{n_{e}})\partial_{x_{2}}^{2}N_{e}\partial_{x_{2}}N_{e}
-\frac{\epsilon^{4}H^{2}}{4}\int(\partial_{x_{2}}^{2}\frac{1}{n_{e}})\partial_{x_{2}}^{2}N_{e}\nonumber\\
&+\frac{\epsilon^{2}H^{2}}{4}\int(\frac{C_{1}}{n_{e}^{2}}-\frac{C_{2}}{n_{e}^{3}}+\frac{C_{3}}{n_{e}^{4}}
+\frac{\mathcal{R}_{4}^{(2)}+\mathcal{R}_{4}^{(3)}}{n_{e}^{4}})N_{e} +\int{N_{e}N_{i}}\nonumber\\
=&:\sum_{i=1}^{15}D_{i}.
\end{align}
Since $\frac{1}{2}<n_{e}<\frac{3}{2}$ and $H$ is a fixed constant, there exists a fixed constant $C$ such that
the LHS of \eqref{equ6} is equal or greater than $C\big(\|N_{e}\|^{2}+\epsilon\|\partial_{x_{1}}N_{e}\|^{2} +\epsilon^{2}\|\partial_{x_{2}}N_{e}\|^{2} +\epsilon^{2}\|\partial_{x_{1}}^{2}N_{e}\|^{2}+\epsilon^{3}\|\partial_{x_{1}x_{2}}N_{e}\|^{2}+\epsilon^{4}\|\partial_{x_{2}}^{2}N_{e}\|^{2}\big)$. Next, we estimate the RHS of \eqref{equ6}. For $D_{1}$, since $n_{e}=1+\epsilon\tilde{n_{e}}+\epsilon^{5}N_{e}$, there exists some constant $C$ such that
\begin{equation*}
\begin{split}
D_{1}&=-\epsilon\int(\epsilon\partial_{x_{1}}\tilde{n_{e}}+\epsilon^{5}\partial_{x_{1}}N_{e})(\partial_{x_{1}}N_{e})N_{e}\\
&\leq C(1+\epsilon^{4}\|N_{e}\|_{L^{\infty}})(\epsilon\|N_{e}\|^{2}+\epsilon^{2}\|\partial_{x_{1}}N_{e}\|^{2})\\
&\leq C(1+\epsilon^{4}\|N_{e}\|_{H^{2}})(\epsilon\|N_{e}\|^{2}+\epsilon^{2}\|\partial_{x_{1}}N_{e}\|^{2})\\
&\leq C(\epsilon\tilde{C})(\epsilon\|N_{e}\|^{2}+\epsilon^{2}\|\partial_{x_{1}}N_{e}\|^{2})\\
&\leq C_{1}(\epsilon\|N_{e}\|^{2}+\epsilon^{2}\|\partial_{x_{1}}N_{e}\|^{2}),
\end{split}
\end{equation*}
where we have used H\"older's inequality, Cauchy inequality, Sobolev embedding $H^2\hookrightarrow L^{\infty}$ and the priori assumption \eqref{priori}. Similarly, we have
\begin{equation*}
\begin{split}
D_{2\sim6}\leq C(\epsilon\|N_{e}\|^{2}+\epsilon^{2}\|\partial_{x_{1}}N_{e}\|^{2}+\epsilon^{3}\|\partial_{x_{2}}N_{e}\|^{2}).
\end{split}
\end{equation*}
Note that
\begin{equation}
\begin{split}\label{one order}
\left|\partial_{x_{1}}\left(\frac{1}{n_{e}}\right)\right|
\leq C\left(\epsilon|\partial_{x_{1}}\tilde{n_{e}}| +\epsilon^{5}|\partial_{x_{1}}N_{e}|\right),
\end{split}
\end{equation}
\begin{equation}
\begin{split}\label{two order}
\left|\partial_{x_{1}}^{2}\left(\frac{1}{n_{e}}\right)\right|
\leq C\big(\epsilon+\epsilon^{5}(|\partial_{x_{1}}N_{e}|+|\partial_{x_{1}}^{2}N_{e}|)+\epsilon^{10}|\partial_{x_{1}}N_{e}|^{2}\big),
\end{split}
\end{equation}
and
\begin{equation}
\begin{split}\label{two orders}
\left|\partial_{x_{1}x_{2}}\left(\frac{1}{n_{e}}\right)\right|
\leq C\big(\epsilon+\epsilon^{5}(|\partial_{x_{1}x_{2}}N_{e}|+|\partial_{x_{2}}N_{e}|)+\epsilon^{10}|\partial_{x_{1}}N_{e}||\partial_{x_{2}}N_{e}|\big).
\end{split}
\end{equation}
Thus similarly we have
\begin{equation*}
\begin{split}
D_{7\sim13}\leq C_{1}\big(\epsilon\|N_{e}\|^{2}+\epsilon^{2}\|\partial_{x_{1}}N_{e}\|^{2}+\epsilon^{3}\|\partial_{x_{1}}^{2}N_{e}\|^{2}
+\epsilon^{3}\|\partial_{x_{2}}N_{e}\|^{2}+\epsilon^{4}\|\partial_{x_{1}x_{2}}N_{e}\|^{2}+\epsilon^{5}\|\partial_{x_{2}}^{2}N_{e}\|^{2}\big).
\end{split}
\end{equation*}
By the expression of $C_{i}(1\leq i\leq3)$ and Lemma \ref{Lem1}, we similarly have
\begin{equation*}
\begin{split}
D_{14}\leq C_{1}\big(\epsilon\|N_{e}\|^{2}+\epsilon^{2}\|\partial_{x_{1}}N_{e}\|^{2}+\epsilon^{3}\|\partial_{x_{1}}^{2}N_{e}\|^{2}
+\epsilon^{3}\|\partial_{x_{2}}N_{e}\|^{2}+\epsilon^{4}\|\partial_{x_{1}x_{2}}N_{e}\|^{2}+\epsilon^{5}\|\partial_{x_{2}}^{2}N_{e}\|^{2}\big),
\end{split}
\end{equation*}
thanks to the priori assumption \eqref{priori} again. By virtue of Young inequality, we obtain
\begin{equation*}
\begin{split}
\int N_{e}N_{i}\leq\delta\|N_{e}\|^{2}+C_{\delta}\|N_{i}\|^{2},
\end{split}
\end{equation*}
for arbitrary $\delta>0$.
Hence, there exists some $\epsilon_1>0$ such that for $0<\epsilon<\epsilon_1$,
\begin{equation}
\begin{split}\label{equ7}
\|N_{e}\|^{2}+\epsilon\|\partial_{x_{1}}N_{e}\|^{2}+\epsilon^{2}\|\partial_{x_{1}}^{2}N_{e}\|^{2} +\epsilon^{2}\|\partial_{x_{2}}N_{e}\|^{2}+\epsilon^{3}\|\partial_{x_{1}x_{2}}N_{e}\|^{2}+\epsilon^{4}\|\partial_{x_{2}}^{2}N_{e}\|^{2}
\leq C_{1}\|N_{i}\|^{2}.
\end{split}
\end{equation}
Taking inner product of \eqref{rem-4} with $\epsilon \partial_{x_{1}}^{2}N_{e}$, $\epsilon^{2} \partial_{x_{2}}^{2}N_{e}$,  $\epsilon^{2}\partial_{x_{1}}^{4}N_{e}$,  and $\epsilon^{4}\partial_{x_{2}}^{4}N_{e}$ respectively, and applying the Cauchy inequality, Sobolev embedding $H^2\hookrightarrow L^{\infty}$ and the priori assumption \eqref{priori}, we have similarly the following inequalities
\begin{equation}
\begin{split}\label{equ8}
\epsilon&\|\partial_{x_{1}}N_{e}\|^{2}+\epsilon^{2}\|\partial_{x_{1}}^{2}N_{e}\|^{2} +\epsilon^{3}\|\partial_{x_{1}x_{2}}N_{e}\|^{2}+\epsilon^{3}\|\partial_{x_{1}}^{3}N_{e}\|^{2}
+\epsilon^{4}\|\partial_{x_{1}}^{2}\partial_{x_{2}}N_{e}\|^{2}+\epsilon^{5}\|\partial_{x_{1}}\partial_{x_{2}}^{2}N_{e}\|^{2}\\
&\leq C_{1}\big(\|N_{i}\|^{2}+\epsilon\|N_{e}\|^{2}+\epsilon^{3}\|\partial_{x_{2}}N_{e}\|^{2}+\epsilon^{5}\|\partial_{x_{2}}^{2}N_{e}\|^{2}\big),
\end{split}
\end{equation}
\begin{equation}
\begin{split}\label{equ9}
\epsilon^{2}&\|\partial_{x_{2}}N_{e}\|^{2}+\epsilon^{4}\|\partial_{x_{2}}^{2}N_{e}\|^{2} +\epsilon^{3}\|\partial_{x_{1}x_{2}}N_{e}\|^{2}
+\epsilon^{4}\|\partial_{x_{1}}^{2}\partial_{x_{2}}N_{e}\|^{2}+\epsilon^{5}\|\partial_{x_{1}}\partial_{x_{2}}^{2}N_{e}\|^{2}
+\epsilon^{6}\|\partial_{x_{2}}^{3}N_{e}\|^{2}\\
&\leq C_{1}\big(\|N_{i}\|^{2}+\epsilon\|N_{e}\|^{2}+\epsilon^{2}\|\partial_{x_{1}}N_{e}\|^{2}+\epsilon^{3}\|\partial_{x_{1}}^{2}N_{e}\|^{2}\big),
\end{split}
\end{equation}
\begin{equation}
\begin{split}\label{equ10}
\epsilon^{2}&\|\partial_{x_{1}}^{2}N_{e}\|^{2}+\epsilon^{3}\|\partial_{x_{1}}^{3}N_{e}\|^{2}
+\epsilon^{4}\|\partial_{x_{1}}^{2}\partial_{x_{2}}N_{e}\|^{2}+\epsilon^{4}\|\partial_{x_{1}}^{4}N_{e}\|^{2}
+\epsilon^{5}\|\partial_{x_{1}}^{3}\partial_{x_{2}}N_{e}\|^{2}+\epsilon^{6}\|\partial_{x_{1}}^{2}\partial_{x_{2}}^{2}N_{e}\|^{2}\\
\leq &C_{1}\big(\|N_{i}\|^{2}+\epsilon\|N_{e}\|^{2}+\epsilon\|\partial_{x_{1}}N_{e}\|^{2}+\epsilon^{3}\|\partial_{x_{2}}N_{e}\|^{2}
+\epsilon^{4}\|\partial_{x_{1}x_{2}}N_{e}\|^{2}+\epsilon^{5}\|\partial_{x_{2}}^{2}N_{e}\|^{2}\\
&+\epsilon^{7}\|\partial_{x_{2}}^{3}N_{e}\|^{2}
+\epsilon^{6}\|\partial_{x_{1}}\partial_{x_{2}}^{2}N_{e}\|^{2}\big),
\end{split}
\end{equation}
\begin{equation}
\begin{split}\label{equ11}
\epsilon^{4}&\|\partial_{x_{2}}^{2}N_{e}\|^{2}+\epsilon^{5}\|\partial_{x_{1}}\partial_{x_{2}}^{2}N_{e}\|^{2}
+\epsilon^{6}\|\partial_{x_{2}}^{3}N_{e}\|^{2}+\epsilon^{6}\|\partial_{x_{1}}^{2}\partial_{x_{2}}^{2}N_{e}\|^{2}
+\epsilon^{7}\|\partial_{x_{1}}\partial_{x_{2}}^{3}N_{e}\|^{2}+\epsilon^{8}\|\partial_{x_{2}}^{4}N_{e}\|^{2}\\
\leq &C_{1}\big(\|N_{i}\|^{2}+\epsilon\|N_{e}\|^{2}+\epsilon\|\partial_{x_{1}}N_{e}\|^{2}+\epsilon^{3}\|\partial_{x_{2}}N_{e}\|^{2}
+\epsilon^{4}\|\partial_{x_{1}x_{2}}N_{e}\|^{2}+\epsilon^{3}\|\partial_{x_{1}}^{2}N_{e}\|^{2}\\
&+\epsilon^{4}\|\partial_{x_{1}}^{3}N_{e}\|^{2}
+\epsilon^{5}\|\partial_{x_{1}}^{2}\partial_{x_{2}}N_{e}\|^{2}\big).
\end{split}
\end{equation}
Putting \eqref{equ7}-\eqref{equ11} together, we obtain
\begin{equation}
\begin{split}\label{equ12}
\sum_{0\leq\alpha+\beta\leq4}\epsilon^{\alpha+2\beta}\|\partial_{x_{1}}^{\alpha}\partial_{x_{2}}^{\beta}N_{e}\|^{2}
\leq C_{1}\|N_{i}\|^{2}.
\end{split}
\end{equation}
On the other hand, by \eqref{rem-4}, it follows from the H\"older inequality, Cauchy inequality and the priori assumption \eqref{priori} that
\begin{equation}
\begin{split}\label{equ13}
C_{1}^{-1}\|N_{i}\|^{2}\leq\sum_{0\leq\alpha+\beta\leq4} \epsilon^{\alpha+2\beta}\|\partial_{x_{1}}^{\alpha} \partial_{x_{2}}^{\beta}N_{e}\|^{2}.
\end{split}
\end{equation}
Combining \eqref{equ12} with \eqref{equ13}, we deduce the inequality \eqref{equL1} for $k=0$. For higher order inequalities, we differentiate \eqref{rem-4} with $\partial_{x_{1}}^{\alpha}$ and $\partial_{x_{2}}^{\beta} \ (\alpha, \beta=k+2, k\geq1) $ and then take inner product with $\epsilon^{\alpha}\partial_{x_{1}}^{\alpha}N_{e}$ and $\epsilon^{2\beta}\partial_{x_{2}}^{\beta}N_{e}$ separately, and then putting the results with \eqref{equ12} together, thus we obtain the RHS of inequality \eqref{equL1}.  On the other hand,  differentiating \eqref{rem-4} with $\partial_{x_{1}}^{\alpha}\partial_{x_{2}}^{\beta} \ (\alpha+\beta\geq k)$  and then taking inner product with $\epsilon^{\alpha+2\beta}\partial_{x_{1}}^{\alpha}\partial_{x_{2}}^{\beta}N_{i}$  separately. The Lemma then follows by the same procedure of the above.

Recall $|\!|\!|(N_{e},\mathbf{U})|\!|\!|_{\epsilon}$ in \eqref{|||}. In fact, we only need $0\leq k \leq3$ in Lemma \ref{L1}.
\end{proof}

\begin{lemma}\label{L2}
 Let $(N_{i},N_{e},\mathbf{U})$ be a solution to \eqref{rem} and $\alpha, \beta, k\geq0$ be integer. There exist some constants $C$ and $C_{1}=C_{1}(\epsilon\tilde{C})$ such that
\begin{equation}\label{equL2}
\begin{split}
\epsilon\|\epsilon\sum_{0\leq\alpha+\beta\leq k}\epsilon^{\alpha+2\beta}\partial_{t}\partial_{x_{1}}^{\alpha}\partial_{x_{2}}^{\beta}N_{i}\|^{2}
\leq C&
\Bigg(\sum_{0\leq\alpha+\beta\leq k+1}\epsilon^{\alpha+2\beta}(\|\partial_{x_{1}}^{\alpha}\partial_{x_{2}}^{\beta}U_{1}\|^{2}
+\|\partial_{x_{1}}^{\alpha}\partial_{x_{2}}^{\beta}U_{2}\|^{2})\\
&+\sum_{0\leq\alpha+\beta\leq k+5}\epsilon^{\alpha+2\beta}\|\partial_{x_{1}}^{\alpha}\partial_{x_{2}}^{\beta}N_{e}\|^{2}\Bigg)+C\epsilon.
\end{split}
\end{equation}
\end{lemma}
\begin{proof}
From \eqref{rem-1}, we have
\begin{equation*}
\begin{split}
\epsilon\partial_{t}N_{i}=&(V-u_{i1})\partial_{x_{1}}N_{i}-\epsilon^{\frac{1}{2}}u_{i2}\partial_{x_{2}}N_{i}
-n_{i}\partial_{x_{1}}U_{1}-\epsilon^{\frac{1}{2}}n_{i}\partial_{x_{2}}U_{2}\\
&-\epsilon(\partial_{x_{1}}\tilde{u_{i1}}
+\epsilon^{\frac{1}{2}}\partial_{x_{2}}\tilde{u_{i2}})N_{i}
-\epsilon\partial_{x_{1}}\tilde{n_{i}}U_{1}-\epsilon\epsilon^{\frac{1}{2}}\partial_{x_{2}}\tilde{n_{i}}U_{2}-\epsilon^{2}\mathcal{R}_{1}.
\end{split}
\end{equation*}
Since $\frac{1}{2}<n_{i}<\frac{3}{2}$ and $|\mathbf{u_{i}}|<\frac{1}{2}$, taking $L^{2}$-norm yields
\begin{equation*}
\begin{split}
\|\epsilon\partial_{t}N_{i}\|^2\leq &\|(V-u_{i})\partial_{x_{1}}N_{i}\|^2 +\epsilon\|u_{i2}\partial_{x_{2}}N_{i}\|^2
+\|n_{i}\partial_{x_{1}}U_{1}\|^2 +\epsilon\|n_{i}\partial_{x_{2}}U_{2}\|^2 +\epsilon^2\|\partial_{x_{1}}\tilde u_{i1}N_{i}\|^2\\
&+\epsilon^3\|\partial_{x_{2}}\tilde u_{i2}N_{i}\|^2
+ \epsilon^2\|\partial_{x_{1}}\tilde n_{i}U_{1}\|^2 + \epsilon^3\|\partial_{x_{2}}\tilde n_{i}U_{2}\|^2+\epsilon^4\|\mathcal R_1\|^2\\
\leq& C\big(\|\partial_{x_{1}}N_{i}\|^2+\epsilon\|\partial_{x_{2}}N_{i}\|^2+\|\partial_{x_{1}}U_{1}\|^2+\epsilon\|\partial_{x_{2}}U_{2}\|^2\big) \\ &+C\epsilon^2\big(\epsilon^2+\|N_{i}\|^2+\|U_{1}\|^2+\|U_{2}\|^2\big).
\end{split}
\end{equation*}
Applying Lemma \ref{L1}, we have inequality for $k=0$,
\begin{equation}
\begin{split}\label{kk}
\epsilon\|\epsilon\partial_{t}N_{i}\|^2
\leq C\big(\|U_{1}\|^2+\|U_{2}\|^2+
\epsilon\|\partial_{x_{1}}U_{1}\|^2+\epsilon^{2}\|\partial_{x_{2}}U_{2}\|^2
+\sum_{0\leq\alpha+\beta\leq4}\epsilon^{\alpha+2\beta}\|\partial_{x_{1}}^{\alpha}\partial_{x_{2}}^{\beta}N_{e}\|^{2}\big).
\end{split}
\end{equation}
To prove \eqref{equL2}, we take $\epsilon^{\alpha+2\beta}\partial_{x_{1}}^{\alpha}\partial_{x_{2}}^{\beta} \ (\alpha+2\beta=k, \ k\geq1)$  of \eqref{rem-1} respectively and then sum the results with \eqref{kk}.
\end{proof}

Recall $|\!|\!|(N_{e},\mathbf{U})|\!|\!|_{\epsilon}$ in \eqref{|||}. In fact, we only need $0\leq k \leq2$ in Lemma \ref{L2}.\\

\begin{lemma}\label{L3}
 Let $(N_{i},N_{e},\mathbf{U})$ be a solution to \eqref{rem} and $\alpha, \beta, k\geq0$ be integer. There exist some constants $C_{1}=C_{1}(\epsilon\tilde{C})$ and $\epsilon_{1}>0$ such that for every $0<\epsilon<\epsilon_{1}$,
\begin{equation}\label{equL3}
\begin{split}
\sum_{0\leq\alpha+\beta\leq k+4}\epsilon^{\alpha+2\beta}\|\partial_{t}\partial_{x_{1}}^{\alpha}\partial_{x_{2}}^{\beta}N_{e}\|^{2}
\leq C\sum_{0\leq\alpha+\beta\leq k}\epsilon^{\alpha+2\beta}\|\partial_{t}\partial_{x_{1}}^{\alpha}\partial_{x_{2}}^{\beta}N_{i}\|^{2}+C_{1}.
\end{split}
\end{equation}
\end{lemma}
\begin{proof}
The proof is similar to that of Lemma \ref{L1}. When $k=0$, by first taking $\partial_{t}$ of \eqref{rem-4} and then taking inner product with $\partial_{t}N_{e}$ and integration by parts, we have
\begin{align}\label{m1}
\|\partial_{t}&N_{e}\|^{2}
+\epsilon\int n_{e}(\partial_{tx_{1}}N_{e})^{2}+\epsilon^{2}\int n_{e}(\partial_{tx_{2}}N_{e})^{2}+\frac{\epsilon^{2}H^{2}}{4}\int\frac{1}{n_{e}}(\partial_{t}\partial_{x_{1}}^{2}N_{e})^{2}\nonumber\\
&+\frac{\epsilon^{3}H^{2}}{2}\int\frac{1}{n_{e}}(\partial_{t}\partial_{x_{1}x_{2}}N_{e})^{2}
+\frac{\epsilon^{4}H^{2}}{4}\int\frac{1}{n_{e}}(\partial_{t}\partial_{x_{2}}^{2}N_{e})^{2}\nonumber\\
=&-\epsilon\int\partial_{x_{1}}n_{e}\partial_{tx_{1}}N_{e}\partial_{t}N_{e}
+\epsilon\int\partial_{t}n_{e}\partial_{x_{1}}^{2}N_{e}\partial_{t}N_{e}
-\epsilon^{2}\int\partial_{x_{2}}n_{e}\partial_{tx_{2}}N_{e}\partial_{t}N_{e}\nonumber\\
&+\epsilon^{2}\int\partial_{t}n_{e}\partial_{x_{2}}^{2}N_{e}\partial_{t}N_{e}
+\epsilon^{2}\int\partial_{t}(\partial_{x_{1}}\tilde{n_{e}}\partial_{x_{1}}N_{e}
+\epsilon\partial_{x_{2}}\tilde{n_{e}}\partial_{x_{2}}N_{e})\partial_{t}N_{e}\nonumber\\
&+\epsilon^{6}\int\partial_{t}((\partial_{x_{1}}N_{e})^{2}
+\epsilon(\partial_{x_{2}}N_{e})^{2})\partial_{t}N_{e}
+\epsilon^{2}\int\partial_{t}(\partial_{x_{1}}^{2}\tilde{n_{e}}N_{e}
+\epsilon\partial_{x_{2}}^{2}\tilde{n_{e}}N_{e})\partial_{t}N_{e}\nonumber\\
&+\epsilon^{3}\int\partial_{t}\mathcal{R}_{4}^{(1)}\partial_{t}N_{e}
-\frac{\epsilon^{2}H^{2}}{2}\int(\partial_{x_{1}}\frac{1}{n_{e}})\partial_{t}\partial_{x_{1}}^{2}N_{e}\partial_{tx_{1}}N_{e}
-\frac{\epsilon^{2}H^{2}}{4}\int(\partial_{x_{1}}^{2}\frac{1}{n_{e}})\partial_{t}\partial_{x_{1}}^{2}N_{e}\partial_{t}N_{e}\nonumber\\
&-\frac{\epsilon^{2}H^{2}}{4}\int(\partial_{t}\frac{1}{n_{e}})\partial_{x_{1}}^{4}N_{e}\partial_{t}N_{e}
-\frac{\epsilon^{3}H^{2}}{2}\int(\partial_{x_{1}}\frac{1}{n_{e}})\partial_{t}\partial_{x_{1}x_{2}}N_{e}\partial_{tx_{2}}N_{e}\nonumber\\
&-\frac{\epsilon^{3}H^{2}}{2}\int(\partial_{x_{2}}\frac{1}{n_{e}})\partial_{t}\partial_{x_{1}x_{2}}N_{e}\partial_{tx_{1}}N_{e}
-\frac{\epsilon^{3}H^{2}}{2}\int(\partial_{x_{1}x_{2}}\frac{1}{n_{e}})\partial_{t}\partial_{x_{1}x_{2}}N_{e}\partial_{t}N_{e}\nonumber\\
&-\frac{\epsilon^{3}H^{2}}{2}\int(\partial_{t}\frac{1}{n_{e}})\partial_{x_{1}}^{2}\partial_{x_{2}}^{2}N_{e}\partial_{t}N_{e}
-\frac{\epsilon^{4}H^{2}}{2}\int(\partial_{x_{2}}\frac{1}{n_{e}})\partial_{t}\partial_{x_{2}}^{2}N_{e}\partial_{tx_{2}}N_{e}\nonumber\\
&-\frac{\epsilon^{4}H^{2}}{4}\int(\partial_{x_{2}}^{2}\frac{1}{n_{e}})\partial_{t}\partial_{x_{2}}^{2}N_{e}\partial_{t}N_{e}
-\frac{\epsilon^{4}H^{2}}{4}\int(\partial_{t}\frac{1}{n_{e}})\partial_{x_{2}}^{4}N_{e}\partial_{t}N_{e}\nonumber\\
&+\frac{\epsilon^{2}H^{2}}{4}\int\partial_{t}\big(\frac{C_{1}}{n_{e}^{2}}-\frac{C_{2}}{n_{e}^{3}}+\frac{C_{3}}{n_{e}^{4}}
+\frac{\epsilon\mathcal{R}_{4}^{(2)}+\epsilon\mathcal{R}_{4}^{(3)}}{n_{e}^{4}}\big)\partial_{t}N_{e}
+\int\partial_{t}N_{i}\partial_{t}N_{e}\nonumber\\
=&:\sum_{i=1}^{20}E_{i}.
\end{align}
\emph{Estimate of the LHS of \eqref{m1}.}
Since $\frac{1}{2}<n_{e}<\frac{3}{2}$ and $H$ is a fixed constant, there exists a fixed constant $C$ such that
the LHS of \eqref{m1} is equal or greater than $C(\|\partial_{t}N_{e}\|^{2}+\epsilon\|\partial_{tx_{1}}N_{e}\|^{2}+\epsilon^{2}\|\partial_{tx_{2}}N_{e}\|^{2}
+\epsilon^{2}\|\partial_{t}\partial_{x_{1}}^{2}N_{e}\|^{2}+\epsilon^{3}\|\partial_{t}\partial_{x_{1}x_{2}}N_{e}\|^{2}
+\epsilon^{2}\|\partial_{t}\partial_{x_{2}}^{2}N_{e}\|^{2})$. Next, we estimate the right hand side terms. For $E_{1}$, by applying H\"older's inequality and Cauchy inequality, we have
\begin{equation*}
\begin{split}
E_{1}\leq &C(1+\epsilon^{9}\|\partial_{x_{1}}N_{e}\|_{L^{\infty}}^{2})(\epsilon^{2}\|\partial_{tx_{1}}N_{e}\|^{2}+\epsilon\|\partial_{t}N_{e}\|^{2})\\
\leq &C(1+\epsilon^{9}\|\partial_{x_{1}}N_{e}\|_{H^{2}}^{2})(\epsilon^{2}\|\partial_{tx_{1}}N_{e}\|^{2}+\epsilon\|\partial_{t}N_{e}\|^{2})\\
\leq &C(\epsilon\tilde{C})(\epsilon^{2}\|\partial_{tx_{1}}N_{e}\|^{2}+\epsilon\|\partial_{t}N_{e}\|^{2})\\
\leq &C_{1}(\epsilon^{2}\|\partial_{tx_{1}}N_{e}\|^{2}+\epsilon\|\partial_{t}N_{e}\|^{2}),
\end{split}
\end{equation*}
where we have used \eqref{priori} and Sobolev embedding $H^2\hookrightarrow L^{\infty}$. Similarly,
\begin{equation*}
\begin{split}
E_{2\sim8}\leq C_{1}(\epsilon\|\partial_{t}N_{e}\|^{2}+\epsilon^{2}\|\partial_{tx_{1}}N_{e}\|^{2}+\epsilon^{3}\|\partial_{tx_{2}}N_{e}\|^{2})+C_{1}.
\end{split}
\end{equation*}
\emph{Estimate of $E_{9}$.}
By applying H\"older's inequality and Cauchy inequality, we have
\begin{equation*}
\begin{split}E_{9}\leq&(1+\epsilon^{6}\|\partial_{x_{1}}N_{e}\|_{L^{\infty}}^{2})(\epsilon^{2}\|\partial_{tx_{1}}N_{e}\|^{2}
+\epsilon^{3}\|\partial_{t}\partial_{x_{1}}^{2}N_{e}\|^{2})\\
\leq &C_{1}(\epsilon^{2}\|\partial_{tx_{1}}N_{e}\|^{2}
+\epsilon^{3}\|\partial_{t}\partial_{x_{1}}^{2}N_{e}\|^{2}),
\end{split}
\end{equation*}
where we have used \eqref{one order}.
Similar to \eqref{one order}, we note that
\begin{equation}
\begin{split}\label{one}
|\partial_{t}\frac{1}{n_{e}}|
\leq C(\epsilon|\partial_{t}\tilde{n_{e}}|+\epsilon^{5}|\partial_{t}N_{e}|).
\end{split}
\end{equation}
Thus, similarly by using \eqref{one order}, \eqref{two order}, \eqref{two orders} and \eqref{one}, we have
\begin{equation*}
\begin{split}
E_{10\sim19}\leq C_{1}(&\epsilon\|\partial_{t}N_{e}\|^{2}+\epsilon^{2}\|\partial_{tx_{1}}N_{e}\|^{2}+\epsilon^{3}\|\partial_{tx_{2}}N_{e}\|^{2}
+\epsilon^{3}\|\partial_{t}\partial_{x_{1}}^{2}N_{e}\|^{2}\\
&+\epsilon^{4}\|\partial_{t}\partial_{x_{1}x_{2}}N_{e}\|^{2}
+\epsilon^{5}\|\partial_{t}\partial_{x_{2}}^{2}N_{e}\|^{2})+C_{1}.
\end{split}
\end{equation*}
\emph{Estimate of $B_{25}$.} Applying Young inequality, we have
\begin{equation*}
\begin{split}
B_{25}=\int\partial_{t}N_{i}\partial_{t}N_{e}\leq\gamma\|\partial_{t}N_{e}\|^{2}+C_{\gamma}\|\partial_{t}N_{i}\|^{2},
\end{split}
\end{equation*}
where for arbitrary small $\gamma>0$. Hence, we have shown that there exists some $\epsilon_1>0$ such that for $0<\epsilon<\epsilon_1$, we have
\begin{equation}\label{lemma3.1}
\begin{split}
\|\partial_{t}N_{e}\|^{2}&+\epsilon\|\partial_{tx_{1}}N_{e}\|^{2} +\epsilon^{2}\|\partial_{tx_{2}}N_{e}\|^{2} +\epsilon^{2}\|\partial_{t}\partial_{x_{1}}^{2}N_{e}\|^{2}\\ &+\epsilon^{3}\|\partial_{t}\partial_{x_{1}x_{2}}N_{e}\|^{2} +\epsilon^{4}\|\partial_{t}\partial_{x_{2}}^{2}N_{e}\|^{2}\leq C_{\gamma_{1}}\|\partial_{t}N_{i}\|^{2}+C_{1}.
\end{split}
\end{equation}
Similarly, taking $\partial_{tx_{1}},\partial_{tx_{2}}, \partial_{t}\partial_{x_{1}}^{2}, \partial_{t}\partial_{x_{1}x_{2}},\partial_{t}\partial_{x_{2}}^{2}$ of \eqref{rem-3} and then taking inner product with $\epsilon\partial_{tx_{1}}N_{e}, \epsilon^{2}\partial_{tx_{2}}N_{e}, \epsilon^{2}\partial_{t}\partial_{x_{1}}^{2}N_{e}, \epsilon^{3}\partial_{t}\partial_{x_{1}x_{2}}, \epsilon^{4}\partial_{x_{2}}^{2}N_{e}$ respectively, we have
\begin{align}
\epsilon\|&\partial_{tx_{1}}N_{e}\|^{2} +\epsilon^{2}\|\partial_{t}\partial_{x_{1}}^{2}N_{e}\|^{2} +\epsilon^{3}\|\partial_{t}\partial_{x_{1}x_{2}}N_{e}\|^{2} +\epsilon^{3}\|\partial_{t}\partial_{x_{1}}^{3}N_{e}\|^{2} +\epsilon^{4}\|\partial_{t}\partial_{x_{1}}^{2} \partial_{x_{2}}N_{e}\|^{2}\nonumber\\ &+\epsilon^{5}\|\partial_{t}\partial_{x_{1}} \partial_{x_{2}}^{2}N_{e}\|^{2} \leq C_{\gamma_{2}} (\|\partial_{t}N_{i}\|^{2}+\epsilon\|\partial_{t}N_{e}\|^{2} +\epsilon^{3}\|\partial_{tx_{2}}N_{e}\|^{2} +\epsilon^{5}\|\partial_{t}\partial_{x_{2}}^{2}N_{e}\|^{2}),\label{lemma3.2}\\
\epsilon^{2}\|&\partial_{tx_{2}}N_{e}\|^{2} +\epsilon^{3}\|\partial_{t}\partial_{x_{1}x_{2}}N_{e}\|^{2} +\epsilon^{4}\|\partial_{t}\partial_{x_{2}}^{2}N_{e}\|^{2} +\epsilon^{4}\|\partial_{t}\partial_{x_{1}}^{2} \partial_{x_{2}}N_{e}\|^{2}\nonumber\\ &+\epsilon^{5}\|\partial_{t}\partial_{x_{1}}\partial_{x_{2}}^{2}N_{e}\|^{2} +\epsilon^{6}\|\partial_{t}\partial_{x_{2}}^{3}N_{e}\|^{2}\nonumber\\ \leq &C_{\gamma_{3}} (\|\partial_{t}N_{i}\|^{2}+\epsilon\|\partial_{t}N_{e}\|^{2} +\epsilon^{2}\|\partial_{tx_{1}}N_{e}\|^{2} +\epsilon^{3}\|\partial_{t}\partial_{x_{1}}^{2}N_{e}\|^{2}),\label{lemma3.3}\\
\epsilon^{2}\|&\partial_{t}\partial_{x_{1}}^{2}N_{e}\|^{2}
+\epsilon^{3}\|\partial_{t}\partial_{x_{1}}^{3}N_{e}\|^{2}
+\epsilon^{4}\|\partial_{t}\partial_{x_{1}}^{2}\partial_{x_{2}}N_{e}\|^{2}
+\epsilon^{4}\|\partial_{t}\partial_{x_{1}}^{4}N_{e}\|^{2}\nonumber\\
&+\epsilon^{5}\|\partial_{t}\partial_{x_{1}}^{3}\partial_{x_{2}}N_{e}\|^{2}
+\epsilon^{6}\|\partial_{t}\partial_{x_{1}}^{2}\partial_{x_{2}}^{2}N_{e}\|^{2}\nonumber\\
\leq &C_{\gamma_{4}} (\|\partial_{t}N_{i}\|^{2}+\epsilon\|\partial_{t}N_{e}\|^{2}+\epsilon^{2}\|\partial_{tx_{1}}N_{e}\|^{2}
+\epsilon^{3}\|\partial_{tx_{2}}N_{e}\|^{2}+\epsilon^{4}\|\partial_{t}\partial_{x_{1}x_{2}}N_{e}\|^{2}\nonumber\\
&+\epsilon^{5}\|\partial_{t}\partial_{x_{2}}^{2}N_{e}\|^{2}+\epsilon^{6}\|\partial_{t}\partial_{x_{1}}\partial_{x_{2}}^{2}N_{e}\|^{2}
+\epsilon^{7}\|\partial_{t}\partial_{x_{2}}^{3}N_{e}\|^{2}),\label{lemma3.4}\\
\epsilon^{3}\|&\partial_{t}\partial_{x_{1}x_{2}}N_{e}\|^{2}
+\epsilon^{4}\|\partial_{t}\partial_{x_{1}}^{2}\partial_{x_{2}}N_{e}\|^{2}
+\epsilon^{5}\|\partial_{t}\partial_{x_{1}}\partial_{x_{2}}^{2}N_{e}\|^{2}
+\epsilon^{5}\|\partial_{t}\partial_{x_{1}}^{3}\partial_{x_{2}}N_{e}\|^{2}\nonumber\\
&+\epsilon^{6}\|\partial_{t}\partial_{x_{1}}^{2}\partial_{x_{2}}^{2}N_{e}\|^{2}
+\epsilon^{7}\|\partial_{t}\partial_{x_{1}}\partial_{x_{2}}^{3}N_{e}\|^{2}\nonumber\\
\leq &C_{\gamma_{5}} (\|\partial_{t}N_{i}\|^{2}+\epsilon\|\partial_{t}N_{e}\|^{2}+\epsilon^{2}\|\partial_{tx_{1}}N_{e}\|^{2}
+\epsilon^{3}\|\partial_{tx_{2}}N_{e}\|^{2}+\epsilon^{3}\|\partial_{t}\partial_{x_{1}}^{2}N_{e}\|^{2}\nonumber\\
&+\epsilon^{4}\|\partial_{t}\partial_{x_{1}}^{3}N_{e}\|^{2}
+\epsilon^{5}\|\partial_{t}\partial_{x_{2}}^{2}N_{e}\|^{2}
+\epsilon^{7}\|\partial_{t}\partial_{x_{2}}^{3}N_{e}\|^{2}),\label{lemma3.5}\\
\epsilon^{4}\|&\partial_{t}\partial_{x_{2}}^{2}N_{e}\|^{2}
+\epsilon^{5}\|\partial_{t}\partial_{x_{1}}\partial_{x_{2}}^{2}N_{e}\|^{2}
+\epsilon^{6}\|\partial_{t}\partial_{x_{2}}^{3}N_{e}\|^{2}
+\epsilon^{7}\|\partial_{t}\partial_{x_{1}}^{2}\partial_{x_{2}}^{2}N_{e}\|^{2}\nonumber\\
&+\epsilon^{8}\|\partial_{t}\partial_{x_{1}}\partial_{x_{2}}^{3}N_{e}\|^{2}
+\epsilon^{6}\|\partial_{t}\partial_{x_{2}}^{4}N_{e}\|^{2}\nonumber\\
\leq &C_{\gamma_{6}} (\|\partial_{t}N_{i}\|^{2}
+\epsilon\|\partial_{t}N_{e}\|^{2}
+\epsilon^{2}\|\partial_{tx_{1}}N_{e}\|^{2}
+\epsilon^{3}\|\partial_{tx_{2}}N_{e}\|^{2}
+\epsilon^{3}\|\partial_{t}\partial_{x_{1}}^{2}N_{e}\|^{2}\nonumber\\
&+\epsilon^{4}\|\partial_{t}\partial_{x_{1}x_{2}}N_{e}\|^{2}
+\epsilon^{4}\|\partial_{t}\partial_{x_{1}}^{3}N_{e}\|^{2}
+\epsilon^{5}\|\partial_{t}\partial_{x_{1}}^{2}\partial_{x_{2}}N_{e}\|^{2}.\label{lemma3.6}
\end{align}
Putting \eqref{lemma3.1} to \eqref{lemma3.6} together, let $C=\max{C_{\gamma_{i}},1\leq i\leq6}$, we obtain
\begin{equation*}
\begin{split}
\sum_{0\leq\alpha+\beta\leq4}\epsilon^{\alpha+2\beta}\|\partial_{t}\partial_{x_{1}}^{\alpha}\partial_{x_{2}}^{\beta}N_{e}\|^{2}
\leq C\|\partial_{t}N_{i}\|^{2}+C_{1}.
\end{split}
\end{equation*}
For higher order inequalities, we differentiate \eqref{rem-3} with $\partial_{t}\partial_{x_{1}}^{\alpha}\partial_{x_{2}}^{\beta}$ for $\alpha+\beta=k+2$ and then take inner product with $\epsilon^{\alpha+2\beta}\partial_{t}\partial_{x_{1}}^{\alpha}\partial_{x_{2}}^{\beta}N_{e}$  separately.
Thus we have proven \eqref{equL3}.
\end{proof}

\subsection{Zeroth to third order estimates for $\mathbf{U}$}\label{2.2}
The zeroth, first, second and third order estimates can be summarized in the following
\begin{proposition}\label{P1}
Let $(N_{i},N_{e},\mathbf{U})$ be a solution to \eqref{rem} and $\alpha, \beta, k$ be integer for $k=0,1,2,3$, we have
\begin{align}\label{equP1}
\frac{1}2&{\frac{d}{dt}}\sum_{\alpha+\beta=k}\epsilon^{\alpha+2\beta}\big(\|\partial_{x_{1}}^{\alpha}\partial_{x_{2}}^{\beta}U_{1}\|^{2}
+\|\partial_{x_{1}}^{\alpha}\partial_{x_{2}}^{\beta}U_{2}\|^{2}\big)\nonumber\\
&+\frac{1}2{\frac{d}{dt}}\int\frac{n_{e}}{n_{i}}\sum_{\alpha+\beta=k}\epsilon^{\alpha+2\beta}
(\partial_{x_{1}}^{\alpha}\partial_{x_{2}}^{\beta}N_{e})^{2}
+\frac{1}2{\frac{d}{dt}}\int(\frac{n_{e}^{2}}{n_{i}}+\frac{H^{2}}{4}\frac{1}{n_{e}n_{i}})
\sum_{\alpha+\beta=k+1}\epsilon^{\alpha+2\beta}(\partial_{x_{1}}^{\alpha}\partial_{x_{2}}^{\beta}N_{e})^{2}\nonumber\\
&+\frac{1}2\frac{H^{2}}{4}{\frac{d}{dt}}\int\frac{1}{n_{i}}\sum_{\alpha+\beta=k+2}\epsilon^{\alpha+2\beta}
(\partial_{x_{1}}^{\alpha}\partial_{x_{2}}^{\beta}N_{e})^{2}
+\frac{1}2\frac{H^{4}}{16}{\frac{d}{dt}}\int\frac{1}{n_{e}^{2}n_{i}}
\sum_{\alpha+\beta=k+3}\epsilon^{\alpha+2\beta}(\partial_{x_{1}}^{\alpha}\partial_{x_{2}}^{\beta}N_{e})^{2}\nonumber\\
\leq &C(1+\epsilon^{2}|\!|\!|(N_{e},\mathbf{U})|\!|\!|_{\epsilon}^{6})(1+|\!|\!|(N_{e},\mathbf{U})|\!|\!|_{\epsilon}^{2}).
\end{align}
\end{proposition}

The proof of this proposition will be omitted for simplicity, which can be proved by `repeating' the proof of Proposition \ref{P2} below. Indeed, the proof here is slightly easier than that of Proposition \ref{P2}, since the norm already consists of higher order norms such as $\dot H^7$ of $N_e$ and hence the nonlinear terms can be controlled by Sobolev embeddings and other techniques. But note that \eqref{equP1} is not closed, 
therefore we need the higher order estimates in Proposition \ref{P2}, from which we obtain a closed inequality \eqref{Gron} by adding \eqref{equP1} to \eqref{equP2}.

\subsection{Fourth order estimates for $\mathbf{U}$}\label{2.3}
\begin{proposition}\label{P2}
 Let $(N_{i},N_{e},\mathbf{U})$ be a solution to \eqref{rem}, then
 \begin{equation}\label{equP2}
\begin{split}
\frac{1}2&{\frac{d}{dt}}\sum_{\alpha+\beta=4}\epsilon^{\alpha+2\beta}(\|\partial_{x_{1}}^{\alpha}\partial_{x_{2}}^{\beta}U_{1}\|^{2}
+\|\partial_{x_{1}}^{\alpha}\partial_{x_{2}}^{\beta}U_{2}\|^{2})\\
 &+\frac{1}2{\frac{d}{dt}}\int\frac{n_{e}}{n_{i}}\sum_{\alpha+\beta=4}\epsilon^{\alpha+2\beta}(\partial_{x_{1}}^{\alpha}\partial_{x_{2}}^{\beta}N_{e})^{2}
 +\frac{1}2{\frac{d}{dt}}\int(\frac{n_{e}^{2}}{n_{i}}+\frac{H^{2}}{4}\frac{1}{n_{e}n_{i}})\sum_{\alpha+\beta=5}\epsilon^{\alpha+2\beta}(\partial_{x_{1}}^{\alpha}\partial_{x_{2}}^{\beta}N_{e})^{2}\\
 &+\frac{1}2\frac{H^{2}}{4}{\frac{d}{dt}}\int\frac{1}{n_{i}}\sum_{\alpha+\beta=6}\epsilon^{\alpha+2\beta}(\partial_{x_{1}}^{\alpha}\partial_{x_{2}}^{\beta}N_{e})^{2}
 +\frac{1}2\frac{H^{4}}{16}{\frac{d}{dt}}\int\frac{1}{n_{e}^{2}n_{i}}\sum_{\alpha+\beta=7}\epsilon^{\alpha+2\beta}(\partial_{x_{1}}^{\alpha}\partial_{x_{2}}^{\beta}N_{e})^{2}\\
 \leq &C(1+\epsilon^{2}|\!|\!|(N_{e},\mathbf{U})|\!|\!|_{\epsilon}^{6})(1+|\!|\!|(N_{e},\mathbf{U})|\!|\!|_{\epsilon}^{2}).
 \end{split}
\end{equation}
\end{proposition}
\begin{proof}
The proof consists of the results of the following Lemmas \ref{P3}-\ref{P7} which are all about the estimates of the fourth order derivatives for $U$. In this subsection, we only prove Lemma \ref{P3} and leave the others to the next subsection.
\end{proof}
\begin{lemma}\label{P3}
 Let $(N_{i},N_{e},\mathbf{U})$ be a solution to \eqref{rem}. Then
\begin{equation}\label{equLem1}
\begin{split}
\frac{\epsilon^{4}}{2}\frac{d}{dt}&(\|\partial_{x_{1}}^{4}U_{1}\|^{2}+\|\partial_{x_{1}}^{4}U_{2}\|^{2})
+\frac{\epsilon^{4}}{2}\frac{d}{dt}\int\frac{n_{e}}{n_{i}}(\partial_{x_{1}}^{4}N_{e})^{2}\\
&+\frac{\epsilon^{5}}{2}\frac{d}{dt}\int(\frac{n_{e}^{2}}{n_{i}}+\frac{{H^{2}}}{4}\frac{1}{n_{e}n_{i}})
\left((\partial_{x_{1}}^{5}N_{e})^{2}+\epsilon(\partial_{x_{1}}^{4}\partial_{x_{2}}N_{e})^{2}\right)\\
&+\frac{\epsilon^{6}}{2}\frac{H^{2}}{4}\int\frac{1}{n_{i}}\left((2\partial_{x_{1}}^{6}N_{e})^{2}+3\epsilon(\partial_{x_{1}}^{5}\partial_{x_{2}}N_{e})^{2}
+2\epsilon^{2}(\partial_{x_{1}}^{4}\partial_{x_{2}}^{2}N_{e})^{2}\right)\\
&+\frac{\epsilon^{7}}{2}\frac{H^{4}}{16}\int\frac{1}{n_{i}}\left((\partial_{x_{1}}^{7}N_{e})^{2}+3\epsilon(\partial_{x_{1}}^{6}\partial_{x_{2}}N_{e})^{2}
+3\epsilon^{2}(\partial_{x_{1}}^{5}\partial_{x_{2}}^{2}N_{e})^{2}+\epsilon^{3}(\partial_{x_{1}}^{4}\partial_{x_{2}}^{3}N_{e})^{2}\right)\\
\leq & C_{1}(1+\epsilon^{2}|\!|\!|(N_{e},\mathbf{U})|\!|\!|_{\epsilon}^{6})(1+|\!|\!|(N_{e},\mathbf{U})|\!|\!|_{\epsilon}^{2}).
\end{split}
\end{equation}
\end{lemma}
\begin{proof}[Proof of Lemma \ref{P3}.] The proof of Lemma \ref{P3} is divided into three steps. For simplicity, the estimates of some crucial terms which appear in step 1 are postponed to step 2 and step 3.

Step 1.
We take $\partial_{x_{1}}^{4}$ of \eqref{rem-2} and \eqref{rem-3} respectively, then take inner product of
$\epsilon^{4}\partial_{x_{1}}^{4}U_{1}, \ \epsilon^{4}\partial_{x_{1}}^{4}U_{2}$ and sum the results. By integration by parts and using commutator notation \eqref{commu}, we obtain
\begin{align}
\frac{\epsilon^{4}}{2}&\frac{d}{dt}\big(\|\partial_{x_{1}}^{4}U_{1}\|^{2} +\|\partial_{x_{1}}^{4}U_{2}\|^{2}\big)\nonumber\\
=&-\int\big(\epsilon^{3}n_{e}\partial_{x_{1}}^{5}N_{e} -\frac{\epsilon^{4}H^{2}}{4}\frac{\partial_{x_{1}}^{7}N_{e}
+\epsilon\partial_{x_{1}}^{5}\partial_{x_{2}}^{2}N_{e}}{n_{e}}\big)
\big(\partial_{x_{1}}^{4}U_{1}+\epsilon^{1/2}\partial_{x_{1}}^{3} \partial_{x_{2}}U_{2}\big)\nonumber\\
&+\epsilon^{3}\int\left(\partial_{x_{1}}^{4}((V-u_{i1})\partial_{x_{1}}U_{1})
-\epsilon^{1/2}\partial_{x_{1}}^{4}(u_{i2}\partial_{x_{2}}U_{1})\right) \partial_{x_{1}}^{4}U_{1} -\epsilon^{3}\int[\partial_{x_{1}}^{4},n_{e}] \partial_{x_{1}}N_{e}\partial_{x_{1}}^{4}U_{1}\nonumber\\
&+\epsilon^{3}\int\left(\partial_{x_{1}}^{4}((V-u_{i1})\partial_{x_{1}}U_{2})
-\epsilon^{1/2}\partial_{x_{1}}^{4}(u_{i2}\partial_{x_{2}}U_{2})\right) \partial_{x_{1}}^{4}U_{2}
-\epsilon^{3}\epsilon^{1/2}\int\partial_{x_{1}}n_{e} \partial_{x_{1}}^{4}N_{e}\partial_{x_{1}}^{3}\partial_{x_{2}}U_{2}\nonumber\\
&+\epsilon^{3}\epsilon^{1/2}\int\partial_{x_{2}}n_{e} \partial_{x_{1}}^{4}N_{e}\partial_{x_{1}}^{4}U_{2}
+\frac{\epsilon^{4}H^{2}}{4}\int[\partial_{x_{1}}^{4}, \frac{1}{n_{e}}](\partial_{x_{1}}^{3}N_{e}
+\epsilon\partial_{x_{1}}\partial_{x_{2}}^{2}N_{e}) \partial_{x_{1}}^{4}U_{1}\nonumber\\
&+\frac{\epsilon^{4}\epsilon^{1/2}H^{2}}{4} \int[\partial_{x_{1}}^{4},\frac{1}{n_{e}}] (\partial_{x_{1}}^{2}\partial_{x_{2}}N_{e} +\epsilon\partial_{x_{2}}^{3}N_{e})\partial_{x_{1}}^{4}U_{2} -\epsilon^{3}\epsilon^{1/2}\int[\partial_{x_{1}}^{4},n_{e}] \partial_{x_{2}}N_{e}\partial_{x_{1}}^{4}U_{2}\nonumber\\
&+\frac{\epsilon^{4}\epsilon^{1/2}H^{2}}{4}\int\partial_{x_{2}} \frac{1}{n_{e}}\big(\partial_{x_{1}}^{7}N_{e}
+\epsilon\partial_{x_{1}}^{5}\partial_{x_{2}}^{2}N_{e}\big) \partial_{x_{1}}^{3}U_{2}\nonumber\\
&-\frac{\epsilon^{4}\epsilon^{1/2}H^{2}}{4}\int\partial_{x_{1}} \frac{1}{n_{e}}\big(\partial_{x_{1}}^{6}\partial_{x_{2}}N_{e}
+\epsilon\partial_{x_{1}}^{4}\partial_{x_{2}}^{3}N_{e}\big) \partial_{x_{1}}^{3}U_{2}\nonumber\\
&-\epsilon^{4}\int\partial_{x_{1}}^{4}\big(\partial_{x_{1}} \tilde{u_{i1}}U_{1}
+\epsilon^{1/2}\partial_{x_{2}}\tilde{u_{i1}}U_{2}\big) \partial_{x_{1}}^{4}U_{1}
-\epsilon^{4}\int\partial_{x_{1}}^{4}\big(\partial_{x_{1}}\tilde{u_{i2}}U_{1}
+\epsilon^{1/2}\partial_{x_{2}}\tilde{u_{i2}}U_{2}\big) \partial_{x_{1}}^{4}U_{2}\nonumber\\
&-\epsilon^{5}\int\big(\partial_{x_{1}}^{4}\mathcal{R}_{2}^{(1)} \partial_{x_{1}}^{4}U_{1}
+\epsilon^{5}\epsilon^{1/2}\partial_{x_{1}}^{4}\mathcal{R}_{3}^{(1)} \partial_{x_{1}}^{4}U_{2}\big)
-\epsilon^{4}\int\partial_{x_{1}}^{4}(\partial_{x_{1}}\tilde{n_{e}}N_{e}) \partial_{x_{1}}^{4}U_{1}\nonumber\\
&-\epsilon^{4}\epsilon^{1/2}\int\partial_{x_{1}}^{4} (\partial_{x_{2}}\tilde{n_{e}}N_{e})\partial_{x_{1}}^{4}U_{2}
+\frac{\epsilon^{4}H^{2}}{4}\int\partial_{x_{1}}^{4} \big(-\frac{A_{1}}{n_{e}^{2}}+\frac{A_{2}}{n_{e}^{3}}
+\frac{\epsilon\mathcal{R}_{2}^{(2)} +\mathcal{R}_{2}^{(3)}}{n_{e}^{3}}\big)\partial_{x_{1}}^{4}U_{1}\nonumber\\
&+\frac{\epsilon^{4}\epsilon^{1/2}H^{2}}{4}\int\partial_{x_{1}}^{4} \big(-\frac{B_{1}}{n_{e}^{2}}+\frac{B_{2}}{n_{e}^{3}}
+\frac{\mathcal{R}_{3}^{(2)} +\epsilon\mathcal{R}_{3}^{(3)}}{n_{e}^{3}}\big)\partial_{x_{1}}^{4}U_{2}\nonumber\\
=&:\sum_{i=1}^{18}F_{i}.\label{equ14}
\end{align}

\emph{Estimate of the RHS of \eqref{equ14}.} First, we estimate the second term on the RHS of \eqref{equ14}. Using commutator notation \eqref{commu} to rewrite it as
\begin{equation*}
\begin{split}
F_{2}=&\epsilon^{3}\int[\partial_{x_{1}}^{4},V-u_{i1}]\partial_{x_{1}}U_{1}\partial_{x_{1}}^{4}U_{1}
-\epsilon^{3}\epsilon^{1/2}\int[\partial_{x_{1}}^{4},u_{i2}]\partial_{x_{2}}U_{1}\partial_{x_{1}}^{4}U_{1}\\
&+\epsilon^{3}\int(V-u_{i1})\partial_{x_{1}}^{5}U_{1}\partial_{x_{1}}^{4}U_{1}
-\epsilon^{3}\epsilon^{1/2}\int u_{i2}\partial_{x_{1}}^{4}\partial_{x_{2}}U_{1}\partial_{x_{1}}^{4}U_{1}\\
=&:\sum_{i=1}^{4}F_{2i}.
\end{split}
\end{equation*}
We first estimate $F_{21}$.  By commutator estimate of Lemma \ref{Lem2}, we have
\begin{equation*}
\begin{split}
\|[\partial_{x_{1}}^{4},V-u_{i1}]\partial_{x_{1}}U_{1}\|
\leq\|\partial_{x_{1}}(V-u_{i1})\|_{L^{\infty}}\|\partial_{x_{1}}^{4}U_{1}\|
+\|\partial_{x_{1}}^{4}(V-u_{i1})\|\|\partial_{x_{1}}U_{1}\|_{L^{\infty}}.
\end{split}
\end{equation*}
This yields that
\begin{equation*}
\begin{split}
F_{21}\leq &\epsilon^{3}\|[\partial_{x_{1}}^{4},V-u_{i1}]\partial_{x_{1}}U_{1}\|\|\partial_{x_{1}}^{4}U_{1}\|\\
\leq &C(1+\epsilon^{7}\|\partial_{x_{1}}U_{1}\|_{L^{\infty}}^{2})(\epsilon^{4}\|\partial_{x_{1}}^{4}U_{1}\|^{2}+\epsilon^{5}\|\partial_{x_{1}}U_{1}\|_{L^{\infty}}^{2})\\
\leq &C(1+\epsilon^{7}\|\partial_{x_{1}}U_{1}\|_{H^{2}}^{2})(\epsilon^{4}\|\partial_{x_{1}}^{4}U_{1}\|^{2}+\epsilon^{5}\|\partial_{x_{1}}U_{1}\|_{H^{2}}^{2})\\
\leq &C(1+\epsilon^{2}|\!|\!|(N_{e},U_{1})|\!|\!|_{\epsilon}^{2})|\!|\!|(N_{e},U_{1})|\!|\!|_{\epsilon}^{2},
\end{split}
\end{equation*}
where $|\!|\!|(N_{e},U_{1})|\!|\!|_{\epsilon}^{2}$ is given in \eqref{|||}. Similarly, we obtain
\begin{equation}
\begin{split}\label{G2}
F_{22}\leq &C(1+\epsilon^{8}\|\partial_{x_{2}}U_{1}\|_{L^{\infty}}^{2})(1+\epsilon^{4}\|\partial_{x_{1}}^{4}U_{1}\|^{2}
+\epsilon^{5}\|\partial_{x_{1}}^{3}\partial_{x_{2}}U_{1}\|^{2})\\
\leq &C(1+\epsilon^{2}|\!|\!|(N_{e},U_{1})|\!|\!|_{\epsilon}^{2})|\!|\!|(N_{e},U_{1})|\!|\!|_{\epsilon}^{2}.
\end{split}
\end{equation}
Next, we estimate $F_{23}$. By integration by parts,
\begin{equation*}
\begin{split}
F_{23}=&-\frac{\epsilon^{3}}{2}\int\partial_{x_{1}}(V-u_{i1})(\partial_{x_{1}}^{4}U_{1})^{2}\\
\leq& C(1+\epsilon^{7}\|\partial_{x_{1}}U_{1}\|_{L^{\infty}}^{2})(\epsilon^{4}\|\partial_{x_{1}}^{4}U_{1}\|^{2})\\
\leq &C(1+\epsilon^{2}|\!|\!|(N_{e},\mathbf{U})|\!|\!|_{\epsilon}^{2})|\!|\!|(N_{e},\mathbf{U})|\!|\!|_{\epsilon}^{2},
\end{split}
\end{equation*}
where the Sobolev embedding theorem $H^2\hookrightarrow L^{\infty}$ is used. Similarly, we have
\begin{equation*}
\begin{split}
F_{24}\leq C(1+\epsilon^{2}|\!|\!|(N_{e},\mathbf{U})|\!|\!|_{\epsilon}^{2})|\!|\!|(N_{e},\mathbf{U})|\!|\!|_{\epsilon}^{2}.
\end{split}
\end{equation*}
Thus we have
\begin{equation*}
\begin{split}
F_{2}\leq C(1+\epsilon^{2}|\!|\!|(N_{e},\mathbf{U})|\!|\!|_{\epsilon}^{2})|\!|\!|(N_{e},\mathbf{U})|\!|\!|_{\epsilon}^{2}.
\end{split}
\end{equation*}
Similarly, we have
\begin{equation*}
\begin{split}
F_{3\sim6}\leq C(1+\epsilon^{2}|\!|\!|(N_{e},\mathbf{U})|\!|\!|_{\epsilon}^{2})|\!|\!|(N_{e},\mathbf{U})|\!|\!|_{\epsilon}^{2}.
\end{split}
\end{equation*}
By applying \eqref{one order}, \eqref{two orders} and Lemma \ref{Lem1}, we also obtain
\begin{equation*}
\begin{split}
F_{7\sim18}\leq C(1+\epsilon^{2}|\!|\!|(N_{e},\mathbf{U})|\!|\!|_{\epsilon}^{6})|\!|\!|(N_{e},\mathbf{U})|\!|\!|_{\epsilon}^{2}.
\end{split}
\end{equation*}

\emph{Estimate of the $F_{1}$.} We take $\partial_{x_{1}}^{3}$ of \eqref{rem-1} and applying commutator notation, we obtain
\begin{equation}
\begin{split}\label{z1}
\partial_{x_{1}}^{4}&U_{1}+\epsilon^{1/2}\partial_{x_{1}}^{3}\partial_{x_{2}}U_{2}\\
=&\frac{1}{n_{i}}\Big(-\epsilon\partial_{t}\partial_{x_{1}}^{3}N_{i}+\partial_{x_{1}}^{3}((V-u_{i1})\partial_{x_{1}}N_{i})
-\epsilon^{1/2}\partial_{x_{1}}^{3}(u_{i2}\partial_{x_{2}}N_{i})-[\partial_{x_{1}}^{3},n_{i}]\partial_{x_{1}}U_{1}\\
&-\epsilon^{1/2}[\partial_{x_{1}}^{3},n_{i}]\partial_{x_{2}}U_{2}
-\epsilon\partial_{x_{1}}^{3}(\partial_{x_{1}}\tilde{u_{i1}}N_{i}+\epsilon^{1/2}\partial_{x_{2}}\tilde{u_{i2}}N_{i})
-\epsilon\partial_{x_{1}}^{3}(\partial_{x_{1}}\tilde{n_{i}}U_{1}\\
&+\epsilon^{1/2}\partial_{x_{2}}\tilde{n_{i}}U_{2})
-\epsilon^{2}\partial_{x_{1}}^{3}\mathcal R_{1}\Big)\\
=&:\sum_{i=1}^{8}G_{i}.
\end{split}
\end{equation}
Using \eqref{z1}, we have
 \begin{equation}
\begin{split}\label{z3}
F_{1}=&-\int\left(\epsilon^{3}n_{e}\partial_{x_{1}}^{5}N_{e}-\frac{\epsilon^{4}H^{2}}{4}\frac{\partial_{x_{1}}^{7}N_{e}
+\epsilon\partial_{x_{1}}^{5}\partial_{x_{2}}^{2}N_{e}}{n_{e}}\right)
\big(\partial_{x_{1}}^{4}U_{1}+\epsilon^{1/2}\partial_{x_{1}}^{3}\partial_{x_{2}}U_{2}\big)\\
=&-\int\left(\epsilon^{3}n_{e}\partial_{x_{1}}^{5}N_{e}-\frac{\epsilon^{4}H^{2}}{4}\frac{\partial_{x_{1}}^{7}N_{e}
+\epsilon\partial_{x_{1}}^{5}\partial_{x_{2}}^{2}N_{e}}{n_{e}}\right)\sum_{i=1}^{8}G_{i}\\
=&:\sum_{i=1}^{8}I_{i}.
\end{split}
\end{equation}
We first estimate the terms $I_{i}$ for $4\leq i\leq8$ and leave $I_{i}$ for $1\leq i\leq3$ in the next two steps. For $I_4$, we have
\begin{equation*}
\begin{split}
I_{4}=&\int\left(\epsilon^{3}n_{e}\partial_{x_{1}}^{5}N_{e}-\frac{\epsilon^{4}H^{2}}{4}\frac{\partial_{x_{1}}^{7}N_{e}
+\epsilon\partial_{x_{1}}^{5}\partial_{x_{2}}^{2}N_{e}}{n_{e}}\right) [\partial_{x_{1}}^{3},n_{i}]\partial_{x_{1}}U_{1}\\
\leq &C\big(1+\epsilon^{7}(\|\partial_{x_{1}}U_{1}\|_{L^{\infty}}^{2}+\|\partial_{x_{1}}n_{i}\|_{L^{\infty}}^{2})\big)\\
&\times\big(1+\epsilon^{5}\|\partial_{x_{1}}^{5}N_{e}\|^{2}+\epsilon^{7}\|\partial_{x_{1}}^{7}\partial_{x_{2}}^{2}N_{e}\|^{2}
+\epsilon^{9}\|\partial_{x_{1}}^{5}\partial_{x_{2}}^{2}N_{e}\|^{2}+\epsilon^{3}\|\partial_{x_{1}}^{3}N_{i}\|^{2}\big)\\
\leq &C_{1}\left(1+\epsilon^{2}|\!|\!|(N_{e},U_{1})|\!|\!|_{\epsilon}^{2}\right)\left(1+|\!|\!|(N_{e},U_{1})|\!|\!|_{\epsilon}^{2}\right),
\end{split}
\end{equation*}
thanks to the Sobolev embedding theorem and commutator estimates in Lemma \ref{Lem2}. Similarly, we have
\begin{equation*}
\begin{split}
I_{5\sim8}\leq C_{1}(1+\epsilon^{2}|\!|\!|(N_{e},\mathbf{U})|\!|\!|_{\epsilon}^{2})(1+|\!|\!|(N_{e},\mathbf{U})|\!|\!|_{\epsilon}^{2}).
\end{split}
\end{equation*}

Step 2. \emph{Estimate of $I_{2}+I_{3}$.}
The $I_{2}$ of \eqref{z3} can be divided into
\begin{align*}
I_{2}=&-\epsilon^{3}\int\left(\frac{n_{e}}{n_{i}}\partial_{x_{1}}^{5}N_{e}-\frac{\epsilon H^{2}}{4}\frac{\partial_{x_{1}}^{7}N_{e}
+\epsilon\partial_{x_{1}}^{5}\partial_{x_{2}}^{2}N_{e}}{n_{e}n_{i}}\right)\partial_{x_{1}}^{3}((V-u_{i1})\partial_{x_{1}}N_{i})\\
=&-\epsilon^{3}\int\left(\frac{ n_{e}(V-u_{i1})}{n_{i}}\partial_{x_{1}}^{5}N_{e}-\frac{\epsilon H^{2}}{4}\frac{(V-u_{i1})}{n_{e}n_{i}}(\partial_{x_{1}}^{7}N_{e}
+\epsilon\partial_{x_{1}}^{5}\partial_{x_{2}}^{2}N_{e})\right)\partial_{x_{1}}^{4}N_{i}\\
&-\epsilon^{3}\int\left(\frac{ n_{e}}{n_{i}}\partial_{x_{1}}^{5}N_{e}-\frac{\epsilon H^{2}}{4}\frac{\partial_{x_{1}}^{7}N_{e}
+\epsilon\partial_{x_{1}}^{5}\partial_{x_{2}}^{2}N_{e}}{n_{e}n_{i}}\right) [\partial_{x_{1}}^{3},V-u_{i1}]\partial_{x_{1}}N_{i}\\
=&:I_{21}+I_{22}.
\end{align*}
The estimate of $I_{22}$ is given by
\begin{equation*}
\begin{split}
I_{22}\leq C_{1}(1+\epsilon^{2}|\!|\!|(N_{e},\mathbf{U})|\!|\!|_{\epsilon}^{2})|\!|\!|(N_{e},\mathbf{U})|\!|\!|_{\epsilon}^{2}.
\end{split}
\end{equation*}
Next we estimate $I_{21}$. For this we recall from \eqref{rem-4} that
\begin{equation*}
\begin{split}
\partial_{x_{1}}^{4}N_{i}=&\partial_{x_{1}}^{4}N_{e}-\epsilon\partial_{x_{1}}^{4}(n_{e}\partial_{x_{1}}^{2}N_{e})
-\epsilon^{2}\partial_{x_{1}}^{4}(n_{e}\partial_{x_{2}}^{2}N_{e})
+\frac{\epsilon^{2}H^{2}}{4}\partial_{x_{1}}^{4}\big(\frac{\partial_{x_{1}}^{4}N_{e}}{n_{e}}\big)
+\frac{\epsilon^{3}H^{2}}{2}\partial_{x_{1}}^{4}\big(\frac{\partial_{x_{1}}^{2}\partial_{x_{2}}^{2}N_{e}}{n_{e}}\big)\\
&+\frac{\epsilon^{4}H^{2}}{4}\partial_{x_{1}}^{4}\big(\frac{\partial_{x_{2}}^{4}N_{e}}{n_{e}}\big)
-\epsilon^{2}\partial_{x_{1}}^{4}\big(\partial_{x_{1}}\tilde{n_{e}}\partial_{x_{1}}N_{e}+\epsilon\partial_{x_{2}}\tilde{n_{e}}\partial_{x_{2}}N_{e}\big)\\
&-\epsilon^{2}\partial_{x_{1}}^{4}\big(\epsilon^{4}(\partial_{x_{1}}N_{e})^{2}+\epsilon^{5}(\partial_{x_{2}}N_{e})^{2}\big)
-\epsilon^{2}\partial_{x_{1}}^{4}\big(\partial_{x_{1}}^{2}\tilde{n_{e}}N_{e}+\epsilon\partial_{x_{2}}^{2}\tilde{n_{e}}N_{e}\big)
-\epsilon^{2}\partial_{x_{1}}^{4}\mathcal{R}_{4}^{(1)}\\
&-\frac{\epsilon^{2}H^{2}}{4}\partial_{x_{1}}^{4}\big(\frac{C_{1}}{n_{e}^{2}}-\frac{C_{2}}{n_{e}^{3}}+\frac{C_{3}}{n_{e}^{4}}
+\frac{\epsilon\mathcal{R}_{4}^{(2)}+\epsilon\mathcal{R}_{4}^{(3)}}{n_{e}^{4}}\big)\\
=&:\sum_{i=1}^{11}H_{i}.
\end{split}
\end{equation*}
Thus we have
\begin{equation*}
\begin{split}
I_{21}&=-\epsilon^{3}\int\left(\frac{ n_{e}(V-u_{i1})}{n_{i}}\partial_{x_{1}}^{5}N_{e}-\frac{\epsilon H^{2}}{4}\frac{(V-u_{i1})}{n_{e}n_{i}}(\partial_{x_{1}}^{7}N_{e}
+\epsilon\partial_{x_{1}}^{5}\partial_{x_{2}}^{2}N_{e})\right)\sum_{i=1}^{11}H_{i}\\
&=:\sum_{i=1}^{11}I_{21i}.
\end{split}
\end{equation*}
By integration by parts and commutator notation, we have
\begin{equation*}
\begin{split}
I_{211}=&\frac{\epsilon^{3}}{2}\int\partial_{x_{1}}\big(\frac{n_{e}(V-u_{i1})}{n_{i}}\big)(\partial_{x_{1}}^{4}N_{e})^{2}
+\frac{\epsilon^{4}H^{2}}{8}\int\partial_{x_{1}}\big(\frac{V-u_{i1}}{n_{e}n_{i}}\big)(\partial_{x_{1}}^{5}N_{e})^{2}\\
&+\frac{\epsilon^{5}H^{2}}{8}\int\partial_{x_{1}}\big(\frac{V-u_{i1}}{n_{e}n_{i}}\big)(\partial_{x_{1}}^{4}\partial_{x_{2}}N_{e})^{2}
-\frac{\epsilon^{4}H^{2}}{4}\int\partial_{x_{1}}\big(\frac{V-u_{i1}}{n_{e}n_{i}}\big)\partial_{x_{1}}^{6}N_{e}\partial_{x_{1}}^{4}N_{e}\\
&-\frac{\epsilon^{5}H^{2}}{4}\int\partial_{x_{2}}\big(\frac{V-u_{i1}}{n_{e}n_{i}}\big)\partial_{x_{1}}^{5}\partial_{x_{2}}N_{e}\partial_{x_{1}}^{4}N_{e}\\
=&:\sum_{i=1}^{5}I_{211}^{(i)}.
\end{split}
\end{equation*}
By computation, we have
\begin{equation*}
\begin{split}
\left|\partial_{x_{1}}\left(\frac{n_{e}(V-u_{i1})}{n_{i}}\right)\right| \leq C\left(\epsilon+\epsilon^{5}(|\partial_{x_{1}}N_{e}| +|\partial_{x_{1}}N_{i}|+|\partial_{x_{1}}U_{1}|)\right),
\end{split}
\end{equation*}
yielding the estimates
\begin{equation*}
\begin{split}
I_{211}^{(1)}
&\leq C\big(1+\epsilon^{7}(\|\partial_{x_{1}}N_{e}\|_{L^{\infty}}^{2}+\|\partial_{x_{1}}U_{1}\|_{L^{\infty}}^{2}
+\|\partial_{x_{1}}N_{i}\|_{L^{\infty}}^{2})\big)
(\epsilon^{4}\|\partial_{x_{1}}^{4}N_{e}\|^{2})\\
&\leq C_{1}\big(1+\epsilon^{2}|\!|\!|(N_{e},\mathbf{U})|\!|\!|_{\epsilon}^{2}\big)|\!|\!|(N_{e},\mathbf{U})|\!|\!|_{\epsilon}^{2}.
\end{split}
\end{equation*}
The other terms in $I_{211}$ can be bounded similarly by
\begin{equation*}
\begin{split}
I_{211}^{(2\sim5)}
\leq C_{1}(1+\epsilon^{2}|\!|\!|(N_{e}, \mathbf{U})|\!|\!|_{\epsilon}^{2})|\!|\!|(N_{e}, \mathbf{U})|\!|\!|_{\epsilon}^{2},
\end{split}
\end{equation*}
which yields
\begin{equation*}
\begin{split}
I_{211}\leq C_{1}(1+\epsilon^{2}|\!|\!|(N_{e},\mathbf{U})|\!|\!|_{\epsilon}^{2})|\!|\!|(N_{e},\mathbf{U})|\!|\!|_{\epsilon}^{2}.
\end{split}
\end{equation*}
By integration by parts and commutator notation again, we have
\begin{equation*}
\begin{split}
I_{212}=&-\frac{\epsilon^{4}}{2}\int\partial_{x_{1}}\left(\frac{n_{e}^{2}(V-u_{i1})}{n_{i}}\right)(\partial_{x_{1}}^{5}N_{e})^{2}
+\frac{\epsilon^{5}H^{2}}{8}\int\partial_{x_{1}}\big(\frac{V-u_{i1}}{n_{i}}\big)(\partial_{x_{1}}^{6}N_{e})^{2}\\
&+\frac{\epsilon^{6}H^{2}}{8}\int\partial_{x_{1}}\big(\frac{V-u_{i1}}{n_{i}}\big)(\partial_{x_{1}}^{5}\partial_{x_{2}}N_{e})^{2}
-\frac{\epsilon^{6}H^{2}}{4}\int\partial_{x_{1}}\big(\frac{V-u_{i1}}{n_{e}n_{i}}\big)\partial_{x_{1}}^{6}N_{e}\partial_{x_{1}}^{4}N_{e}\\
&+\epsilon^{4}\int\left(\frac{n_{e}^{2}(V-u_{i1})}{n_{i}}\partial_{x_{1}}^{5}N_{e}
-\frac{\epsilon H^{2}}{4}\frac{V-u_{i1}}{n_{i}}(\partial_{x_{1}}^{7}N_{e}
+\epsilon\partial_{x_{1}}^{5}\partial_{x_{2}}^{2}N_{e})\right) [\partial_{x_{1}}^{4},n_{e}]\partial_{x_{1}}^{2}N_{e}\\
&-\frac{\epsilon^{6}H^{2}}{4}\int\partial_{x_{2}}\big(\frac{V-u_{i1}}{n_{i}}\big)\partial_{x_{1}}^{5}\partial_{x_{2}}N_{e}\partial_{x_{1}}^{6}N_{e}.
\end{split}
\end{equation*}
Similar to $I_{211}$, using commutator estimate, we have
\begin{equation*}
\begin{split}
I_{212}\leq C_{1}(1+\epsilon^{2}|\!|\!|(N_{e},\mathbf{U})|\!|\!|_{\epsilon}^{2})|\!|\!|(N_{e},\mathbf{U})|\!|\!|_{\epsilon}^{2}.
\end{split}
\end{equation*}
Similarly, we have
\begin{equation*}
\begin{split}
I_{213}\leq C_{1}(1+\epsilon^{2}|\!|\!|(N_{e},\mathbf{U})|\!|\!|_{\epsilon}^{2})|\!|\!|(N_{e},\mathbf{U})|\!|\!|_{\epsilon}^{2}.
\end{split}
\end{equation*}
By integration by parts and commutator notation, we have
\begin{equation*}
\begin{split}
I_{214}=&-\frac{\epsilon^{5}H^{2}}{8}\int\partial_{x_{1}}\big(\frac{V-u_{i1}}{n_{i}}\big)(\partial_{x_{1}}^{6}N_{e})^{2}
-\frac{1}{2}\frac{\epsilon^{6}H^{4}}{16}\int\partial_{x_{1}}\big(\frac{V-u_{i1}}{n_{e}^{2}n_{i}}\big)(\partial_{x_{1}}^{7}N_{e})^{2}\\
&+\frac{1}{2}\frac{\epsilon^{7}H^{2}}{16}\int\partial_{x_{1}}\big(\frac{V-u_{i1}}{n_{e}^{2}n_{i}}\big)(\partial_{x_{1}}^{6}\partial_{x_{2}}N_{e})^{2}
+\frac{\epsilon^{5}H^{2}}{4}\int\partial_{x_{1}}\big(\frac{V-u_{i1}}{n_{i}}\big)\partial_{x_{1}}^{5}N_{e}\partial_{x_{1}}^{7}N_{e}\\
&+\frac{\epsilon^{7}H^{4}}{16}\int\partial_{x_{1}}\big(\frac{V-u_{i1}}{n_{e}^{2}n_{i}}\big)\partial_{x_{1}}^{5}\partial_{x_{2}}^{2}N_{e}\partial_{x_{1}}^{7}N_{e}
-\frac{\epsilon^{7}H^{4}}{16}\int\partial_{x_{2}}\big(\frac{V-u_{i1}}{n_{e}^{2}n_{i}}\big)\partial_{x_{1}}^{6}\partial_{x_{2}}N_{e}\partial_{x_{1}}^{7}N_{e}\\
&-\frac{\epsilon^{5}H^{4}}{4}\int\left(\frac{n_{e}(V-u_{i1})}{n_{i}}\partial_{x_{1}}^{5}N_{e}
-\frac{\epsilon H^{2}}{4}\frac{V-u_{i1}}{n_{e}n_{i}}(\partial_{x_{1}}^{7}N_{e}
+\epsilon\partial_{x_{1}}^{5}\partial_{x_{2}}^{2}N_{e})\right) [\partial_{x_{1}}^{4},\frac{1}{n_{e}}]\partial_{x_{1}}^{4}N_{e}.
\end{split}
\end{equation*}
Similar to $I_{211}$, using commutator estimate, we have
\begin{equation*}
\begin{split}
I_{214}\leq C_{1}(1+\epsilon^{2}|\!|\!|(N_{e},\mathbf{U})|\!|\!|_{\epsilon}^{4})|\!|\!|(N_{e},\mathbf{U})|\!|\!|_{\epsilon}^{2}.
\end{split}
\end{equation*}
Similarly, we have
\begin{equation*}
\begin{split}
I_{215\sim2111}\leq C_{1}(1+\epsilon^{2}|\!|\!|(N_{e},\mathbf{U})|\!|\!|_{\epsilon}^{2})|\!|\!|(N_{e},\mathbf{U})|\!|\!|_{\epsilon}^{2}.
\end{split}
\end{equation*}
Thus we have
\begin{equation*}
\begin{split}
I_{2}\leq C_{1}(1+\epsilon^{2}|\!|\!|(N_{e},\mathbf{U})|\!|\!|_{\epsilon}^{4})|\!|\!|(N_{e},\mathbf{U})|\!|\!|_{\epsilon}^{2}.
\end{split}
\end{equation*}
Similarly, we have
\begin{equation*}
\begin{split}
I_{3}\leq C_{1}(1+\epsilon^{2}|\!|\!|(N_{e},\mathbf{U})|\!|\!|_{\epsilon}^{4})|\!|\!|(N_{e},\mathbf{U})|\!|\!|_{\epsilon}^{2}.
\end{split}
\end{equation*}

Step 3. \emph{Estimate of $I_{1}$.}
Taking $\partial_{x_{1}}^{3}$ with \eqref{rem-4}, we have
\begin{equation*}
\begin{split}
\partial_{t}\partial_{x_{1}}^{3}N_{i}=&\partial_{t} \partial_{x_{1}}^{3}N_{e}-\epsilon\partial_{t} \partial_{x_{1}}^{3}\left(n_{e}\partial_{x_{1}}^{2}N_{e}\right) -\epsilon^{2}\partial_{t}\partial_{x_{1}}^{3}\left(n_{e} \partial_{x_{2}}^{2}N_{e}\right)\\
&+\frac{\epsilon^{2}H^{2}}{4}\partial_{t} \partial_{x_{1}}^{3}\left(\frac{\partial_{x_{1}}^{4}N_{e}}{n_{e}}\right) +\frac{\epsilon^{3}H^{2}}{2}\partial_{t}\partial_{x_{1}}^{3} \left(\frac{\partial_{x_{1}}^{2}\partial_{x_{2}}^{2}N_{e}}{n_{e}}\right) +\frac{\epsilon^{4}H^{2}}{4}\partial_{t}\partial_{x_{1}}^{3} \left(\frac{\partial_{x_{2}}^{4}N_{e}}{n_{e}}\right)\\
&-\epsilon^{2}\partial_{t}\partial_{x_{1}}^{3} \left(\partial_{x_{1}}\tilde{n_{e}}\partial_{x_{1}}N_{e} +\epsilon\partial_{x_{2}}\tilde{n_{e}}\partial_{x_{2}}N_{e}\right) -\epsilon^{2}\partial_{t}\partial_{x_{1}}^{3} \left(\epsilon^{4}\left(\partial_{x_{1}}N_{e}\right)^{2} +\epsilon^{5}\left(\partial_{x_{2}}N_{e}\right)^{2}\right)\\
&-\epsilon^{2}\partial_{t}\partial_{x_{1}}^{3} \left(\partial_{x_{1}}^{2}\tilde{n_{e}}N_{e} +\epsilon\partial_{x_{2}}^{2}\tilde{n_{e}}N_{e}\right) -\epsilon^{2}\partial_{t}\partial_{x_{1}}^{3}\mathcal{R}_{4}^{(1)}\\
&-\frac{\epsilon^{2}H^{2}}{4}\partial_{t}\partial_{x_{1}}^{3} \left(\frac{C_{1}}{n_{e}^{2}}-\frac{C_{2}}{n_{e}^{3}} +\frac{C_{3}}{n_{e}^{4}}+\frac{\epsilon\mathcal{R}_{4}^{(2)} +\epsilon\mathcal{R}_{4}^{(3)}}{n_{e}^{4}}\right)\\
=&:\sum_{i=1}^{11}K_{i}.
\end{split}
\end{equation*}
From \eqref{z3}, we have
\begin{equation}
\begin{split}\label{orin}
I_{1}=&\epsilon^{4}\int\left(\frac{n_{e}}{n_{i}}\partial_{x_{1}}^{5}N_{e}-\frac{\epsilon H^{2}}{4}\frac{\partial_{x_{1}}^{7}N_{e}
+\epsilon\partial_{x_{1}}^{5}\partial_{x_{2}}^{2}N_{e}}{n_{e}n_{i}}\right)\sum_{i=1}^{11}K_{i} =:\sum_{i=1}^{11}I_{1i}.
\end{split}
\end{equation}
For convenience, we denote
\begin{equation*}
\begin{split}
I_{11}=&\epsilon^{4}\int\left(\frac{n_{e}}{n_{i}}\partial_{x_{1}}^{5}N_{e}-\frac{\epsilon H^{2}}{4}\frac{\partial_{x_{1}}^{7}N_{e}
+\epsilon\partial_{x_{1}}^{5}\partial_{x_{2}}^{2}N_{e}}{n_{e}n_{i}}\right)\partial_{t}\partial_{x_{1}}^{3}N_{e}
=:\sum_{i=1}^{3}I_{11i}.
\end{split}
\end{equation*}
By integration by parts, we have
\begin{equation*}
\begin{split}
I_{111}=&-\frac{\epsilon^{4}}{2}\frac{d}{dt}\int\frac{n_{e}}{n_{i}}(\partial_{x_{1}}^{4}N_{e})^{2}
+\frac{\epsilon^{4}}{2}\int\partial_{t}(\frac{n_{e}}{n_{i}})(\partial_{x_{1}}^{4}N_{e})^{2}\\
&-\epsilon^{4}\int\partial_{x_{1}}(\frac{n_{e}}{n_{i}})\partial_{x_{1}}^{4}N_{e}\partial_{t}\partial_{x_{1}}^{3}N_{e}.
\end{split}
\end{equation*}
By Lemma \ref{L2} and \ref{L3}, the second term and the third term can be bounded respectively by
\begin{equation*}
\begin{split}
\frac{\epsilon^{4}}{2}\int\partial_{t}(\frac{n_{e}}{n_{i}}) (\partial_{x_{1}}^{4}N_{e})^{2}
&\leq C_{1}\big(1+\epsilon^{7}(\|\epsilon\partial_{t}N_{e}\|_{L^{\infty}}^{2}+\|\epsilon\partial_{t}N_{i}\|_{L^{\infty}}^{2})\big)
(\epsilon^{4}\|\partial_{x_{1}}^{4}N_{e}\|^{2})\\
&\leq C_{1}\big(1+\epsilon^{7}(\|\epsilon\partial_{t}N_{e}\|_{H^{2}}^{2}+\|\epsilon\partial_{t}N_{i}\|_{H^{2}}^{2})\big)
(\epsilon^{4}\|\partial_{x_{1}}^{4}N_{e}\|^{2})\\
&\leq C_{1}\big(1+\epsilon^{2}|\!|\!|(N_{e},\mathbf{U})|\!|\!|_{\epsilon}^{2}\big)|\!|\!|(N_{e},\mathbf{U})|\!|\!|_{\epsilon}^{2},
\end{split}
\end{equation*}
and
\begin{equation*}
\begin{split}
-\epsilon^{4}&\int\partial_{x_{1}}(\frac{n_{e}}{n_{i}})\partial_{x_{1}}^{4}N_{e}\partial_{t}\partial_{x_{1}}^{3}N_{e}\\
&\leq C_{1}\big(1+\epsilon^{7}(\|\partial_{x_{1}}N_{e}\|_{L^{\infty}}^{2}+\|\partial_{x_{1}}N_{i}\|_{L^{\infty}}^{2})\big)
\big(\epsilon^{4}\|\partial_{x_{1}}^{4}N_{e}\|^{2}+\epsilon^{6}\|\partial_{t}\partial_{x_{1}}^{3}N_{e}\|^{2}\big)\\
&\leq C_{1}(1+\epsilon^{2}|\!|\!|(N_{e},\mathbf{U})|\!|\!|_{\epsilon}^{2})|\!|\!|(N_{e},\mathbf{U})|\!|\!|_{\epsilon}^{2}.
\end{split}
\end{equation*}
Thus, we have
\begin{equation}
\begin{split}\label{a}
I_{111}\leq-\frac{\epsilon^{4}}{2}\frac{d}{dt}\int\frac{n_{e}}{n_{i}} (\partial_{x_{1}}^{4}N_{e})^{2} +C_{1}(1+\epsilon^{2}|\!|\!|(N_{e},\mathbf{U})|\!|\!|_{\epsilon}^{2}) |\!|\!|(N_{e},\mathbf{U})|\!|\!|_{\epsilon}^{2}.
\end{split}
\end{equation}
Similarly, the other two  terms in $I_{11}$ can be bounded by
\begin{equation}
\begin{split}\label{b}
I_{112}=&-\frac{1}{2}\frac{\epsilon^{5}H^{2}}{4}\frac{d}{dt}\int\frac{1}{n_{e}n_{i}}(\partial_{x_{1}}^{5}N_{e})^{2}
+\frac{1}{2}\frac{\epsilon^{5}H^{2}}{4}\int\partial_{t}\big(\frac{1}{n_{e}n_{i}}\big)(\partial_{x_{1}}^{5}N_{e})^{2}\\
&-\frac{\epsilon^{5}H^{2}}{2}\int\partial_{x_{1}}\big(\frac{1}{n_{e}n_{i}}\big)\partial_{x_{1}}^{5}N_{e}\partial_{t}\partial_{x_{1}}^{4}N_{e}
-\frac{\epsilon^{5}H^{2}}{4}\int\partial_{x_{1}}^{2}\big(\frac{1}{n_{e}n_{i}}\big)\partial_{x_{1}}^{5}N_{e}\partial_{t}\partial_{x_{1}}^{3}N_{e}\\
\leq& -\frac{1}{2}\frac{\epsilon^{5}H^{2}}{4}\frac{d}{dt}\int\frac{1}{n_{e}n_{i}}(\partial_{x_{1}}^{5}N_{e})^{2}
+C_{1}(1+\epsilon^{2}|\!|\!|(N_{e},\mathbf{U})|\!|\!|_{\epsilon}^{4})|\!|\!|(N_{e},\mathbf{U})|\!|\!|_{\epsilon}^{2},
\end{split}
\end{equation}
and
\begin{equation}
\begin{split}\label{c}
I_{113}
\leq -\frac{1}{2}\frac{\epsilon^{6}H^{2}}{4}\frac{d}{dt}\int\frac{1}{n_{e}n_{i}}(\partial_{x_{1}}^{4}\partial_{x_{2}}N_{e})^{2}
+C_{1}(1+\epsilon^{2}|\!|\!|(N_{e},\mathbf{U})|\!|\!|_{\epsilon}^{4})|\!|\!|(N_{e},\mathbf{U})|\!|\!|_{\epsilon}^{2}.
\end{split}
\end{equation}
respectively, thanks to Lemma \ref{L2} and \ref{L3}. By \eqref{a}, \eqref{b}, \eqref{c}, we have
\begin{equation}
\begin{split}\label{d}
I_{11}
\leq &-\frac{\epsilon^{4}}{2}\frac{d}{dt}\int\frac{n_{e}}{n_{i}} \left(\partial_{x_{1}}^{4}N_{e}\right)^{2}
-\frac{\epsilon^{5}}{2}\frac{H^{2}}{4}\frac{d}{dt}\int\frac{1}{n_{e}n_{i}} \left(\left(\partial_{x_{1}}^{5}N_{e}\right)^{2}
+\epsilon\left(\partial_{x_{1}}^{4}\partial_{x_{2}}N_{e}\right)^{2}\right)\\
&+C_{1}(1+\epsilon^{2}|\!|\!|(N_{e},\mathbf{U})|\!|\!|_{\epsilon}^{4})|\!|\!|(N_{e},\mathbf{U})|\!|\!|_{\epsilon}^{2}.
\end{split}
\end{equation}
For convenience, we rewrite
\begin{equation*}
\begin{split}
I_{12}=&-\epsilon^{5}\int\left(\frac{n_{e}}{n_{i}}\partial_{x_{1}}^{5}N_{e}-\frac{\epsilon H^{2}}{4}\frac{\partial_{x_{1}}^{7}N_{e}
+\epsilon\partial_{x_{1}}^{5}\partial_{x_{2}}^{2}N_{e}}{n_{e}n_{i}}\right)\partial_{t}\partial_{x_{1}}^{3}(n_{e}\partial_{x_{1}}^{2}N_{e})\\
=&:\sum_{i=1}^{3}I_{12i}.
\end{split}
\end{equation*}
By integration by parts, we have
\begin{equation*}
\begin{split}
I_{121}=&-\frac{\epsilon^{5}}{2}\frac{d}{dt}\int\frac{n_{e}^{2}}{n_{i}}(\partial_{x_{1}}^{5}N_{e})^{2}
-\frac{\epsilon^{5}}{2}\int\partial_{t}\big(\frac{n_{e}^{2}}{n_{i}}\big)(\partial_{x_{1}}^{5}N_{e})^{2}\\
&-\epsilon^{5}\int\frac{n_{e}}{n_{i}}\partial_{x_{1}}^{5}N_{e}
\left([\partial_{x_{1}}^{3},\partial_{t}n_{e}]\partial_{x_{1}}^{2}N_{e} +[\partial_{x_{1}}^{3},n_{e}]\partial_{t}\partial_{x_{1}}^{2}N_{e}\right)\\
=:&\sum_{i=1}^{3}I_{121}^{(i)}.
\end{split}
\end{equation*}
Applying thanks to Lemma \ref{L2} and \ref{L3}, we have
\begin{equation*}
\begin{split}
I_{121}^{(1)}\leq C_{1}(1+\epsilon^{2}|\!|\!|(N_{e},\mathbf{U})|\!|\!|_{\epsilon}^{2}) |\!|\!|(N_{e},\mathbf{U})|\!|\!|_{\epsilon}^{2},
\end{split}
\end{equation*}
and
\begin{equation*}
\begin{split}
I_{121}^{(3)}\leq &C\epsilon^{5}\|\partial_{x_{1}}^{5}N_{e}\|\big(\|[\partial_{x_{1}}^{3},\partial_{t}n_{e}]\partial_{x_{1}}^{2}N_{e}\|
+\|[\partial_{x_{1}}^{3},n_{e}]\partial_{t}\partial_{x_{1}}^{2}N_{e}\big)\|)\\
\leq &C\epsilon^{5}\|\partial_{x_{1}}^{5}N_{e}\|\big(\|\partial_{t}\partial_{x_{1}}n_{e}\|_{L^{\infty}}\|\partial_{x_{1}}^{4}N_{e}\|
+\|\partial_{x_{1}}^{2}N_{e}\|_{L^{\infty}}\|\partial_{t}\partial_{x_{1}}^{3}n_{e}\|\\
&+\|\partial_{x_{1}}n_{e}\|_{L^{\infty}}\|\partial_{t}\partial_{x_{1}}^{4}N_{e}\|
+\|\partial_{t}\partial_{x_{1}}^{2}N_{e}\|_{L^{\infty}}\|\partial_{x_{1}}^{3}n_{e}\|\big)\\
\leq &C_{1}(1+\epsilon^{2}|\!|\!|(N_{e}, \mathbf{U})|\!|\!|_{\epsilon}^{2})|\!|\!|(N_{e}, \mathbf{U})|\!|\!|_{\epsilon}^{2},
\end{split}
\end{equation*}
thanks to the commutator estimates. Thus we have
\begin{equation*}
\begin{split}
I_{121}\leq &-\frac{\epsilon^{5}}{2}\frac{d}{dt}\int\frac{n_{e}^{2}}{n_{i}}(\partial_{x_{1}}^{5}N_{e})^{2}
+C_{1}(1+\epsilon^{2}|\!|\!|(N_{e},\mathbf{U})|\!|\!|_{\epsilon}^{2})|\!|\!|(N_{e},\mathbf{U})|\!|\!|_{\epsilon}^{2}.
\end{split}
\end{equation*}
By integration by parts, we have
\begin{equation*}
\begin{split}
I_{122}=&-\frac{\epsilon^{6}H^{2}}{4}\frac{d}{dt}\int\frac{1}{n_{i}}(\partial_{x_{1}}^{6}N_{e})^{2}\\
&+\frac{1}{2}\frac{\epsilon^{6}H^{2}}{4}\int\partial_{t}(\frac{1}{n_{i}})(\partial_{x_{1}}^{6}N_{e})^{2}
-\frac{\epsilon^{6}H^{2}}{4}\int\partial_{x_{1}}(\frac{1}{n_{i}})\partial_{x_{1}}^{6}N_{e}\partial_{t}\partial_{x_{1}}^{5}N_{e}\\
&+\frac{\epsilon^{6}H^{2}}{4}\int\frac{1}{n_{e}n_{i}}\partial_{x_{1}}^{7}N_{e}
\big([\partial_{x_{1}}^{3},\partial_{t}n_{e}]\partial_{x_{1}}^{2}N_{e}+[\partial_{x_{1}}^{3},n_{e}]\partial_{t}\partial_{x_{1}}^{2}N_{e}\big),
\end{split}
\end{equation*}
and
\begin{equation*}
\begin{split}
I_{123}=&-\frac{\epsilon^{7}H^{2}}{4}\frac{d}{dt}\int\frac{1}{n_{i}}(\partial_{x_{1}}^{5}\partial_{x_{2}}N_{e})^{2}\\
&+\frac{\epsilon^{7}H^{2}}{4}\int\frac{1}{n_{e}n_{i}}\partial_{x_{1}}^{5}\partial_{x_{2}}^{2}N_{e}
\big([\partial_{x_{1}}^{3},\partial_{t}n_{e}]\partial_{x_{1}}^{2}N_{e}+[\partial_{x_{1}}^{3},n_{e}]\partial_{t}\partial_{x_{1}}^{2}N_{e}\big)\\
&+\frac{1}{2}\frac{\epsilon^{7}H^{2}}{4}\int\partial_{t}(\frac{1}{n_{i}})(\partial_{x_{1}}^{5}\partial_{x_{2}}N_{e})^{2}
-\frac{\epsilon^{7}H^{2}}{4}\int\partial_{x_{2}}(\frac{1}{n_{i}}) \partial_{x_{1}}^{5}\partial_{x_{2}}N_{e}\partial_{t} \partial_{x_{1}}^{5}N_{e},
\end{split}
\end{equation*}
yielding
\begin{equation*}
\begin{split}
I_{122}+I_{123}\leq &-\frac{\epsilon^{6}H^{2}}{4}\frac{d}{dt}\int\frac{1}{n_{i}}(\partial_{x_{1}}^{6}N_{e})^{2}
-\frac{\epsilon^{7}H^{2}}{4}\frac{d}{dt}\int\frac{1}{n_{i}}(\partial_{x_{1}}^{5}\partial_{x_{2}}N_{e})^{2}\\
&+C_{1}(1+\epsilon^{2}|\!|\!|(N_{e},\mathbf{U})|\!|\!|_{\epsilon}^{2})|\!|\!|(N_{e},\mathbf{U})|\!|\!|_{\epsilon}^{2}.
\end{split}
\end{equation*}
Thus we have
\begin{equation}
\begin{split}\label{e}
I_{12}\leq &-\frac{\epsilon^{5}}{2}\frac{d}{dt}\int\frac{n_{e}^{2}}{n_{i}}(\partial_{x_{1}}^{5}N_{e})^{2}
-\frac{\epsilon^{6}H^{2}}{4}\frac{d}{dt}\int\frac{1}{n_{i}}\big((\partial_{x_{1}}^{6}N_{e})^{2}
+\epsilon(\partial_{x_{1}}^{5}\partial_{x_{2}}N_{e})^{2}\big)\\
&+C_{1}(1+\epsilon^{2}|\!|\!|(N_{e},\mathbf{U})|\!|\!|_{\epsilon}^{2})|\!|\!|(N_{e},\mathbf{U})|\!|\!|_{\epsilon}^{2}.
\end{split}
\end{equation}
For $I_{12}$, we have
\begin{equation}
\begin{split}\label{f}
I_{13}=&-\epsilon^{6}\int\big(\frac{n_{e}}{n_{i}}\partial_{x_{1}}^{5}N_{e}-\frac{\epsilon H^{2}}{4}\frac{\partial_{x_{1}}^{7}N_{e}
+\epsilon\partial_{x_{1}}^{5}\partial_{x_{2}}^{2}N_{e}}{n_{e}n_{i}}\big)\partial_{t}\partial_{x_{1}}^{3}(n_{e}\partial_{x_{2}}^{2}N_{e})\\
\leq &-\frac{\epsilon^{6}}{2}\frac{d}{dt}\int\frac{n_{e}^{2}}{n_{i}}(\partial_{x_{1}}^{4}\partial_{x_{2}}N_{e})^{2}
-\frac{1}{2}\frac{\epsilon^{7}H^{2}}{4}\frac{d}{dt}\int\frac{1}{n_{i}}\big((\partial_{x_{1}}^{5}\partial_{x_{2}}N_{e})^{2}
+\epsilon(\partial_{x_{1}}^{4}\partial_{x_{2}}^{2}N_{e})^{2}\big)\\
&+C_{1}(1+\epsilon^{2}|\!|\!|(N_{e},\mathbf{U})|\!|\!|_{\epsilon}^{2})|\!|\!|(N_{e},\mathbf{U})|\!|\!|_{\epsilon}^{2}.
\end{split}
\end{equation}
For $I_{14}$, we rewrite
\begin{equation*}
\begin{split}
I_{14}=&\frac{\epsilon^{6}H^{2}}{4}\int\left(\frac{n_{e}}{n_{i}} \partial_{x_{1}}^{5}N_{e}-\frac{\epsilon H^{2}}{4}\frac{\partial_{x_{1}}^{7}N_{e} +\epsilon\partial_{x_{1}}^{5}\partial_{x_{2}}^{2}N_{e}}{n_{e}n_{i}}\right) \partial_{t}\partial_{x_{1}}^{3}\left(\frac{\partial_{x_{1}}^{4}N_{e}}{n_{e}}\right) =:\sum_{i=1}^{3}I_{14i}.
\end{split}
\end{equation*}
By integration by parts, we can rewrite
\begin{equation*}
\begin{split}
I_{141}=&-\frac{1}{2}\frac{\epsilon^{6}H^{2}}{4}\frac{d}{dt}\int\frac{1}{n_{i}}(\partial_{x_{1}}^{6}N_{e})^{2}
+\frac{1}{2}\frac{\epsilon^{6}H^{2}}{4}\int\partial_{t}(\frac{1}{n_{i}})(\partial_{x_{1}}^{6}N_{e})^{2}\\
&-\frac{\epsilon^{6}H^{2}}{4}\int\partial_{x_{1}}(\frac{1}{n_{i}})\partial_{x_{1}}^{5}N_{e}\partial_{t}\partial_{x_{1}}^{6}N_{e}\\
&+\frac{\epsilon^{6}H^{2}}{4}\int\frac{n_{e}}{n_{i}}\partial_{x_{1}}^{5}N_{e}
\left([\partial_{x_{1}}^{3},\partial_{t}\frac{1}{n_{e}}]\partial_{x_{1}}^{4}N_{e}
+[\partial_{x_{1}}^{3},\frac{1}{n_{e}}]\partial_{t} \partial_{x_{1}}^{4}N_{e}\right),
\end{split}
\end{equation*}
which, thanks to Lemma \ref{L2} and \ref{L3}, yields the estimates
\begin{equation*}
\begin{split}
I_{141}\leq-\frac{1}{2}\frac{\epsilon^{6}H^{2}}{4}\frac{d}{dt}\int\frac{1}{n_{i}}(\partial_{x_{1}}^{6}N_{e})^{2}
+C_{1}(1+\epsilon^{2}|\!|\!|(N_{e},\mathbf{U})|\!|\!|_{\epsilon}^{4})|\!|\!|(N_{e},\mathbf{U})|\!|\!|_{\epsilon}^{2}.
\end{split}
\end{equation*}
For $I_{142}$, we have by integration by parts
\begin{equation*}
\begin{split}
I_{142}=&-\frac{1}{2}\frac{\epsilon^{7}H^{4}}{16}\frac{d}{dt}\int\frac{1}{n_{e}^{2}n_{i}}(\partial_{x_{1}}^{7}N_{e})^{2}
+\frac{1}{2}\frac{\epsilon^{7}H^{4}}{16}\int\partial_{t}\big(\frac{1}{n_{e}^{2}n_{i}}\big)(\partial_{x_{1}}^{7}N_{e})^{2}\\
&-\frac{\epsilon^{7}H^{4}}{16}\int\frac{1}{n_{e}n_{i}}\partial_{x_{1}}^{7}N_{e}
\left([\partial_{x_{1}}^{3},\partial_{t}\frac{1}{n_{e}}]\partial_{x_{1}}^{4}N_{e}
+[\partial_{x_{1}}^{3},\frac{1}{n_{e}}] \partial_{t}\partial_{x_{1}}^{4}N_{e}\right),
\end{split}
\end{equation*}
yielding
\begin{equation*}
\begin{split}
I_{142}\leq-\frac{\epsilon^{7}}{2}\frac{H^{4}}{16} \frac{d}{dt}\int\frac{1}{n_{e}^{2}n_{i}}(\partial_{x_{1}}^{7}N_{e})^{2} +C_{1}(1+\epsilon^{2}|\!|\!|(N_{e}, \mathbf{U})|\!|\!|_{\epsilon}^{4})|\!|\!|(N_{e}, \mathbf{U})|\!|\!|_{\epsilon}^{2},
\end{split}
\end{equation*}
again thanks to Lemma \ref{L2} and \ref{L3}. For $I_{143}$, we have
\begin{equation*}
\begin{split}
I_{143}=&-\frac{1}{2}\frac{\epsilon^{8}H^{4}}{16}\frac{d}{dt}\int\frac{1}{n_{i}}(\partial_{x_{1}}^{6}\partial_{x_{2}}N_{e})^{2}
+\frac{1}{2}\frac{\epsilon^{8}H^{4}}{16}\int\partial_{t}(\frac{1}{n_{i}})(\partial_{x_{1}}^{6}\partial_{x_{2}}N_{e})^{2}\\
&-\frac{\epsilon^{8}H^{4}}{16}\int\partial_{x_{2}}(\frac{1}{n_{i}})\partial_{x_{1}}^{6}\partial_{x_{2}}N_{e}\partial_{t}\partial_{x_{1}}^{6}N_{e}
+\frac{\epsilon^{8}H^{4}}{16}\int\partial_{x_{1}}(\frac{1}{n_{i}})\partial_{x_{1}}^{5}\partial_{x_{2}}^{2}N_{e}\partial_{t}\partial_{x_{1}}^{6}N_{e}\\
&-\frac{\epsilon^{8}H^{4}}{16}\int\frac{1}{n_{e}n_{i}}\partial_{x_{1}}^{5}\partial_{x_{2}}^{2}N_{e}
\left([\partial_{x_{1}}^{3},\partial_{t}\frac{1}{n_{e}}]\partial_{x_{1}}^{4}N_{e}
+[\partial_{x_{1}}^{3},\frac{1}{n_{e}}]\partial_{t}\partial_{x_{1}}^{4}N_{e}\right),
\end{split}
\end{equation*}
and hence
\begin{equation*}
\begin{split}
I_{143}\leq-\frac{1}{2}\frac{\epsilon^{8}H^{4}}{16}\frac{d}{dt}\int\frac{1}{n_{i}}(\partial_{x_{1}}^{6}\partial_{x_{2}}N_{e})^{2}
+C_{1}(1+\epsilon^{2}|\!|\!|(N_{e},\mathbf{U})|\!|\!|_{\epsilon}^{4})|\!|\!|(N_{e},\mathbf{U})|\!|\!|_{\epsilon}^{2}.
\end{split}
\end{equation*}
Thus, we have
\begin{equation}
\begin{split}\label{g}
I_{14}=&-\frac{\epsilon^{6}}{2}\frac{H^{2}}{4}\frac{d}{dt}\int\frac{1}{n_{i}}(\partial_{x_{1}}^{6}N_{e})^{2}
-\frac{\epsilon^{7}}{2}\frac{H^{4}}{16}\frac{d}{dt}\int\frac{1}{n_{e}^{2}n_{i}}\big((\partial_{x_{1}}^{7}N_{e})^{2}
+\epsilon(\partial_{x_{1}}^{6}\partial_{x_{2}}N_{e})^{2}\big)\\
&+C_{1}(1+\epsilon^{2}|\!|\!|(N_{e},\mathbf{U})|\!|\!|_{\epsilon}^{2})|\!|\!|(N_{e},\mathbf{U})|\!|\!|_{\epsilon}^{2}.
\end{split}
\end{equation}
For $I_{15}$, we can divide
\begin{equation*}
\begin{split}
I_{15}=&\frac{\epsilon^{7}H^{2}}{2}\int\left(\frac{n_{e}}{n_{i}}\partial_{x_{1}}^{5}N_{e}-\frac{\epsilon H^{2}}{4}\frac{\partial_{x_{1}}^{7}N_{e}
+\epsilon\partial_{x_{1}}^{5}\partial_{x_{2}}^{2}N_{e}}{n_{e}n_{i}}\right)\partial_{t}\partial_{x_{1}}^{3}(\frac{\partial_{x_{1}}^{2}\partial_{x_{2}}^{2}N_{e}}{n_{e}})\\
=&:\sum_{i=1}^{3}I_{15i}.
\end{split}
\end{equation*}
By integration by parts, we have
\begin{equation*}
\begin{split}
I_{151}=&-\frac{1}{2}\frac{\epsilon^{7}H^{2}}{2}\frac{d}{dt}\int\frac{1}{n_{i}}(\partial_{x_{1}}^{5}\partial_{x_{2}}N_{e})^{2}
+\frac{1}{2}\frac{\epsilon^{7}H^{2}}{2}\int\partial_{t}(\frac{1}{n_{i}})(\partial_{x_{1}}^{5}\partial_{x_{2}}N_{e})^{2}\\
&-\frac{\epsilon^{7}H^{2}}{2}\int\partial_{x_{2}}(\frac{1}{n_{i}})\partial_{x_{1}}^{5}N_{e}\partial_{t}\partial_{x_{1}}^{5}\partial_{x_{2}}N_{e}\\
&+\frac{\epsilon^{7}H^{2}}{2}\int\frac{n_{e}}{n_{i}}\partial_{x_{1}}^{5}N_{e}
\left([\partial_{x_{1}}^{3},\partial_{t}\frac{1}{n_{e}}]\partial_{x_{1}}^{2}\partial_{x_{2}}^{2}N_{e}
+[\partial_{x_{1}}^{3},\frac{1}{n_{e}}] \partial_{t}\partial_{x_{1}}^{2}\partial_{x_{2}}^{2}N_{e}\right),
\end{split}
\end{equation*}
which can be bounded similarly to $I_{14}$,
\begin{equation*}
\begin{split}
I_{151}\leq-\frac{1}{2}\frac{\epsilon^{7}H^{2}}{4}\frac{d}{dt}\int\frac{1}{n_{i}}(\partial_{x_{1}}^{5}\partial_{x_{2}}N_{e})^{2}
+C_{1}\big(1+\epsilon^{2}|\!|\!|(N_{e},\mathbf{U})|\!|\!|_{\epsilon}^{4}\big)|\!|\!|(N_{e},\mathbf{U})|\!|\!|_{\epsilon}^{2}.
\end{split}
\end{equation*}
By integration by parts, we have
\begin{equation*}
\begin{split}
I_{152}=&-\frac{1}{2}\frac{\epsilon^{8}H^{4}}{8}\frac{d}{dt}\int\frac{1}{n_{i}}(\partial_{x_{1}}^{6}\partial_{x_{2}}N_{e})^{2}
+\frac{\epsilon^{8}H^{4}}{8}\int\partial_{x_{1}x_{2}}(\frac{1}{n_{i}})\partial_{x_{1}}^{6}N_{e}\partial_{t}\partial_{x_{1}}^{5}\partial_{x_{2}}N_{e}\\
&-\frac{\epsilon^{8}H^{4}}{8}\int\partial_{x_{1}}^{2}(\frac{1}{n_{i}})\partial_{x_{1}}^{6}N_{e}\partial_{t}\partial_{x_{1}}^{4}\partial_{x_{2}}^{2}N_{e}
+\frac{1}{2}\frac{\epsilon^{8}H^{4}}{8}\int\partial_{t}(\frac{1}{n_{i}})(\partial_{x_{1}}^{6}\partial_{x_{2}}N_{e})^{2}\\
&+\frac{\epsilon^{8}H^{4}}{8}\int\partial_{x_{2}}(\frac{1}{n_{i}})\partial_{x_{1}}^{7}N_{e}\partial_{t}\partial_{x_{1}}^{5}\partial_{x_{2}}N_{e}
-\frac{\epsilon^{8}H^{4}}{8}\int\partial_{x_{1}}(\frac{1}{n_{i}})\partial_{x_{1}}^{7}N_{e}\partial_{t}\partial_{x_{1}}^{4}\partial_{x_{2}}^{2}N_{e}\\
&-\frac{\epsilon^{8}H^{4}}{8}\int\frac{1}{n_{e}n_{i}}\partial_{x_{1}}^{7}N_{e}
\left([\partial_{x_{1}}^{3},\partial_{t}\frac{1}{n_{e}}]\partial_{x_{1}}^{2}\partial_{x_{2}}^{2}N_{e}
+[\partial_{x_{1}}^{3},\frac{1}{n_{e}}]\partial_{t}\partial_{x_{1}}^{2}\partial_{x_{2}}^{2}N_{e}\right)\\
=&:\sum_{i=1}^{7}I_{152}^{(i)}.
\end{split}
\end{equation*}
Noting
\begin{equation*}
\begin{split}
&\left|\partial_{x_{1}x_{2}}(\frac{1}{n_{i}})\right|
\leq C\left(\epsilon+\epsilon^{5}(|\partial_{x_{1}}N_{i}|+|\partial_{x_{2}}N_{i}|+|\partial_{x_{1}x_{2}}N_{i}|)
+\epsilon^{10}|\partial_{x_{1}}N_{i}||\partial_{x_{2}}N_{i}|\right),
\end{split}
\end{equation*}
and
\begin{equation*}
\begin{split}
\left|\partial_{x_{1}}^{2}(\frac{1}{n_{i}})\right|
\leq C\left(\epsilon+\epsilon^{5}(|\partial_{x_{1}}N_{i}|+|\partial_{x_{1}}^{2}N_{i}|)
+\epsilon^{10}(|\partial_{x_{1}}N_{i}|)^{2}\right),
\end{split}
\end{equation*}
the term $I_{152}^{(2)}$ and $I_{152}^{(3)}$ can be bounded by
\begin{equation*}
\begin{split}
I_{152}^{(2)}\leq &C\big(1+\epsilon^{7}(\|\partial_{x_{1}}N_{i}\|_{L^{\infty}}^{2}+\epsilon\|\partial_{x_{2}}N_{i}\|_{L^{\infty}}^{2})\big)
(\epsilon^{6}\|\partial_{x_{1}}^{6}N_{e}\|^{2}+\epsilon^{10}\|\partial_{t}\partial_{x_{1}}^{5}\partial_{x_{2}}N_{e}\|^{2})\\
&+C\big(1+\epsilon^{12}\|\epsilon\partial_{t}\partial_{x_{1}}^{5}\partial_{x_{2}}N_{e}\|^{2}\big)
\big(\epsilon^{5}\|\partial_{x_{1}x_{2}}N_{i}\|_{L^{3}}^{2}+\epsilon^{8}\|\partial_{x_{1}}^{6}N_{e}\|_{L^{6}}^{2}\big)\\
\leq &C\big(1+\epsilon^{7}(\|\partial_{x_{1}}N_{i}\|_{H^{2}}^{2}+\epsilon\|\partial_{x_{2}}N_{i}\|_{H^{2}}^{2})\big)
\big(\epsilon^{6}\|\partial_{x_{1}}^{6}N_{e}\|^{2}+\epsilon^{10}\|\partial_{t}\partial_{x_{1}}^{5}\partial_{x_{2}}N_{e}\|^{2}\big)\\
&+C\big(1+\epsilon^{12}\|\epsilon\partial_{t}\partial_{x_{1}}^{5}\partial_{x_{2}}N_{e}\|^{2}\big)
\big(\epsilon^{5}\|\partial_{x_{1}x_{2}}N_{i}\|_{H^{1}}^{2}+\epsilon^{8}\|\partial_{x_{1}}^{6}N_{e}\|_{H^{1}}^{2}\big)\\
\leq &C_{1}(1+\epsilon^{2}|\!|\!|(N_{e},\mathbf{U})|\!|\!|_{\epsilon}^{2})|\!|\!|(N_{e},\mathbf{U})|\!|\!|_{\epsilon}^{2},
\end{split}
\end{equation*}
and
\begin{equation*}
\begin{split}
I_{152}^{(3)} \leq C_{1}(1+\epsilon^{2}|\!|\!|(N_{e}, \mathbf{U})|\!|\!|_{\epsilon}^{2})|\!|\!|(N_{e}, \mathbf{U})|\!|\!|_{\epsilon}^{2},
\end{split}
\end{equation*}
respectively, thanks to Lemma \ref{L2} and \ref{L3} and the Sobolev embedding inequalities. The other terms in $I_{152}$ can be similarly bounded by
\begin{equation*}
\begin{split}
I_{152}^{(4\sim7)}
\leq C_{1}(1+\epsilon^{2}|\!|\!|(N_{e},\mathbf{U})|\!|\!|_{\epsilon}^{4})|\!|\!|(N_{e},\mathbf{U})|\!|\!|_{\epsilon}^{2}.
\end{split}
\end{equation*}
Therefore, we have
\begin{equation*}
\begin{split}
I_{152}\leq-\frac{1}{2}\frac{\epsilon^{8}H^{4}}{8}\frac{d}{dt}\int\frac{1}{n_{i}}(\partial_{x_{1}}^{6}\partial_{x_{2}}N_{e})^{2}
+C_{1}(1+\epsilon^{2}|\!|\!|(N_{e},\mathbf{U})|\!|\!|_{\epsilon}^{4})|\!|\!|(N_{e},\mathbf{U})|\!|\!|_{\epsilon}^{2}.
\end{split}
\end{equation*}
The $I_{153}$ term can be bounded by
\begin{equation*}
\begin{split}
I_{153}=&-\frac{\epsilon^{9}}{2}\frac{H^{4}}{8}\frac{d}{dt}\int\frac{1}{n_{e}^{2}n_{i}}(\partial_{x_{1}}^{5}\partial_{x_{2}}^{2}N_{e})^{2}
+\frac{\epsilon^{9}}{2}\frac{H^{4}}{8}\int\partial_{t}\big(\frac{1}{n_{e}^{2}n_{i}}\big)(\partial_{x_{1}}^{5}\partial_{x_{2}}^{2}N_{e})^{2}\\
&-\frac{\epsilon^{9}H^{4}}{8}\int\frac{1}{n_{e}n_{i}}\partial_{x_{1}}^{5}\partial_{x_{2}}^{2}N_{e}
\big([\partial_{x_{1}}^{3},\partial_{t}\frac{1}{n_{e}}]\partial_{x_{1}}^{2}\partial_{x_{2}}^{2}N_{e}
+[\partial_{x_{1}}^{3},\frac{1}{n_{e}}]\partial_{t}\partial_{x_{1}}^{2}\partial_{x_{2}}^{2}N_{e}\big)\\
\leq& -\frac{\epsilon^{9}}{2}\frac{H^{4}}{16}\frac{d}{dt}\int\frac{1}{n_{e}^{2}n_{i}}(\partial_{x_{1}}^{5}\partial_{x_{2}}^{2}N_{e})^{2}
+C_{1}(1+\epsilon^{2}|\!|\!|(N_{e},\mathbf{U})|\!|\!|_{\epsilon}^{4})|\!|\!|(N_{e},\mathbf{U})|\!|\!|_{\epsilon}^{2},
\end{split}
\end{equation*}
and finally yields the estimates
\begin{equation}
\begin{split}\label{h}
I_{15}\leq&-\frac{\epsilon^{7}}{2}\frac{H^{2}}{2}\frac{d}{dt}\int\frac{1}{n_{i}}(\partial_{x_{1}}^{5}\partial_{x_{2}}N_{e})^{2}
-\frac{\epsilon^{8}}{2}\frac{H^{4}}{8}\frac{d}{dt}\int\frac{1}{n_{i}}\big((\partial_{x_{1}}^{6}\partial_{x_{2}}N_{e})^{2}
+\epsilon(\partial_{x_{1}}^{5}\partial_{x_{2}}^{2}N_{e})^{2}\big)\\
&+C_{1}(1+\epsilon^{2}|\!|\!|(N_{e},\mathbf{U})|\!|\!|_{\epsilon}^{4})|\!|\!|(N_{e},\mathbf{U})|\!|\!|_{\epsilon}^{2}.
\end{split}
\end{equation}
From\eqref{orin}, the term $I_{16}$ can be rewritten as
\begin{equation*}
\begin{split}
I_{16}=&\frac{\epsilon^{8}H^{2}}{4}\int\left(\frac{n_{e}}{n_{i}}\partial_{x_{1}}^{5}N_{e}-\frac{\epsilon H^{2}}{4}\frac{\partial_{x_{1}}^{7}N_{e}
+\epsilon\partial_{x_{1}}^{5}\partial_{x_{2}}^{2}N_{e}}{n_{e}n_{i}}\right) \partial_{t}\partial_{x_{1}}^{3} \left(\frac{\partial_{x_{2}}^{4}N_{e}}{n_{e}}\right) =:\sum_{i=1}^{3}I_{16i}.
\end{split}
\end{equation*}
By integration by parts, the first term is divided into
\begin{equation*}
\begin{split}
I_{161}=&-\frac{\epsilon^{8}}{2}\frac{H^{2}}{4}\frac{d}{dt}\int\frac{1}{n_{i}}(\partial_{x_{1}}^{4}\partial_{x_{2}}^{2}N_{e})^{2}
+\frac{\epsilon^{8}H^{2}}{8}\int\partial_{t}(\frac{1}{n_{i}})(\partial_{x_{1}}^{4}\partial_{x_{2}}^{2}N_{e})^{2}\\
&-\frac{\epsilon^{8}H^{2}}{2}\int\partial_{x_{2}}(\frac{1}{n_{i}})
\partial_{x_{1}}^{4}\partial_{x_{2}}N_{e}\partial_{t}\partial_{x_{1}}^{4}\partial_{x_{2}}^{2}N_{e}
-\frac{\epsilon^{8}H^{2}}{4}\int\partial_{x_{2}}^{2}(\frac{1}{n_{i}})
\partial_{x_{1}}^{4}N_{e}\partial_{t}\partial_{x_{1}}^{4}\partial_{x_{2}}^{2}N_{e}\\
&+\frac{\epsilon^{8}H^{2}}{4}\int\partial_{x_{1}}^{2}(\frac{1}{n_{i}})
\partial_{x_{1}}^{4}N_{e}\partial_{t}\partial_{x_{1}}^{2}\partial_{x_{2}}^{4}N_{e}
+\frac{\epsilon^{8}H^{2}}{4}\int\partial_{x_{1}}(\frac{1}{n_{i}})
\partial_{x_{1}}^{5}N_{e}\partial_{t}\partial_{x_{1}}^{2}\partial_{x_{2}}^{4}N_{e}\\
&+\frac{\epsilon^{8}H^{2}}{4}\int\frac{n_{e}}{n_{i}}\partial_{x_{1}}^{5}N_{e}
\left([\partial_{x_{1}}^{3},\partial_{t}\frac{1}{n_{e}}]\partial_{x_{2}}^{4}N_{e}
+[\partial_{x_{1}}^{3},\frac{1}{n_{e}}] \partial_{t}\partial_{x_{2}}^{4}N_{e}\right),
\end{split}
\end{equation*}
and can be bounded by
\begin{equation*}
\begin{split}
I_{161}\leq-\frac{\epsilon^{8}}{2}\frac{H^{2}}{4}\frac{d}{dt} \int\frac{1}{n_{i}}(\partial_{x_{1}}^{4}\partial_{x_{2}}^{2}N_{e})^{2} +C_{1}(1+\epsilon^{2}|\!|\!|(N_{e}, \mathbf{U})|\!|\!|_{\epsilon}^{4})|\!|\!|(N_{e}, \mathbf{U})|\!|\!|_{\epsilon}^{2},
\end{split}
\end{equation*}
again thanks to Lemma \ref{L2} and \ref{L3}. For $I_{162}$, we have by integration by parts that
\begin{equation*}
\begin{split}
I_{162}=&-\frac{\epsilon^{9}}{2}\frac{H^{4}}{16}\frac{d}{dt}\int\frac{1}{n_{e}^{2}n_{i}}(\partial_{x_{1}}^{5}\partial_{x_{2}}^{2}N_{e})^{2}
+\frac{\epsilon^{9}H^{4}}{16}\int\partial_{t}\big(\frac{1}{n_{e}^{2}n_{i}}\big)(\partial_{x_{1}}^{5}\partial_{x_{2}}^{2}N_{e})^{2}\\
&+\frac{\epsilon^{9}H^{4}}{8}\int\partial_{x_{2}}(\frac{1}{n_{e}^{2}n_{i}})
\partial_{x_{1}}^{5}\partial_{x_{2}}^{2}N_{e}\partial_{t}\partial_{x_{1}}^{5}\partial_{x_{2}}N_{e}
+\frac{3\epsilon^{9}H^{4}}{16}\int\partial_{x_{2}}^{2}\big(\frac{1}{n_{e}^{2}n_{i}}\big)
\partial_{x_{1}}^{5}\partial_{x_{2}}N_{e}\partial_{t}\partial_{x_{1}}^{5}\partial_{x_{2}}N_{e}\\
&+\frac{\epsilon^{9}H^{4}}{16}\int\partial_{x_{2}}^{3}\big(\frac{1}{n_{e}^{2}n_{i}}\big)
\partial_{x_{1}}^{5}N_{e}\partial_{t}\partial_{x_{1}}^{5}\partial_{x_{2}}N_{e}
-\frac{\epsilon^{9}H^{4}}{8}\int\partial_{x_{1}}\big(\frac{1}{n_{e}^{2}n_{i}}\big)
\partial_{x_{1}}^{7}N_{e}\partial_{t}\partial_{x_{1}}^{2}\partial_{x_{2}}^{4}N_{e}\\
&-\frac{3\epsilon^{9}H^{4}}{16}\int\partial_{x_{1}}^{2}\big(\frac{1}{n_{e}^{2}n_{i}}\big)
\partial_{x_{1}}^{6}N_{e}\partial_{t}\partial_{x_{1}}^{2}\partial_{x_{2}}^{4}N_{e}
-\frac{\epsilon^{9}H^{4}}{16}\int\partial_{x_{1}}^{3}\big(\frac{1}{n_{e}^{2}n_{i}}\big)
\partial_{x_{1}}^{5}N_{e}\partial_{t}\partial_{x_{1}}^{2}\partial_{x_{2}}^{4}N_{e}\\
&-\frac{\epsilon^{9}H^{4}}{16}\int\frac{1}{n_{e}n_{i}}\partial_{x_{1}}^{7}N_{e}
\big([\partial_{x_{1}}^{3},\partial_{t}\frac{1}{n_{e}}]\partial_{x_{2}}^{4}N_{e}
+[\partial_{x_{1}}^{3},\frac{1}{n_{e}}]\partial_{t}\partial_{x_{2}}^{4}N_{e}\big).
\end{split}
\end{equation*}
By Lemma \ref{L2} and \ref{L3} and various Sobolev embeddings $H^2\hookrightarrow L^{\infty}$, $H^1\hookrightarrow L^{3}$ and $H^1\hookrightarrow L^{6}$, we have
\begin{equation*}
\begin{split}
I_{162}\leq-\frac{\epsilon^{9}}{2}\frac{H^{4}}{16}\frac{d}{dt}\int\frac{1}{n_{e}^{2}n_{i}}(\partial_{x_{1}}^{5}\partial_{x_{2}}^{2}N_{e})^{2}
+C_{1}(1+\epsilon^{2}|\!|\!|(N_{e},\mathbf{U})|\!|\!|_{\epsilon}^{4})|\!|\!|(N_{e},\mathbf{U})|\!|\!|_{\epsilon}^{2}.
\end{split}
\end{equation*}
For $I_{163}$, we have
\begin{equation*}
\begin{split}
I_{163}=&-\frac{\epsilon^{10}}{2}\frac{H^{4}}{16}\frac{d}{dt}\int\frac{1}{n_{e}^{2}n_{i}}(\partial_{x_{1}}^{4}\partial_{x_{2}}^{3}N_{e})^{2}
+\frac{\epsilon^{10}}{2}\frac{H^{4}}{16}\int\partial_{t}\big(\frac{1}{n_{e}^{2}n_{i}}\big)(\partial_{x_{1}}^{4}\partial_{x_{2}}^{3}N_{e})^{2}\\
&-\frac{\epsilon^{10}H^{4}}{16}\int\partial_{x_{1}}\big(\frac{1}{n_{e}^{2}n_{i}}\big)
\partial_{x_{1}}^{4}\partial_{x_{2}}^{3}N_{e}\partial_{t}\partial_{x_{1}}^{3}\partial_{x_{2}}^{3}N_{e}
+\frac{\epsilon^{10}H^{4}}{16}\int\partial_{x_{2}}\big(\frac{1}{n_{e}^{2}n_{i}}\big)
\partial_{x_{1}}^{5}\partial_{x_{2}}^{2}N_{e}\partial_{t}\partial_{x_{1}}^{3}\partial_{x_{2}}^{3}N_{e}\\
&-\frac{\epsilon^{10}H^{4}}{16}\int\frac{1}{n_{e}n_{i}}\partial_{x_{1}}^{5}\partial_{x_{2}}^{2}N_{e}
\left([\partial_{x_{1}}^{3},\partial_{t}\frac{1}{n_{e}}]\partial_{x_{2}}^{4}N_{e}
+[\partial_{x_{1}}^{3},\frac{1}{n_{e}}]\partial_{t}\partial_{x_{2}}^{4}N_{e}\right),
\end{split}
\end{equation*}
and hence
\begin{equation*}
\begin{split}
I_{163}\leq-\frac{\epsilon^{10}}{2}\frac{H^{4}}{16}\frac{d}{dt}\int\frac{1}{n_{e}^{2}n_{i}}(\partial_{x_{1}}^{4}\partial_{x_{2}}^{3}N_{e})^{2}
+C_{1}(1+\epsilon^{2}|\!|\!|(N_{e},\mathbf{U})|\!|\!|_{\epsilon}^{4})|\!|\!|(N_{e},\mathbf{U})|\!|\!|_{\epsilon}^{2}.
\end{split}
\end{equation*}
These three inequalities yield the estimate for $I_{16}$ that
\begin{equation}
\begin{split}\label{i}
I_{16}\leq&-\frac{\epsilon^{8}}{2}\frac{H^{2}}{4}\frac{d}{dt} \int\frac{1}{n_{i}}\left(\partial_{x_{1}}^{4}\partial_{x_{2}}^{2} N_{e}\right)^{2}\\
& -\frac{\epsilon^{9}}{2}\frac{H^{4}}{16}\frac{d}{dt} \int\frac{1}{n_{e}^{2}n_{i}}\left(\left(\partial_{x_{1}}^{5} \partial_{x_{2}}^{2}N_{e}\right)^{2} +\epsilon(\partial_{x_{1}}^{4}\partial_{x_{2}}^{3}N_{e})^{2}\right)\\
& +C_{1}\left(1+\epsilon^{2}|\!|\!|(N_{e}, \mathbf{U})|\!|\!|_{\epsilon}^{4}\right)|\!|\!|(N_{e}, \mathbf{U})|\!|\!|_{\epsilon}^{2}.
\end{split}
\end{equation}
Finally, using Sobolev inequalities and Lemma \ref{L2} and \ref{L3}, we have for $I_{1}$ that
\begin{equation}
\begin{split}\label{j}
I_{17\sim111}\leq C_{1}(1+\epsilon^{2}|\!|\!|(N_{e},\mathbf{U})|\!|\!|_{\epsilon}^{6})|\!|\!|(N_{e},\mathbf{U})|\!|\!|_{\epsilon}^{2}.
\end{split}
\end{equation}
Summing up all these inequalities from \eqref{a} to \eqref{j}, we have
\begin{equation*}
\begin{split}
I_{1}
\leq &-\frac{\epsilon^{4}}{2}\frac{d}{dt}\int\frac{n_{e}}{n_{i}}(\partial_{x_{1}}^{4}N_{e})^{2}
-\frac{\epsilon^{5}}{2}\frac{d}{dt}\int\big(\frac{n_{e}^{2}}{n_{i}}
+\frac{H^{2}}{4}\frac{1}{n_{e}n_{i}}\big)\big((\partial_{x_{1}}^{5}N_{e})^{2}+(\partial_{x_{1}}^{4}\partial_{x_{2}}N_{e})^{2}\big)\\
&-\frac{\epsilon^{6}}{2}\frac{H^{2}}{4}\frac{d}{dt}\int\frac{1}{n_{i}}\big(2\partial_{x_{1}}^{6}N_{e})^{2}
+3\epsilon(\partial_{x_{1}}^{5}\partial_{x_{2}}N_{e})^{2}+2\epsilon^{2}(\partial_{x_{1}}^{4}\partial_{x_{2}}^{2}N_{e})^{2}\big)\\
&-\frac{\epsilon^{7}}{2}\frac{H^{4}}{16}\frac{d}{dt}\int\frac{1}{n_{e}^{2}n_{i}}\big((\partial_{x_{1}}^{7}N_{e})^{2}
+3\epsilon(\partial_{x_{1}}^{6}\partial_{x_{2}}N_{e})^{2}+3\epsilon^{2}(\partial_{x_{1}}^{5}\partial_{x_{2}}^{2}N_{e})^{2}
+\epsilon^{3}(\partial_{x_{1}}^{4}\partial_{x_{2}}^{3}N_{e})^{2}\big)\\
&+C_{1}(1+\epsilon^{2}|\!|\!|(N_{e}, \mathbf{U})|\!|\!|_{\epsilon}^{6})(1+|\!|\!|(N_{e}, \mathbf{U})|\!|\!|_{\epsilon}^{2}),
\end{split}
\end{equation*}
completing the proof of Lemma \ref{P3}.
\end{proof}

\subsection{The estimates of the other fourth order for $\mathbf{U}$}\label{2.5}
\begin{lemma}\label{P4}
 Let $(N_{i},N_{e},\mathbf{U})$ be a solution to \eqref{rem}. Then
\begin{equation}\label{equ15}
\begin{split}
\frac{\epsilon^{5}}{2}\frac{d}{dt}&\big(\|\partial_{x_{1}}^{3}\partial_{x_{2}}U_{1}\|^{2}+\|\partial_{x_{1}}^{3}\partial_{x_{2}}U_{2}\|^{2}\big)\\
&+\frac{\epsilon^{5}}{2}\frac{d}{dt}\int\frac{n_{e}}{n_{i}}\big(\partial_{x_{1}}^{3}\partial_{x_{2}}N_{e}\big)^{2}
+\frac{\epsilon^{6}}{2}\frac{d}{dt}\int\big(\frac{n_{e}^{2}}{n_{i}}+\frac{^{H^{2}}}{4}\frac{1}{n_{e}n_{i}}\big)
\left((\partial_{x_{1}}^{4}\partial_{x_{2}}N_{e})^{2}+\epsilon(\partial_{x_{1}}^{3}\partial_{x_{2}}^{2}N_{e})^{2}\right)\\
&+\frac{\epsilon^{7}}{2}\frac{H^{2}}{4}\int\frac{1}{n_{i}}\left((2\partial_{x_{1}}^{5}\partial_{x_{2}}N_{e})^{2}
+3\epsilon(\partial_{x_{1}}^{4}\partial_{x_{2}}^{2}N_{e})^{2}
+2\epsilon^{2}(\partial_{x_{1}}^{3}\partial_{x_{2}}^{3}N_{e})^{2}\right)\\
&+\frac{\epsilon^{8}}{2}\frac{H^{4}}{16}\int\frac{1}{n_{e}^{2}n_{i}}\left((\partial_{x_{1}}^{6}\partial_{x_{2}}N_{e})^{2}
+3\epsilon(\partial_{x_{1}}^{5}\partial_{x_{2}}^{2}N_{e})^{2}
+3\epsilon^{2}(\partial_{x_{1}}^{4}\partial_{x_{2}}^{3}N_{e})^{2}+\epsilon^{3}(\partial_{x_{1}}^{3}\partial_{x_{2}}^{4}N_{e})^{2}\right)\\
\leq & C_{1}(1+\epsilon^{2}|\!|\!|(N_{e},\mathbf{U})|\!|\!|_{\epsilon}^{6})(1+|\!|\!|(N_{e},\mathbf{U})|\!|\!|_{\epsilon}^{2}).
\end{split}
\end{equation}
\end{lemma}
\begin{proof}[Proof of Lemma \ref{P4}.] The proof of this lemma is similar to that of Lemma \ref{P3}. We take $\partial_{x_{1}}^{3}\partial_{x_{2}}$ of \eqref{rem-2} and \eqref{rem-3} respectively, then take inner product of $\epsilon^{5}\partial_{x_{1}}^{3}\partial_{x_{2}}U_{1}, \epsilon^{5}\partial_{x_{1}}^{3}\partial_{x_{2}}U_{2}$ and sum the results.
\end{proof}

\begin{lemma}\label{P5}
 Let $(N_{i},N_{e},\mathbf{U})$ be a solution to \eqref{rem}. Then
\begin{equation}\label{equ16}
\begin{split}
\frac{\epsilon^{6}}{2}\frac{d}{dt}&\big(\|\partial_{x_{1}}^{2}\partial_{x_{2}}^{2}U_{1}\|^{2}
+\|\partial_{x_{1}}^{2}\partial_{x_{2}}^{2}U_{2}\|^{2}\big)\\
&+\frac{\epsilon^{6}}{2}\frac{d}{dt}\int\frac{n_{e}}{n_{i}}(\partial_{x_{1}}^{2}\partial_{x_{2}}N_{e})^{2}
+\frac{\epsilon^{7}}{2}\frac{d}{dt}\int\big(\frac{n_{e}^{2}}{n_{i}}+\frac{^{H^{2}}}{4}\frac{1}{n_{e}n_{i}}\big)
\left((\partial_{x_{1}}^{4}\partial_{x_{2}}^{2}N_{e})^{2}+\epsilon(\partial_{x_{1}}^{2}\partial_{x_{2}}^{4}N_{e})^{2}\right)\\
&+\frac{\epsilon^{8}}{2}\frac{H^{2}}{4}\int\frac{1}{n_{i}}\left((2\partial_{x_{1}}^{4}\partial_{x_{2}}^{2}N_{e})^{2}
+3\epsilon(\partial_{x_{1}}^{3}\partial_{x_{2}}^{3}N_{e})^{2}
+2\epsilon^{2}(\partial_{x_{1}}^{2}\partial_{x_{2}}^{4}N_{e})^{2}\right)\\
&+\frac{\epsilon^{9}}{2}\frac{H^{4}}{16}\int\frac{1}{n_{e}^{2}n_{i}}\left((\partial_{x_{1}}^{5}\partial_{x_{2}}^{2}N_{e})^{2}
+3\epsilon(\partial_{x_{1}}^{4}\partial_{x_{2}}^{3}N_{e})^{2}
+3\epsilon^{2}(\partial_{x_{1}}^{3}\partial_{x_{2}}^{4}N_{e})^{2}+\epsilon^{3}(\partial_{x_{1}}^{2}\partial_{x_{2}}^{5}N_{e})^{2}\right)\\
\leq & C_{1}(1+\epsilon^{2}|\!|\!|(N_{e},\mathbf{U})|\!|\!|_{\epsilon}^{6})(1+|\!|\!|(N_{e},\mathbf{U})|\!|\!|_{\epsilon}^{2}).
\end{split}
\end{equation}
\end{lemma}
\begin{proof}[Proof of Lemma \ref{P5}.] The proof of this lemma is similar to Lemma \ref{P3}. We take $\partial_{x_{1}}^{2}\partial_{x_{2}}^{2}$ of \eqref{rem-2} and \eqref{rem-3} respectively, then take inner product of $\epsilon^{6}\partial_{x_{1}}^{2}\partial_{x_{2}}^{2}U_{1}, \epsilon^{6}\partial_{x_{1}}^{2}\partial_{x_{2}}^{2}U_{2}$ and sum the results.
\end{proof}

\begin{lemma}\label{P6}
 Let $(N_{i},N_{e},\mathbf{U})$ be a solution to \eqref{rem}. Then
\begin{equation}\label{equ17}
\begin{split}
\frac{\epsilon^{7}}{2}\frac{d}{dt}&\big(\|\partial_{x_{1}}\partial_{x_{2}}^{3}U_{1}\|^{2}+\|\partial_{x_{1}}\partial_{x_{2}}^{3}U_{2}\|^{2}\big)\\
&+\frac{\epsilon^{7}}{2}\frac{d}{dt}\int\frac{n_{e}}{n_{i}}(\partial_{x_{1}}\partial_{x_{2}}^{3}N_{e})^{2}
+\frac{\epsilon^{8}}{2}\frac{d}{dt}\int\big(\frac{n_{e}^{2}}{n_{i}}+\frac{^{H^{2}}}{4}\frac{1}{n_{e}n_{i}}\big)
\left((\partial_{x_{1}}^{2}\partial_{x_{2}}^{3}N_{e})^{2}+\epsilon(\partial_{x_{1}}\partial_{x_{2}}^{4}N_{e})^{2}\right)\\
&+\frac{\epsilon^{9}}{2}\frac{H^{2}}{4}\int\frac{1}{n_{i}}\left((2\partial_{x_{1}}^{3}\partial_{x_{2}}^{3}N_{e})^{2}
+3\epsilon(\partial_{x_{1}}^{2}\partial_{x_{2}}^{4}N_{e})^{2}
+2\epsilon^{2}(\partial_{x_{1}}\partial_{x_{2}}^{5}N_{e})^{2}\right)\\
&+\frac{\epsilon^{10}}{2}\frac{H^{4}}{16}\int\frac{1}{n_{e}^{2}n_{i}}\left((\partial_{x_{1}}^{4}\partial_{x_{2}}^{3}N_{e})^{2}
+3\epsilon(\partial_{x_{1}}^{3}\partial_{x_{2}}^{4}N_{e})^{2}
+3\epsilon^{2}(\partial_{x_{1}}^{2}\partial_{x_{2}}^{5}N_{e})^{2}+\epsilon^{3}(\partial_{x_{1}}\partial_{x_{2}}^{6}N_{e})^{2}\right)\\
\leq & C_{1}(1+\epsilon^{2}|\!|\!|(N_{e},\mathbf{U})|\!|\!|_{\epsilon}^{6})(1+|\!|\!|(N_{e},\mathbf{U})|\!|\!|_{\epsilon}^{2}).
\end{split}
\end{equation}
\end{lemma}
\begin{proof}[Proof of Lemma \ref{P6}.] The proof of this lemma is similar to Lemma \ref{P3}. We take $\partial_{x_{1}}\partial_{x_{2}}^{3}$ of \eqref{rem-2} and \eqref{rem-3} respectively, then take inner product of $\epsilon^{7}\partial_{x_{1}}\partial_{x_{2}}^{3}U_{1}, \epsilon^{7}\partial_{x_{1}}\partial_{x_{2}}^{3}U_{2}$ and sum the results.
\end{proof}

\begin{lemma}\label{P7}
 Let $(N_{i},N_{e},\mathbf{U})$ be a solution to \eqref{rem}. Then
\begin{equation}\label{equ18}
\begin{split}
\frac{\epsilon^{8}}{2}\frac{d}{dt}&\big(\|\partial_{x_{2}}^{4}U_{1}\|^{2}+\|\partial_{x_{2}}^{4}U_{2}\|^{2}\big)\\
&+\frac{\epsilon^{8}}{2}\frac{d}{dt}\int\frac{n_{e}}{n_{i}}(\partial_{x_{2}}^{4}N_{e})^{2}
+\frac{\epsilon^{9}}{2}\frac{d}{dt}\int\big(\frac{n_{e}^{2}}{n_{i}}+\frac{^{H^{2}}}{4}\frac{1}{n_{e}n_{i}}\big)
\left((\partial_{x_{1}}\partial_{x_{2}}^{4}N_{e})^{2}+\epsilon(\partial_{x_{2}}^{5}N_{e})^{2}\right)\\
&+\frac{\epsilon^{10}}{2}\frac{H^{2}}{4}\int\frac{1}{n_{i}}\left((2\partial_{x_{1}}^{2}\partial_{x_{2}}^{4}N_{e})^{2}
+3\epsilon(\partial_{x_{1}}\partial_{x_{2}}^{5}N_{e})^{2}
+2\epsilon^{2}(\partial_{x_{2}}^{6}N_{e})^{2}\right)\\
&+\frac{\epsilon^{11}}{2}\frac{H^{4}}{16}\int\frac{1}{n_{e}^{2}n_{i}}\left((\partial_{x_{1}}^{3}\partial_{x_{2}}^{4}N_{e})^{2}
+3\epsilon(\partial_{x_{1}}^{2}\partial_{x_{2}}^{5}N_{e})^{2}
+3\epsilon^{2}(\partial_{x_{1}}\partial_{x_{2}}^{6}N_{e})^{2}+\epsilon^{3}(\partial_{x_{2}}^{7}N_{e})^{2}\right)\\
\leq & C_{1}(1+\epsilon^{2}|\!|\!|(N_{e},\mathbf{U})|\!|\!|_{\epsilon}^{6})(1+|\!|\!|(N_{e},\mathbf{U})|\!|\!|_{\epsilon}^{2}).
\end{split}
\end{equation}
\end{lemma}
\begin{proof}[Proof of Lemma \ref{P7}.] The proof of this lemma is similar to Lemma \ref{P3}. We take $\partial_{x_{2}}^{4}$ of \eqref{rem-2} and \eqref{rem-3} respectively, then take inner product of $\epsilon^{8}\partial_{x_{2}}^{4}U_{1}, \epsilon^{8}\partial_{x_{2}}^{4}U_{2}$ and sum the results.
\end{proof}
Summing the results of Lemma $\ref{P3} \sim \ref{P7}$, we complete the proof of the Proposition \ref{P2}.

\section{Proof of Theorem \ref{thm1}}
\begin{proof}[\textbf{Proof of Theorem \ref{thm1} }]
Adding Propositions \ref{P1} with $k=0,1,2,3$  and Proposition \ref{P2} together, we obtain
\begin{equation}\label{Gron}
\begin{split}
\frac{1}2&{\frac{d}{dt}}\sum_{0\leq\alpha+\beta\leq4} \epsilon^{\alpha+2\beta}\left(\|\partial_{x_{1}}^{\alpha} \partial_{x_{2}}^{\beta}U_{1}\|^{2} +\|\partial_{x_{1}}^{\alpha}\partial_{x_{2}}^{\beta}U_{2}\|^{2}\right)\\
&+\frac{1}2{\frac{d}{dt}}\int\frac{n_{e}}{n_{i}}{N_{e}}^{2} +\frac{1}2{\frac{d}{dt}}\int\left(\frac{n_{e}}{n_{i}} +\frac{n_{e}^{2}}{n_{i}}+\frac{H^{2}}{4}\frac{1}{n_{e}n_{i}}\right) \sum_{\alpha+\beta=1}\epsilon^{\alpha+2\beta}(\partial_{x_{1}}^{\alpha} \partial_{x_{2}}^{\beta}N_{e})^{2}\\
&+\frac{1}2{\frac{d}{dt}}\int\left(\frac{n_{e}}{n_{i}} +\frac{n_{e}^{2}}{n_{i}}+\frac{H^{2}}{4}\frac{1}{n_{e}n_{i}} +\frac{H^{2}}{4}\frac{1}{{n_{i}}}\right) \sum_{\alpha+\beta=2}\epsilon^{\alpha+2\beta} \left(\partial_{x_{1}}^{\alpha}\partial_{x_{2}}^{\beta}N_{e}\right)^{2}\\
&+\frac{1}2{\frac{d}{dt}}\int\left(\frac{n_{e}}{n_{i}}+\frac{n_{e}^{2}}{n_{i}} +\frac{H^{2}}{4}\frac{1}{n_{e}n_{i}}+\frac{H^{2}}{4}\frac{1}{{n_{i}}} +\frac{H^{4}}{16}\frac{1}{{n_{e}^{2}n_{i}}}\right) \sum_{3\leq\alpha+\beta\leq4}\epsilon^{\alpha+2\beta} \left(\partial_{x_{1}}^{\alpha}\partial_{x_{2}}^{\beta}N_{e}\right)^{2}\\
&+\frac{1}2{\frac{d}{dt}}\int\left(\frac{n_{e}^{2}}{n_{i}} +\frac{H^{2}}{4}\frac{1}{n_{e}n_{i}} +\frac{H^{2}}{4}\frac{1}{{n_{i}}} +\frac{H^{4}}{16}\frac{1}{{n_{e}^{2}n_{i}}}\right) \sum_{\alpha+\beta=5}\epsilon^{\alpha+2\beta} \left(\partial_{x_{1}}^{\alpha}\partial_{x_{2}}^{\beta}N_{e}\right)^{2}\\
&+\frac{1}2\frac{H^{2}}{4}{\frac{d}{dt}}\int\left(\frac{H^{2}}{4} \frac{1}{{n_{i}}}+\frac{H^{4}}{16}\frac{1}{{n_{e}^{2}n_{i}}}\right) \sum_{\alpha+\beta=6}\epsilon^{\alpha+2\beta}\left(\partial_{x_{1}}^{\alpha} \partial_{x_{2}}^{\beta}N_{e}\right)^{2}\\
&+\frac{1}2\frac{H^{4}}{16}{\frac{d}{dt}}\int\frac{1}{n_{e}^{2}n_{i}}
\sum_{\alpha+\beta=7}\epsilon^{\alpha+2\beta} \left(\partial_{x_{1}}^{\alpha}\partial_{x_{2}}^{\beta}N_{e}\right)^{2}\\
\leq &C\left(1+\epsilon^{2}|\!|\!|(N_{e},\mathbf{U})|\!|\!|_{\epsilon}^{6}\right) \left(1+|\!|\!|(N_{e},\mathbf{U})|\!|\!|_{\epsilon}^{2}\right).
\end{split}
\end{equation}
Integrating the inequality \eqref{Gron} over $(0,t)$ yields
\begin{equation*}
\begin{split}
|\!|\!|(N_{e},\mathbf{U})(t)|\!|\!|_{\epsilon}^{2}
&\leq C|\!|\!|(N_{e},\mathbf{U})(0)|\!|\!|_{\epsilon}^{2}
+\int_{0}^{t}C_{1}(1+\epsilon^{2}|\!|\!|(N_{e},\mathbf{U})|\!|\!|_{\epsilon}^{6})
(1+|\!|\!|(N_{e},\mathbf{U})|\!|\!|_{\epsilon}^{2})ds\\
&\leq C|\!|\!|(N_{e},\mathbf{U})(0)|\!|\!|_{\epsilon}^{2}+\int_{0}^{t}C_{1}
(1+\epsilon\tilde{C})
(1+|\!|\!|(N_{e},\mathbf{U})|\!|\!|_{\epsilon}^{2})ds,
\end{split}
\end{equation*}
where $C$ is an absolute constant.

Recall that $C_1$ depends on $|\!|\!|(N_{e},\mathbf{U})|\!|\!|^2_{\epsilon}$ through $\epsilon|\!|\!|(N_{e},\mathbf{U})|\!|\!|^2_{\epsilon}$ and is nondecreasing. Let $C_1'=C_1(1)$ and $C_2>C\sup_{\epsilon<1}|\!|\!|(u^{\epsilon}_R,\phi^{\epsilon}_R) (0)|\!|\!|^2_{\epsilon}$. For any arbitrarily given $\tau>0$, we choose $\tilde C$ sufficiently large such that $\tilde C>e^{4C_1'\tau}(1+C_2)(1+C_1')$. Then there exists $\epsilon_0>0$ such that ${\epsilon}\tilde C\leq 1$ for all $\epsilon<\epsilon_0$, using Gronwall inequality, we have
\begin{equation}\label{e999}
\begin{split}
\sup_{0\leq t\leq\tau}|\!|\!|(N_{e},\mathbf{U})(t)|\!|\!|^2_{\epsilon}\leq e^{4C_1'\tau}(C_2+1)<\tilde C.
\end{split}
\end{equation}
In particular, we have  the uniform bound for $(N_{e},\mathbf{U})$,
\begin{equation}
\begin{split}\label{final}
\sup_{0\leq t\leq\tau}\left(\sum_{0\leq\alpha+\beta\leq4}\epsilon^{\alpha+2\beta} \|\partial_{x_{1}}^{\alpha}\partial_{x_{2}}^{\beta}(U_{1}, U_{2})\|_{L^{2}}^{2} +\sum_{0\leq\alpha+\beta\leq7}\epsilon^{\alpha+2\beta} \|\partial_{x_{1}}^{\alpha}\partial_{x_{2}}^{\beta}N_{e}\|_{L^{2}}^{2}\right) \leq \tilde C.
\end{split}
\end{equation}
On the other hand, by Lemma \ref{L1} and \eqref{final}, we have
\begin{equation*}
\begin{split}
\sup_{0\leq t\leq\tau}\sum_{0\leq\alpha+\beta\leq3}\epsilon^{\alpha+2\beta}\|\partial_{x_{1}}^{\alpha}\partial_{x_{2}}^{\beta}N_{i}\|_{L^{2}}^{2}\leq \tilde C.
\end{split}
\end{equation*}
It is now standard to obtain uniform estimates independent of $\epsilon$ by the continuity method.

\end{proof}

\appendix
\section{}
The concrete expression of $(A_{i}, B_{i})(1\leq i\leq2)$ and $C_{j}(1\leq j\leq3)$ are given by
\begin{align*}
A_{1}=&\big(2\epsilon\partial_{x_{1}}\tilde{n_{e}} +2\epsilon^{5}\partial_{x_{1}}N_{e}\big)\partial_{x_{1}}^{2}N_{e}
+\big(2\epsilon\partial_{x_{1}}^{2}\tilde{n_{e}} +\epsilon^{2}\partial_{x_{2}}^{2}\tilde{n_{e}}\big)\partial_{x_{1}}N_{e}\\
&+\big(\epsilon^{2}\partial_{x_{2}}\tilde{n_{e}} +\epsilon^{5}\partial_{x_{2}}N_{e}\big)\partial_{x_{1}x_{2}}N_{e}
+\big(\epsilon^{2}\partial_{x_{1}}\tilde{n_{e}} +\epsilon^{6}\partial_{x_{1}}N_{e}\big)\partial_{x_{2}}^{2}N_{e}
+\epsilon^{2}\partial_{x_{1}x_{2}}\tilde{n_{e}}\partial_{x_{2}}N_{e},
\end{align*}
\begin{align*}
A_{2}=&\epsilon^{9}(\partial_{x_{1}}N_{e})^{3} +\big(\epsilon^{7}\partial_{x_{1}}\tilde{n_{e}}
+\epsilon^{10}\partial_{x_{1}}N_{e}\big)(\partial_{x_{2}}N_{e})^{2} +3\epsilon^{6}\partial_{x_{1}}\tilde{n_{e}}(\partial_{x_{1}}N_{e})^{2}\\
&+2\epsilon^{7}\partial_{x_{2}}\tilde{n_{e}} \partial_{x_{1}}N_{e}\partial_{x_{2}}N_{e}
+(3\epsilon^{4}(\partial_{x_{1}}\tilde{n_{e}})^{2} +\epsilon^{3}(\partial_{x_{2}}\tilde{n_{e}})^{2})\partial_{x_{1}}N_{e}
+2\epsilon^{3}\partial_{x_{1}}\tilde{n_{e}} \partial_{x_{2}}\tilde{n_{e}}\partial_{x_{2}}N_{e},
\end{align*}
\begin{align*}
B_{1}=&\big(\epsilon\partial_{x_{2}}\tilde{n_{e}} +\epsilon^{5}\partial_{x_{2}}N_{e}\big)\partial_{x_{1}}^{2}N_{e}
+\big(\epsilon\partial_{x_{1}}\tilde{n_{e}} +\epsilon^{5}\partial_{x_{1}}N_{e}\big)\partial_{x_{1}x_{2}}N_{e}\\
&+\big(2\epsilon^{3}\partial_{x_{2}}\tilde{n_{e}} +\epsilon^{6}\partial_{x_{2}}N_{e}\big)\partial_{x_{2}}^{2}N_{e}
+\big(\epsilon\partial_{x_{1}}^{2}\tilde{n_{e}} +2\epsilon^{2}\partial_{x_{2}}^{2}\tilde{n_{e}}\big)\partial_{x_{2}}N_{e}
+\epsilon\partial_{x_{1}x_{2}}\tilde{n_{e}}\partial_{x_{1}}N_{e},
\end{align*}
\begin{align*}
B_{2}=&\epsilon^{10}(\partial_{x_{2}}N_{e})^{3} +\big(\epsilon^{6}\partial_{x_{2}}\tilde{n_{e}}
+\epsilon^{10}\partial_{x_{2}}N_{e}\big)(\partial_{x_{1}}N_{e})^{2} +3\epsilon^{7}\partial_{x_{2}}\tilde{n_{e}}(\partial_{x_{2}}N_{e})^{2}\\
&+2\epsilon^{6}\partial_{x_{1}}\tilde{n_{e}} \partial_{x_{1}}N_{e}\partial_{x_{2}}N_{e}
+(\epsilon^{2}(\partial_{x_{1}}\tilde{n_{e}})^{2} +3\epsilon^{3}(\partial_{x_{2}}\tilde{n_{e}})^{2})\partial_{x_{2}}N_{e}
+2\epsilon^{2}\partial_{x_{1}}\tilde{n_{e}} \partial_{x_{2}}\tilde{n_{e}}\partial_{x_{1}}N_{e},
\end{align*}
\begin{align*}
C_{1}=&\big(3\epsilon\partial_{x_{1}}\tilde{n_{e}} +3\epsilon^{5}\partial_{x_{1}}N_{e}\big)
\big(\partial_{x_{1}}^{3}N_{e} +\epsilon\partial_{x_{1}}\partial_{x_{2}}^{2}N_{e}\big) \\
&+\big(3\epsilon^{2}\partial_{x_{2}}\tilde{n_{e}} +3\epsilon^{7}\partial_{x_{2}}N_{e}\big)
\big(\partial_{x_{1}}^{2}\partial_{x_{2}}N_{e} +\epsilon\partial_{x_{2}}^{3}N_{e}\big) \\
&+\big(4\epsilon\partial_{x_{1}}^{2}\tilde{n_{e}} +2\epsilon^{2}\partial_{x_{2}}^{2}\tilde{n_{e}}
+2\epsilon^{5}\partial_{x_{1}}^{2}N_{e}\big)\partial_{x_{1}}^{2}N_{e}+\big(4\epsilon^{2}\partial_{x_{1}x_{2}}\tilde{n_{e}} +2\epsilon^{6}\partial_{x_{1}x_{2}}N_{e}\big)\partial_{x_{1}x_{2}}N_{e}\\
&+\big(2\epsilon^{2}\partial_{x_{1}}^{2}\tilde{n_{e}} +4\epsilon^{3}\partial_{x_{2}}^{2}\tilde{n_{e}} +2\epsilon^{6}\partial_{x_{1}}^{2}N_{e} +2\epsilon^{7}\partial_{x_{2}}^{2}N_{e}\big)\partial_{x_{2}}^{2}N_{e}\\
&+\big(3\epsilon\partial_{x_{1}}^{3}\tilde{n_{e}} +3\epsilon^{2}\partial_{x_{1}}\partial_{x_{2}}^{2} \tilde{n_{e}}\big)\partial_{x_{1}}N_{e} +\big(3\epsilon^{2}\partial_{x_{1}}^{2}\partial_{x_{2}}\tilde{n_{e}} +3\epsilon^{3}\partial_{x_{2}}^{3}\tilde{n_{e}}\big)\partial_{x_{2}}N_{e},
\end{align*}
\begin{align*}
C_{2}=&\big(7\epsilon^{10}(\partial_{x_{1}}N_{e})^{2}
+3\epsilon^{11}(\partial_{x_{2}}N_{e})^{2}
+14\epsilon^{6}\partial_{x_{1}}\tilde{n_{e}}\partial_{x_{1}}N_{e}
+6\epsilon^{7}\partial_{x_{2}}\tilde{n_{e}} \partial_{x_{2}}N_{e}\big)\partial_{x_{1}}^{2}N_{e}\\
&+\big(3\epsilon^{3}(\partial_{x_{2}}\tilde{n_{e}})^{2}
+7\epsilon^{2}(\partial_{x_{1}}\tilde{n_{e}})^{2}\big) \partial_{x_{1}}^{2}N_{e}
+\big(3\epsilon^{3}(\partial_{x_{1}}\tilde{n_{e}})^{2}
+7\epsilon^{4}(\partial_{x_{2}}\tilde{n_{e}})^{2}\big) \partial_{x_{2}}^{2}N_{e}\\
&+\big(3\epsilon^{11}(\partial_{x_{1}}N_{e})^{2}
+7\epsilon^{12}(\partial_{x_{2}}N_{e})^{2}
+6\epsilon^{7}\partial_{x_{1}}\tilde{n_{e}}\partial_{x_{1}}N_{e}
+14\epsilon^{8}\partial_{x_{2}}\tilde{n_{e}} \partial_{x_{2}}N_{e}\big)\partial_{x_{2}}^{2}N_{e}\\
&+\big(8\epsilon^{11}\partial_{x_{1}}N_{e}\partial_{x_{2}}N_{e} +8\epsilon^{7}\partial_{x_{1}}\tilde{n_{e}}\partial_{x_{2}}N_{e}
+8\epsilon^{7}\partial_{x_{2}}\tilde{n_{e}}\partial_{x_{1}}N_{e}
+8\epsilon^{3}\partial_{x_{1}}\tilde{n_{e}}\partial_{x_{2}} \tilde{n_{e}}\big)\partial_{x_{1}x_{2}}N_{e}\\
&+\big(14\epsilon^{2}\partial_{x_{1}}^{2}\tilde{n_{e}} \partial_{x_{1}}\tilde{n_{e}}
+8\epsilon^{3}\partial_{x_{2}}\tilde{n_{e}} \partial_{x_{1}x_{2}}\tilde{n_{e}}
+6\epsilon^{3}\partial_{x_{2}}^{2}\tilde{n_{e}} \partial_{x_{1}}\tilde{n_{e}}\big)\partial_{x_{1}}N_{e}\\
&+\big(3\epsilon^{7}\partial_{x_{2}}^{2}\tilde{n_{e}}\partial_{x_{1}}N_{e}
+8\epsilon^{7}\partial_{x_{1}x_{2}}\tilde{n_{e}}\partial_{x_{2}}N_{e}
+7\epsilon^{6}\partial_{x_{1}}^{2}\tilde{n_{e}} \partial_{x_{1}}N_{e}\big)\partial_{x_{1}}N_{e}\\
&+\big(6\epsilon^{3}\partial_{x_{1}}^{2}\tilde{n_{e}} \partial_{x_{2}}\tilde{n_{e}}
+8\epsilon^{3}\partial_{x_{1}}\tilde{n_{e}} \partial_{x_{1}x_{2}}\tilde{n_{e}}
+14\epsilon^{4}\partial_{x_{2}}^{2}\tilde{n_{e}} \partial_{x_{2}}\tilde{n_{e}}\big)\partial_{x_{2}}N_{e}\\
&+\big(3\epsilon^{7}\partial_{x_{1}}^{2}\tilde{n_{e}}\partial_{x_{2}}N_{e}
+7\epsilon^{8}\partial_{x_{2}}^{2}\tilde{n_{e}} \partial_{x_{2}}N_{e}\big)\partial_{x_{2}}N_{e},
\end{align*}
\begin{align*}
C_{3}=&3\epsilon^{16}(\partial_{x_{1}}N_{e})^{4}
+6\epsilon^{17}(\partial_{x_{1}}N_{e})^{2}(\partial_{x_{2}}N_{e})^{2}
+3\epsilon^{18}(\partial_{x_{2}}N_{e})^{4}
+12\epsilon^{11}\partial_{x_{1}}\tilde{n_{e}}(\partial_{x_{1}}N_{e})^{3}\\ &+12\epsilon^{12}\partial_{x_{1}}\tilde{n_{e}} \partial_{x_{1}}N_{e}(\partial_{x_{2}}N_{e})^{2}
+12\epsilon^{12}\partial_{x_{2}}\tilde{n_{e}} (\partial_{x_{1}}N_{e})^{2}\partial_{x_{2}}N_{e}
+12\epsilon^{13}\partial_{x_{2}}\tilde{n_{e}} (\partial_{x_{2}}N_{e})^{3}\\
&+18\epsilon^{7}(\partial_{x_{1}}\tilde{n_{e}})^{2} (\partial_{x_{1}}N_{e})^{2}
+6\epsilon^{8}(\partial_{x_{1}}\tilde{n_{e}})^{2} (\partial_{x_{2}}N_{e})^{2}
+24\epsilon^{8}\partial_{x_{1}}\tilde{n_{e}} \partial_{x_{2}}\tilde{n_{e}}\partial_{x_{1}}N_{e}\partial_{x_{2}}N_{e}\\ &+6\epsilon^{8}(\partial_{x_{2}}\tilde{n_{e}})^{2}(\partial_{x_{1}}N_{e})^{2}
+18\epsilon^{9}(\partial_{x_{2}}\tilde{n_{e}})^{4}(\partial_{x_{2}}N_{e})^{2}
+12\epsilon^{3}(\partial_{x_{2}}\tilde{n_{e}})^{3}\partial_{x_{1}}N_{e}\\ &+12\epsilon^{4}(\partial_{x_{1}}\tilde{n_{e}})^{2}\partial_{x_{2}}\tilde{n_{e}}\partial_{x_{2}}N_{e}
+12\epsilon^{4}\partial_{x_{1}}\tilde{n_{e}}(\partial_{x_{2}}\tilde{n_{e}})^{2}\partial_{x_{1}}N_{e}
+12\epsilon^{5}(\partial_{x_{2}}\tilde{n_{e}})^{3}\partial_{x_{2}}N_{e}.
\end{align*}

For reader's convenience, we give the following
\begin{lemma}\label{Lem1}
For $\alpha=0,1,\cdots$ integers and $\gamma=max\{2,\alpha-1\}$, there exist constants $C_{1}=C_{1}(\|n_{e}^{(i)}\|_{H^{\alpha}})$ and $C_{2}=C_{2}(\epsilon\|N_{e}\|_{H^{\gamma}})$ such that
\begin{equation}\label{equ3}
\begin{split}
\|\mathcal R_1,\mathcal R_2^{1,2},\mathcal R_3^{1,2},\mathcal R_4^{1,2}\|_{H^{\alpha}}\leq C_{1}(\|n_{e}^{(i)}\|_{H^{\alpha}}) ,\ \ \ \alpha=0,1,\cdots,
\end{split}
\end{equation}
\begin{equation}\label{equ4}
\begin{split}
\|\mathcal R_2^{3},\mathcal R_3^{3},\mathcal R_4^{3}\|_{H^{\alpha}}\leq C_{2}(\epsilon\|N_{e}\|_{H^{\gamma}})(1+\|N_{e}\|_{H^{\alpha}}), \ \ \ \alpha=0,1,\cdots,
\end{split}
\end{equation}
and
\begin{equation}\label{equ5}
\begin{split}
\|\partial_t\mathcal{R}_2^{3},\partial_t \mathcal{R}_3^{3},\partial_t\mathcal R_4^{3}\|_{H^{\alpha}}\leq C_{2}(\epsilon\|N_{e}\|_{H^{\gamma}})(1+\|\partial_tN_{e}\|_{H^{\alpha}}), \ \ \ \alpha=0,1,\cdots,
\end{split}
\end{equation}
\end{lemma}
\begin{proof}
By H\"older inequality and Sobolev embedding, the estimates for Lemma \ref{Lem1} are straightforward. The details 
are hence omitted.
\end{proof}

\end{document}